\documentclass[12pt,twoside,reqno]{amsart}
\pdfoutput=1
\usepackage{amsmath, amssymb, amsthm, enumitem, esint}
\usepackage{xcolor, comment}
\usepackage[backend=biber,style=numeric]{biblatex}
\DeclareNameAlias{default}{family-given}
\usepackage[hyperindex]{hyperref}
\hypersetup{linktocpage}
\usepackage{tikz}
\usepackage{tkz-euclide}
\usepackage{pgfplots}
\usetikzlibrary{decorations.pathreplacing,shapes.misc,calc, patterns,through, pgfplots.fillbetween, positioning}
\usepackage{blindtext}
\makeatletter
\def\@tocline#1#2#3#4#5#6#7{\relax
  \ifnum #1>\c@tocdepth 
  \else
    \par \addpenalty\@secpenalty\addvspace{#2}%
    \begingroup \hyphenpenalty\@M
    \@ifempty{#4}{%
      \@tempdima\csname r@tocindent\number#1\endcsname\relax
    }{%
      \@tempdima#4\relax
    }%
    \parindent\z@ \leftskip#3\relax \advance\leftskip\@tempdima\relax
    \rightskip\@pnumwidth plus4em \parfillskip-\@pnumwidth
    #5\leavevmode\hskip-\@tempdima
      \ifcase #1
       \or\or \hskip 1em \or \hskip 2em \else \hskip 3em \fi%
      #6\nobreak\relax
    \hfill\hbox to\@pnumwidth{\@tocpagenum{#7}}\par
    \nobreak
    \endgroup
  \fi}
\makeatother

\usepackage{enumitem}
\usepgfplotslibrary{fillbetween}
\pgfplotsset{compat=1.11}
\usepackage[normalem]{ulem}

\usepackage{fancyhdr}
\hypersetup{bookmarksdepth=3}

\newcommand{\EMPTY}[1]{}

\theoremstyle{plain}
\newtheorem{Theorem}{Theorem}[section]
\newtheorem{theorem}{Theorem}[section]
\newtheorem{Lemma}[Theorem]{Lemma}
\newtheorem{lem}[Theorem]{Lemma}
\newtheorem{prop}[Theorem]{Proposition}
\newtheorem{Proposition}[Theorem]{Proposition}

\newtheorem{Corollary}[Theorem]{Corollary}
\theoremstyle{definition}
\newtheorem{Definition}[Theorem]{Definition}
\newtheorem{defn}[Theorem]{Definition}
\newtheorem{Remark}[Theorem]{Remark}
\newtheorem{rmk}[Theorem]{Remark}
\newtheorem*{Conjecture}{Conjecture}
\newtheorem*{Notation}{Notation}

\numberwithin{equation}{section}

\newcommand{\vphi}{\varphi}

\newcommand{\pl}{\partial}

\newcommand{\na}{\nabla}
\newcommand{\lt}{\left}
\newcommand{\rt}{\right}
\newcommand{\rw}{\rightarrow}
\newcommand{\R}{\mathbb{R}}

\renewcommand{\tilde}{\widetilde}

\begin{document}

\pagestyle{fancy}
\fancyhf{} 
\renewcommand{\headrulewidth}{0pt}
\fancyhf[EHC]{B. Harvie and Y.-K. Wang}
\fancyhf[OHC]{Geometry and Uniqueness of Static Black Holes}
\fancyhf[FC]{\thepage}

\title{On the Geometry and Uniqueness of Asymptotically Locally Hyperbolic Static Vacuum Black Holes}
\author{Brian Harvie and Ye-Kai Wang}

\begin{abstract}
We establish several geometric characterizations and rigidity results for $3$-dimensional asymptotically locally hyperbolic (ALH) static spaces with horizon boundary. Notably, a $3$-dimensional ALH static space with toroidal infinity and strictly non-spherical horizons is isometric to a toroidal Kottler metric. Furthermore, we show that the surface gravity of a static horizon with spherical infinity is bounded below by $\sqrt{3}$, with equality achieved only by the critical AdS-Schwarzschild metric. Consequently, Poincar\'e-Einstein fillings of $S^{2} \times S^{1}(\lambda)$ arising from these spaces have length parameter $\lambda \leq (\sqrt{3})^{-1}$, which supports a recent conjecture of Chang-Yang-Zhang \cite{CYZ25}. Finally, static horizons with hyperbolic infinity and non-negative Chru\'sciel-Herzlich mass obey the reverse Riemannian Penrose inequality. In conjunction with work of Ge-Wang-Wu-Xia \cite{GWWX13}, we use this fact to obtain uniqueness of static ALH graphs with hyperbolic infinities.

These results follow from a generalization of the Minkowski inequality in AdS-Schwarzschild space due to Brendle-Hung-Wang \cite{BHW12}. Using optimal coefficients for the sub-static Heintze-Karcher inequality from \cite{FP22}, we construct a new monotone quantity under inverse mean curvature flow (IMCF) in static spaces with negative cosmological constant. Another fundamental tool developed in this paper is a regularity theorem for IMCF in ALH manifolds. Specifically, we prove that a weak solution of IMCF in an ALH $3$-manifold with horizon boundary is eventually smooth. This extends the regularity theorem for a spherical infinity due to Shi-Zhu \cite{SZ21}.

\end{abstract}
\maketitle

\tableofcontents

\section{Introduction}
The classification of stationary black hole solutions to the Einstein field equations, broadly referred to as the ``no-hair conjecture", remains a challenging and fundamental line of research in general relativity. This problem is more tractable for black hole spacetimes which are static, but even under this stronger assumption the existing results are mostly limited to black holes with zero or positive cosmological constant. The main aim of this paper is to establish geometric properties and uniqueness theorems for the time slices of isolated static vacuum black holes with a negative cosmological constant.

In this setting, the model solutions are the \textit{Kottler metrics}. Given a closed oriented surface $(\widehat{\Sigma}^{2},\widehat{g})$ with constant Gauss curvature $\widehat{k} \in \{ -1, 0, 1 \}$, a (nondegenerate) Kottler metric with mass $m$ is a warped product Riemannian manifold $(M^{3},g_{m,\widehat{k}})$ over $(\widehat{\Sigma}^{2},\widehat{g})$ equipped with a potential function $V_{m,\widehat{k}}$. Explicitly,
\begin{eqnarray} \label{ads_schwarzschild}
    M^{3}&=& \widehat{\Sigma}^{2} \times (\rho_{m},\infty), \hspace{0.5cm} g_{m,\widehat{k}}= \left(\widehat{k} + \rho^{2} - \frac{2m}{\rho}\right)^{-1} d\rho^{2} + \rho^{2} \widehat{g}, \hspace{0.5cm} V_{m,\widehat{k}}(\rho)= \sqrt{\widehat{k} + \rho^{2} - \frac{2m}{\rho}}, \nonumber \\
    & & \text{where} \hspace{0.4cm} m \in (m_{\text{crit}},+\infty) \hspace{1cm} \text{for} \hspace{1cm} m_{\text{crit}} = \begin{cases} 0 & \widehat{k}=0,1 \\
    -\frac{1}{3\sqrt{3}} & \widehat{k}=-1\end{cases}.
\end{eqnarray}
Here, $\rho_{m}$ is the largest positive zero of $V_{m,\widehat{k}}$, which always exists for mass $m > m_{\text{crit}}$. The metrics \eqref{ads_schwarzschild} with $\widehat{k}=+1$ are often referred to as \textit{anti-de Sitter Schwarzschild metrics} (AdS-Schwarzschild in short) in the literature. The system \eqref{ads_schwarzschild} is a $\{ \tau= \text{const} \}$ slice of Kottler spacetime $(L^{4},\widetilde{g}) = (M^{3} \times \mathbb{R}, - V_{m,\kappa}^{2}(\rho) d\tau^{2} + g_{m})$, which is a solution to the Einstein equations $\widetilde{\text{Ric}}=-3\widetilde{g}$ with cosmological constant $\Lambda=-3$ and contains a horizon at $\{ \rho=\rho_{m}\}$. In general, a time slice of an isolated static vacuum black hole, possibly with multiple horizons, is modeled by an \textit{asymptotically locally hyperbolic static system}.

\begin{defn} \label{alh_static}
An \textbf{asymptotically locally hyperbolic (ALH) static system} $(M^{3},g,V)$ consists of a connected, non-compact Riemannian manifold $(M^{3},g)$ equipped with a positive potential function $V \in C^{\infty}(M^{3})$ which satisfy the following:

\begin{enumerate}
    \item The metric and potential solve the static vacuum Einstein equations
    \begin{eqnarray} \label{static_equations}
    D^2 V &=& V \text{Ric} + \left( 3V \right)g, \\
    \Delta_{g} V&=& 3 V, \nonumber
\end{eqnarray}
on $M^{3}$. Here, $\text{Ric}$, $D^2 V$, $\Delta_{g} V$ are the Ricci tensor, Hessian tensor of $V$, and Laplacian of $V$, respectively.
\item There exists a closed oriented surface $(\widehat{\Sigma}^{2},\widehat{g})$ with constant Gauss curvature $\widehat{k} \in \{ 1, 0, -1 \}$, a compact set $K \subset \subset M^{3}$, and a diffeomorphism $x: \widehat{\Sigma}^{2} \times (\rho_{0},\infty) \rightarrow M^{3} \setminus K$ such that the pull-backs of the metric and potential under $x$ have the asymptotic expansions
\begin{eqnarray}\label{ALH_intro}
g&=& \left( \widehat{k} + \rho^{2} \right)^{-1} d\rho^{2} + \rho^{2} \widehat{g} + q, \notag \\
& & \text{where} \hspace{0.4cm} |q|_{\bar g} +  |\bar\na q|_{\bar g} +  |\bar\na^2 q|_{\bar g} + |\bar\na^3 q|_{\bar g} = O(\rho^{-\alpha}) \hspace{0.2cm} \text{for some} \hspace{0.2cm} \alpha > 2, \label{ah1}\hspace{0.4cm} \text{and} \\
V&=& \sqrt{\widehat{k} + \rho^{2}} + o_{1}(\rho^{-1}). \label{ah2}
\end{eqnarray}
Here, $\overline{g}$ is the constant curvature metric $\left( \widehat{k} + \rho^{2} \right)^{-1} d\rho^{2} + \rho^{2} \widehat{g}$ on $\widehat{\Sigma}^{2} \times (\rho_{0},\infty)$.

\item $M^{3}$ has compact $C^{\infty}$ boundary $\partial M= \bigsqcup_{j=1}^{J} \partial_{j} M$ consisting of connected components $\partial_{j} M$, $V$ extends smoothly to $0$ on $\partial M$, and $\partial M= \{ V = 0\}$ is a regular level set of $V$. 

\end{enumerate}

\end{defn}
We highlight that item (3) precludes degenerate horizons, which are not addressed in this paper. Our results center around a geometric inequality for outer-minimizing surfaces $\Sigma^{2} \subset (M^{3},g,V)$ relating the total mean curvature of $\Sigma$ with horizon data and with the infinity $(\widehat{\Sigma},\widehat{g})$ of $(M^{3},g,V)$. 

\begin{Theorem}[Minkowski Inequality for ALH Static Systems] \label{minkowski}
Let $(M^{3},g,V)$ be an ALH static system, and assume also that the Euler characteristics satisfy \begin{align}\label{topological assumption for genus control in IMCF}
    \min_{j \in \{ 1, \dots, J \}} \chi(\partial_{j} M) \leq \chi(\widehat{\Sigma}). \tag{T1}
\end{align} Let $\Omega \subset M^{3}$ be a domain, i.e. a connected and bounded open set with $C^{\infty}$ boundary $\partial \Omega$, such that 
\begin{equation*}
\partial \Omega = \partial_{1} M \bigsqcup \dots \bigsqcup \partial_{K} M \bigsqcup \Sigma, \hspace{1cm} 1 \leq K \leq J,
\end{equation*}
 for some connected, outer-minimizing surface $\Sigma^{2} \subset (M^{3},g)$. Suppose also that \begin{align}\label{topological assumption for monotonicity}
    \max_{j \in \{ K+1,\dots, J \}} \chi(\partial_{j}M) \leq 0. \tag{T2}
 \end{align} Then we have the inequality 
 
    \begin{equation} \label{minkowski2}
    \frac{1}{2w_{2}} \int_{\Sigma} VH d\sigma - \frac{3}{w_{2}} \int_{\Omega} V d\mbox{vol} + 2 \sum_{j=1}^{K}  \frac{ 2\pi \chi(\partial_{j} M)}{3|\partial_{j} M| + 2 \pi \chi(\partial_{j} M)} \kappa_{j} \left( \frac{|\partial_{j} M|}{w_{2}} \right) \geq \widehat{k} \left( \frac{|\Sigma|}{w_{2}} \right)^{\frac{1}{2}}
\end{equation}
on $\Sigma$. Here, $H$ is the outward mean curvature of $\Sigma^{2} \subset (M^{3},g)$, $w_{2}$ is the area of $(\widehat{\Sigma}^{2},\widehat{g})$, and $\kappa_{j}= \frac{\partial V}{\partial \nu}|_{\partial_{j} M}$ is the surface gravity of the horizon $\partial_{j} M$. Moreover, equality holds in \eqref{minkowski2} if and only if $J=1$, $(M^{3},g,V)$ is isometric to Kottler space \eqref{ads_schwarzschild}, and $\Sigma= \widehat{\Sigma} \times \{ \rho \}$ is a slice.

\end{Theorem}
\begin{Remark} \label{HK_constants}
The denominator in the left-hand side of \eqref{minkowski2} is necessarily well-defined and positive for any ALH static system, see Section 3. Furthermore, the surface gravity $\kappa_{j}$ in \eqref{minkowski2} is necessarily constant over the horizon $\partial_{j} M$ as a consequence of the static equations, see Section 2.
\end{Remark}
\begin{Remark}
\cite{GSWW99} gives an inequality that restricts horizon topologies in terms of the topology of $\widehat{\Sigma}$. The conditions (T1)-(T2) taken together with the inequality from \cite{GSWW99} actually impose strong topological constraints on $\partial_{j}M$. We will not assume these stronger constraints for our main results because the proofs do not require them. However, we do include an explanation of the relationship between the genus inequality from \cite{GSWW99} and assumptions (T1)-(T2) in Figure \ref{topology} for the curious reader.

\end{Remark}
Brendle-Hung-Wang \cite{BHW12} proved the original version of this inequality for star-shaped hypersurfaces in $n$-dimensional AdS-Schwarzschild space, and Ge-Wang-Wu-Xia \cite{GWWX13} extended this to star-shaped hypersurfaces in Kottler space with $\widehat{k} \leq 0$. \eqref{minkowski2} is referred to as a Minkowski inequality because it provides a lower bound on the mean curvature integral of $\Sigma$ like the classical Minkowski inequality for convex hypersurfaces in $\mathbb{R}^{n}$.

\subsection{The Minkowski Inequality and Black Hole Uniqueness}
The horizon $\partial_{1} M$ of \\$(M^3,g,V)$ is a natural choice for the surface $\Sigma$ in Theorem \ref{minkowski}, in which case $K=1$ and $\Omega=\varnothing$. Indeed, $\partial_{1} M$ automatically satisfies the outer-minimizing requirement of the theorem (see Section 2), although we also need that the remaining horizons are non-spherical. Since the first two integrals in \eqref{minkowski2} vanish if $\Sigma=\partial_{1} M$, inequality \eqref{minkowski2} reduces to a lower bound on the product of the Euler characteristic and surface gravity of $\partial_{1} M$. After some algebra, this may be expressed as

\begin{equation} \label{surface_grav_ineq}
    \chi(\partial_{1} M) \kappa_{1} \geq \frac{1}{2} \left( 3\chi(\widehat{\Sigma}) \left( \frac{|\partial_{1} M|}{w_{2}} \right)^{\frac{1}{2}} + \widehat{k}\chi(\partial_{1}M) \left( \frac{|\partial_{1} M|}{w_{2}} \right)^{-\frac{1}{2}}  \right).
\end{equation}
We emphasize that the sign of the right-hand side equals the sign of $\widehat{k}$ in view of Remark \ref{HK_constants}. The conditions on $\partial_{1} M$ which imply saturation in \eqref{surface_grav_ineq} and hence serve as uniqueness conditions for Kottler are significantly different for spherical, toroidal, and hyperbolic infinities. First, we address the case of a spherical infinity.  There are a myriad of proofs for the uniqueness of the Schwarzschild metric, c.f. \cite{I67}, \cite{R78}, \cite{BM87} \cite{CCLP24}, \cite{R21}, and uniqueness theorems for de Sitter-Schwarzschild were recently provided in \cite{BM18}, \cite{BM20}, \cite{BCM23}. On the other hand, uniqueness of AdS-Schwarzschild remains an outstanding problem. Applying \eqref{minkowski2} to a spherical infinity, we achieve the following characterization of AdS-Schwarzschild.

\begin{Theorem}[Geometry and Uniqueness of Static Black Holes with Spherical Infinity] \label{bh_uniqueness_spherical}
Let $(M^{3},g,V)$ be an ALH static system with $\widehat{\Sigma}=S^{2}$, and suppose also that $\chi(\partial_{j}M) \leq 0$ for each $j \in \{ 2, \dots, J \}$. Then:
\begin{enumerate}
\item $\partial_{1} M$ is homeomorphic to $S^{2}$.
\item $\kappa_{1} \in [\sqrt{3},\infty)$, where $\kappa_{1}= \frac{\partial V}{\partial \nu}|_{\partial_{1} M}$.
\item If $\kappa_{1} = \sqrt{3}$, then $J=1$ and $(M^{3},g,V)$ is isometric to AdS-Schwarzschild space with mass $m=\frac{2}{3\sqrt{3}}$.
\item If $\kappa_{1} > \sqrt{3}$, then

\begin{equation} \label{area_bounds}
4\pi \rho_{-}(\kappa_{1})^{2} \leq |\partial_{1} M| \leq 4\pi \rho_{+}(\kappa_{1})^{2}
\end{equation}
for 
\begin{eqnarray}
\rho_{-}(\kappa_{1}) &=& \frac{1}{3} \left( \kappa_{1} - \sqrt{\kappa_{1}^{2} -3} \right), \label{s_minus} \\
\rho_{+}(\kappa_{1}) &=& \frac{1}{3} \left( \kappa_{1} + \sqrt{\kappa_{1}^{2}-3} \right). \label{s_plus}
\end{eqnarray}
Furthermore, if $|\partial_{1} M|=4\pi \rho_{-}(\kappa_{1})^{2}$ or $|\partial_{1} M| = 4\pi \rho_{+}(\kappa_{1})^{2}$, then $J=1$ and $(M^{3},g,V)$ is isometric to AdS-Schwarzschild space with respective mass parameter

\begin{eqnarray}
m_{-}(\kappa_{1}) &=& \frac{\rho_{-}(\kappa_{1})}{2} \left( 1 + \rho_{-}(\kappa_{1})^{2} \right), \\
m_{+}(\kappa) &=& \frac{\rho_{+}(\kappa_{1})}{2} \left( 1 + \rho_{+}(\kappa_{1})^{2} \right).
\end{eqnarray}
\end{enumerate}
\end{Theorem}
\begin{Remark}
In this paper, we have fixed the cosmological constant $\Lambda=-3$, but Theorem \ref{bh_uniqueness_spherical} applies for any negative $\Lambda$. AdS-Schwarzschild metrics for a general cosmological constant are given by

\begin{equation*}
    M^{3}= S^{2} \times (\rho_{m},\infty), \hspace{0.5cm} g_{m}= \left(1- \frac{\Lambda}{3} \rho^{2} - \frac{2m}{\rho}\right)^{-1} d\rho^{2} + \rho^{2} g_{S^{2}}, \hspace{0.5cm} V_{m}(s)= \sqrt{1 -\frac{\Lambda}{3} \rho^{2} - \frac{2m}{\rho}}. 
\end{equation*}
Assuming that $(M^{3},g,V)$ is asymptotic to hyperbolic space with curvature $\frac{1}{3} \Lambda$, we instead obtain that $\kappa \in [ \sqrt{-\Lambda},\infty )$. $\eqref{s_minus}-\eqref{s_plus}$ are then given by
\begin{eqnarray*}
\rho_{-}(\kappa) &=& \frac{1}{-\Lambda} \left( \kappa - \sqrt{\kappa^{2} + \Lambda} \right), \\
\rho_{+}(\kappa) &=& \frac{1}{-\Lambda} \left( \kappa + \sqrt{\kappa^{2} + \Lambda} \right),
\end{eqnarray*}
and the corresponding mass parameters are
\begin{eqnarray*}
m_{-}(\kappa) &=& \frac{\rho_{-}(\kappa)}{2} \left( 1 - \frac{\Lambda}{3} \rho_{-}(\kappa)^{2} \right), \\
m_{+}(\kappa) &=& \frac{\rho_{+}(\kappa)}{2} \left( 1 - \frac{\Lambda}{3} \rho_{+}(\kappa)^{2} \right).
\end{eqnarray*}
In particular, if $\kappa=\sqrt{-\Lambda}$ then $(M^{3},g)$ isometric to the Schwarzschild metric with mass $m= \frac{2}{3\sqrt{-\Lambda}}$.
\end{Remark}
To further interpret this theorem, we provide a graph of the surface gravity $\kappa$ of an AdS-Schwarzschild black hole as a function of its Schwarzschild radius in Figure \ref{surface_grav}. We see from this graph that $\kappa \geq \sqrt{3}$ within the AdS-Schwarzschild family for any $m >0$. Item (3) in Theorem \ref{bh_uniqueness_spherical} may therefore be interpreted as a uniqueness theorem for ``critical" AdS-Schwarzschild space. Also from this graph, we see for any $\kappa > \sqrt{3}$ that there are \textit{two} different AdS-Schwarzschild black holes with surface gravity $\kappa$. This is a distinct feature of the AdS-Schwarzschild family compared to other Kottler metrics. and it renders comparison techniques of an ALH static $(M^{3},g,V)$ to a ``reference" Kottler solution with a given surface gravity, used to great effect in \cite{CS01}, \cite{LN13}, \cite{B22}, an ill-defined concept when $\widehat{\Sigma}=S^{2}$. For any $\kappa> \sqrt{3}$, the area radii of the two AdS-Schwarzschild metrics with surface gravity $\kappa$ are given by \eqref{s_minus}-\eqref{s_plus} . So altogether Theorem \ref{bh_uniqueness_spherical} states that the horizon area of $\partial_{1} M$ lies between the horizon areas of the two AdS-Schwarzschild metrics of the same surface gravity, with equality triggering uniqueness.

\begin{figure}
\centering
\textbf{Surface Gravity and Schwarzschild Radii of Kottler Black Holes}
\vspace{0.2cm}

\begin{tikzpicture}
\begin{scope}[xshift=90, domain=0.15:2.5]
\draw[line width=2pt, color=blue]    plot (\x,{0.5*( (3)*(\x) + (1)/(\x))})             node[anchor=west] {$\kappa = \frac{1}{2} \left( 3 \rho_{m} + \frac{1}{\rho_{m}} \right)$};
  \draw[ultra thick,->] (-0.2,0) -- (5,0) node[right] {$\rho_{m}$};
  \draw[ultra thick,->] (0,-1) -- (0,3.8) node[above] {$\kappa$};
  \draw (0,1.73) node[anchor=east] {$\sqrt{3}$};
  \draw (0,1.73) node{$-$};
  \draw (3,5) node{$\widehat{k}=+1$ (Anti-de Sitter Schwarzschild)};
  \end{scope}
  
\begin{scope}[yshift=-200, xshift=90, domain=0:2.5]
 \draw[line width=2pt, color=blue]    plot (\x,{1.5*( \x)})             node[anchor=south] {$\kappa = \frac{3}{2} \rho_{m}$};
  \draw[ultra thick,->] (-0.2,0) -- (5,0) node[right] {$\rho_{m}$};
  \draw[ultra thick,->] (0,-1.2) -- (0,4.2) node[above] {$\kappa$};
  \draw (3,5) node{$\widehat{k}=0$ (Toroidal Kottler) };
\end{scope}

\begin{scope}[xshift=90, yshift=-400, domain=0.577:2.6]
\draw[line width=2pt, color=blue]    plot (\x,{0.5*( (-1)/(\x) + 3*\x)})  node[right] {$\kappa = \frac{1}{2} \left( 3\rho_{m} - \frac{1}{\rho_{m}} \right)$};
  \draw[ultra thick,->] (-0.2,0) -- (5,0) node[right] {$\rho_{m}$};
  \draw[ultra thick,->] (0,-1.2) -- (0,4.2) node[above] {$\kappa$};
  \draw (0.577,0) node {$|$};
  \draw (0.577,-0.6) node {$\frac{1}{\sqrt{3}}$};
  \draw (3,5) node{$\widehat{k}=-1$ (Hyperbolic Kottler)};
  \end{scope}
\end{tikzpicture}
\caption{Graphs of the surface gravity of Kottler black holes as functions of their Schwarzschild radii. The distinguishing feature of the AdS-Schwarzschild black holes is that for any $\kappa > \sqrt{3}$, there are two distinct solutions with surface gravity $\kappa$. Theorem \ref{bh_uniqueness_spherical} shows that the area radius of a general static black hole with spherical infinity lies between the Schwarzschild radii of these solutions.}
\label{surface_grav}
\end{figure}

Theorem \ref{bh_uniqueness_spherical} is of further interest from the viewpoint of conformal geometry, specifically to the characterization of Poincar\'e-Einstein fillings. Let $S^1(\lambda)$ denote the circle with length $2\pi\lambda$. Hawking-Page \cite{HP83} observed that for any ALH static system $(M^{3},g,V)$, the metric 
\begin{equation} \label{einstein_product}
    \mathrm g =V^{2}d\theta^{2} + g, \hspace{0.5cm} \theta \in S^{1}(\lambda),
\end{equation}
is a $4$-dimensional Einstein metric, $\text{Ric}(\mathrm g)=-3 \mathrm g$, with conformal infinity
\begin{equation} \label{conf_rep}
(\widehat{\Sigma} \times S^{1}(\lambda), d\theta^2 + \hat g).
\end{equation} 
 In general, the metric \eqref{einstein_product} has conical singularity at $\{ V=0 \}$. However, if all horizons of $(M^{3},g,V)$ have the same surface gravity $\kappa$, then $\bar g$ is smooth precisely for the length parameter $\lambda=\kappa^{-1}$, see Theorem \ref{Hawking-Page} for the explanation. So altogether, $\bar g$ gives rise to a \textit{Poincar\'e-Einstein filling} of \eqref{conf_rep} for $\lambda=\kappa^{-1}$.

In a recent breakthrough, Chang-Yang-Zhang \cite{CYZ25} showed that $\mathbb{H}^{3} \times S^{1}(\lambda)$ is the only Poincar\'e-Einstein filling of \eqref{conf_rep} if $\lambda$ is sufficiently large, and conjectured (\cite{CYZ25}, Question 6.2) that this uniqueness should hold for all $\lambda > (\sqrt{3})^{-1}$. The threshold comes from AdS-Schwarzschild, which yield a family of fillings with $\lambda \le (\sqrt{3})^{-1}$ through \eqref{einstein_product}. The following immediate implication of Theorem \ref{bh_uniqueness_spherical} provides evidence in support of this conjecture, as it implies that the ALH static systems considered here are not counterexamples:

\begin{Corollary}
Let $(M^{3},g,V)$ be an ALH static system with $\widehat{\Sigma}=S^{2}$, $\chi(\partial_{j}M) \leq 0$ for $j \in \{ 2,\dots,J \}$, and $\kappa_{j}=\kappa$ for each $j \in \{ 1,\dots,J \}$. Then the metric \eqref{einstein_product} is smooth only if $\lambda=\kappa^{-1} \leq \sqrt{3}$. In particular, $(M^{3},g,V)$ does not yield a Poincar\'e-Einstein filling of \eqref{conf_rep} for any $\lambda > (\sqrt{3})^{-1}$.
\end{Corollary}

Next we address the case of toroidal infinity. Before addressing black hole uniqueness, we first highlight that there is a complete solution to the static equations \eqref{static_equations} with the asymptotics \eqref{ah1}-\eqref{ah2} and infinity $\widehat{\Sigma}=T^{2}$ which is not a warped product. It is called the \textit{Horowitz-Myers geon},
\begin{equation} \label{horowitz_myers}
    M^{3}=[1,\infty) \times T^{2}, \hspace{1cm} g= \left(\rho^{2} - \frac{1}{\rho} \right)^{-1}d\rho^{2} +\left( \rho^{2} - \frac{1}{\rho} \right) d\theta_{1}^{2} + \rho^{2}d\theta_{2}, \hspace{1cm} V(r)=\rho^{2} - \frac{1}{\rho},
\end{equation}
and it served as the basis for the Horowitz-Myers conjecture. Roughly speaking, Horowitz-Myers \cite{HM98} conjectured that a complete ALH Riemannian manifold $(M^{3},g)$ with toroidal infinity has mass bounded below by that of \eqref{horowitz_myers}. Galloway-Surya-Woolgar \cite{GSW02}, \cite{GSW03} established static uniqueness theorems for the Horowitz-Myers geon under suitable asymptotic conditions, and recently Brendle-Hung \cite{BH24}, \cite{BH25} gave a resolution of the conjecture itself. For initial data sets with toroidal infinity and horizon boundary, a static black hole uniqueness theorem might provide a useful hint about potential positive energy theorems. When $\widehat{\Sigma}=T^{2}$, it is straightforward to see that the Minkowski inequality \eqref{minkowski2} is saturated if each horizon $(M^{3},g,V)$ has a non-positive Euler characteristic.
\begin{Theorem}[Uniqueness of Static Black Holes with Toroidal Infinity] \label{toroidal_uniqueness}
Let $(M^{3},g,V)$ be an ALH static system with $\widehat{\Sigma}=T^{2}$. Suppose also that $\partial M= \bigsqcup_{j=1}^{J} \partial_{j} M$ does not contain a topological $2$-sphere. Then $J=1$ and $(M^{3},g,V)$ is isometric to Kottler space \eqref{ads_schwarzschild} with $\widehat{k}=0$ and $m >0$.
\end{Theorem}
\begin{Remark}
In \cite{ACD02}, Theorems 3.4 and 3.8, Anderson-Chru\'{s}ciel-Delay also give static uniqueness theorems for the AdS-Schwarzschild and toroidal Kottler metrics. However, these theorems utilize results from \cite{A08}, and currently sufficient conditions to ensure these results hold, particularly those in Section 7 of \cite{A08}, are not known.
\end{Remark}
Theorem \ref{toroidal_uniqueness} hints at the possibility of a Penrose inequality for asymptotically locally hyperbolic manifolds with toroidal infinity for suitable horizon topologies.

Lastly, we consider the case of a hyperbolic infinity, and here \eqref{surface_grav_ineq} carries significant implications for the total mass of $(M^{3},g,V)$. In \cite{CH03}, Chru\'{s}ciel and Herzlich introduced a coordinate-invariant quantity for asymptotically locally hyperbolic Riemannian manifolds which is a natural analogue of the ADM mass for asymptotically flat manifolds.

\begin{Definition}[Mass of ALH Manifolds] \label{wang_mass}
Let $(M^{3},g)$ be an asymptotically locally hyperbolic manifold--that is, a connected, non-compact Riemannian manifold admitting the asymptotic expansion \eqref{ah1} for decay exponent $\alpha > \frac{3}{2}$. Assume also that the scalar curvature satisfies $R_g + 6 \in L^{1}(M^{3},g)$. The \textbf{Chru\'{s}ciel-Herzlich mass}\footnote{We take the term from B. Michel \cite{M11}.} of $(M^{3},g)$ is the quantity

\begin{equation}
m= \lim_{\rho \rw \infty} \frac{1}{4w_{2}} \int_{\widehat{\Sigma}^{2} \times \{ \rho \}} [ f \left(\text{div}_{\overline{g}} q - d(\text{Tr}_{\overline{g}} q)\right)(\overline{\nu}) + (\text{Tr}_{\overline{g}} q) df (\overline{\nu}) - q (\bar D f, \overline{\nu})] d\sigma_{\overline{g}} ,
\end{equation}
where $\widehat{\Sigma}^{2} \times \{ \rho \}$ is a slice of the expansion \eqref{ah1}, $\overline{g}$ is the metric $(\widehat{k} + \rho^{2}) d\rho^{2} + \rho^{2} \widehat{g}$ on $\widehat{\Sigma}^{2} \times (\rho_{0},\infty)$, $\bar D$ is the Levi-Civita connection of $\bar g$, $f$ is the function $\sqrt{\widehat{k} + \rho^{2}}$ on $\widehat{\Sigma}^{2} \times (\rho_{0},\infty)$, and $q=g - \overline{g}$.   
\end{Definition}
In Kottler space, the Chru\'sciel-Herzlich mass coincides with the mass parameter $m$, and this is related to the Schwarzschild radius by
\begin{equation*}
    m= \frac{1}{2} \left( \widehat{k} \rho_{m} + \rho_{m}^{3} \right).
\end{equation*}
It is conjectured (e.g. \cite{GWWX13}, Conjecture 1) that these metrics minimize mass for a given horizon area.
\begin{Conjecture}
Let $(M^{3},g)$ be an ALH manifold with decay exponent $\alpha > \frac{3}{2}$. Suppose that $(M^{3},g)$ has compact boundary $\partial M = \bigsqcup_{j=1}^{J} \partial_{j} M$ consisting of outermost minimal surfaces $\partial_{j} M$ and that the function $R_{g} + 6$ is non-negative and $L^{1}$ integrable on $(M^{3},g)$. Then

\begin{equation} \label{RPI}
   m \geq \frac{1}{2}\left( \widehat{k} \left( \frac{|\partial_{1} M|}{w_{2}} \right)^{\frac{1}{2}} + \left( \frac{|\partial_{1}M|}{w_{2}} \right)^{\frac{3}{2}} \right),
\end{equation}
where $m$ is the Chru\'{s}ciel-Herzlich mass of $(M^{3},g)$ and $\widehat{k}$, $w_{2}$ are the curvature and area of the infinity $(\widehat{\Sigma}^{2},\widehat{g})$. 
\end{Conjecture}
So far, this conjecture has been established for $(M^{3},g)$ with mass $m \leq 0$ by Lee-Neves \cite{LN13} \footnote{Actually, Lee-Neves prove \eqref{RPI} for the mass introduced by Wang in \cite{W01}. This requires the existence of a mass aspect function for $(M^{3},g)$, and therefore the asymptotic assumptions are stronger than those in \eqref{ah1}.} and for asymptotically locally hyperbolic graphs by Ge-Wang-Wu-Xia \cite{GWWX13}. In this paper, we find an explicit connection between \eqref{RPI} and the uniqueness of hyperbolic Kottler metrics with nonnegative mass. Note that uniqueness theorems for nonpositive mass were previously established by Lee-Neves \cite{LN13} and later Borghini \cite{B22}. 

\begin{Theorem}[Reverse Penrose Inequality for Static Black Holes with Hyperbolic Infinity] \label{Reverse Penrose Inequality}
Let $(M^{3},g,V)$ be an ALH static system with $\widehat{k}=-1$ and connected boundary $\partial M$ satisfying $\chi(\partial M) \leq \chi(\widehat{\Sigma})$. Suppose that $(M^{3},g)$ has Chru\'{s}ciel-Herzlich mass $m \geq 0$. Then

\begin{equation} \label{reverse_penrose}
    m \leq \frac{1}{2} \left( - \left( \frac{|\partial M|}{w_{2}} \right)^{\frac{1}{2}} + \left( \frac{|\partial M|}{w_{2}} \right)^{\frac{3}{2}} \right).
\end{equation}
Moreover, equality is achieved if and only if $(M^{3},g)$ is isometric to Kottler space \eqref{ads_schwarzschild} with $\widehat{k}=-1$ and $m \geq 0$. 
\end{Theorem}
Theorem \ref{Reverse Penrose Inequality} is a black hole analogue to Theorem 1.3(iii) in \cite{CS01}, which roughly states that static solitons with hyperbolic infinity, if they exist, must have mass $m < m_{\text{crit}}= -\frac{1}{3\sqrt{3}}$. Additional ``reverse positive mass inequalities" for static spaces have also been found in \cite{GW15}, \cite{W05}, \cite{AD98}. As a clear consequence of Theorem \ref{Reverse Penrose Inequality}, static uniqueness for a single black hole with $\widehat{k}=-1$ and $m \geq 0$ holds within any regime where inequality \eqref{RPI} holds. 

As we alluded to, one such regime is ALH graphs. The idea to study Penrose inequalities for graphical hypersurfaces in a reference space goes back to Lamm \cite{L10}. This approach was later applied in the asymptotically hyperbolic setting by Dahl-Gicquaud-Sakovich \cite{DGS12} and de Lima-Girao \cite{DG15}, and in the asymptotically locally hyperbolic setting by Ge-Wang-Wu-Xia \cite{GWWX13}. Theorem 1.5 in \cite{GWWX13} states that inequality \eqref{RPI} holds if $(M^{3},g)$ can be isometrically embedded as a suitable graph in $4$-dimensional Kottler spacetime with Riemannian signature. The methods in this paper actually allow us to relax the condition on this embedding when $n=3$. Specifically, we replace the assumption from \cite{GWWX13} that the boundary of $M^{3}$ is star-shaped within a time slice of the product with an outer-minimizing assumption. 
\begin{Corollary}[Static Black Hole Uniqueness for ALH Graphs] \label{alh_graph}
Let $(M^{3},g,V)$ be an ALH static system with $\widehat{k}=-1$, $\partial M=\partial_{1} M$, and $m \geq 0$. Suppose that $(M^{3},g)$ arises as a hypersurfaces in the Riemannian Kottler manifold
\begin{eqnarray} \label{riem_kottler}
   & & ( \mathbb{R} \times (\rho_{m_{0}},\infty) \times \widehat{\Sigma}^{2}, V_{m_{0},-1}(\rho)^{2} d\tau^{2} + g_{m_{0},-1}), \hspace{1cm} \text{where} \hspace{0.5cm} m_{0} > m_{\text{crit}},
\end{eqnarray}
 such that
\begin{enumerate}
    \item $\langle \eta, \frac{\partial}{\partial \tau} \rangle > 0$ on $M^{3}$, where $\eta$ is the unit normal of $M^{3}$ in \eqref{riem_kottler},
    \item $\partial M$ lies in a slice $\{ \tau = \tau_{0} \}$ of \eqref{riem_kottler}, and $\overline{M^{3}}$ intersects this slice orthogonally, AND
    \item $\partial M=\Sigma \subset (\{ \tau= \tau_{0} \},g_{m_{0},\widehat{k}})$ is an outer-minimizing surface enclosing the horizon.
\end{enumerate}
Then $(M^{3},g,V)$ is isometric to a Kottler metric \eqref{ads_schwarzschild} with $\widehat{k}=-1$ and $m \geq 0$.
\end{Corollary}

\subsection{Further Directions}
We now discuss possible extensions of the main theorems in this paper, beginning with the prospect of a stronger black hole uniqueness theorem for spherical infinity. One potential approach along these lines might involve an additional Minkowski inequality to the one in Theorem \ref{minkowski}. In \cite{HW24a}, the authors gave a new proof of Israel's theorem \cite{I67} on the uniqueness of Schwarzschild space based on a Minkowski inequality for the level sets of the static potential. Namely, we proved that if $\Sigma=\{ V = s \}$ is an outer-minimizing level set of a static potential $V$ on an asymptotically flat $3$-manifold $(M^{3},g)$ \footnote{The level-set inequality holds in all dimensions, but we focus on $n=3$.}, then

\begin{equation} \label{level_set_minkowski}
\frac{1}{2} \int_{\{ V=s \}} H d\sigma \geq s \left( \frac{|\{ V = s \} |}{4\pi} \right)^{\frac{1}{2}}.
\end{equation}
Taking \eqref{level_set_minkowski} over near-horizon level sets in an asymptotically flat static $(M^{3},g,V)$ implies an upper bound on surface gravity $\kappa$, and in turn the rigidity statement from \cite{HW24a} forces rotational symmetry of $(M^{3},g)$. If \eqref{level_set_minkowski} also holds on level sets of an asymptotically hyperbolic static potential (note that this is trivially true in AdS-Schwarzschild), the resulting upper bound on $\kappa$ would imply saturation in \eqref{surface_grav_ineq} for $\widehat{\Sigma}=S^{2}$. In other words, given inequality \eqref{minkowski2}, static uniqueness is in fact \textit{equivalent} to inequality \eqref{level_set_minkowski} over near-horizon level sets. 

In the asymptotically flat setting, the proof of inequality \eqref{level_set_minkowski} utilizes a conformal symmetry of the $\Lambda=0$ static vacuum equations, see Corollary 8.2 in \cite{HW24a}. Unfortunately, no analogous conformal symmetry exists for the static equations with $\Lambda < 0$, and so the proof method completely breaks down in the asymptotically hyperbolic setting. Therefore, either an alternative approach to proving \eqref{level_set_minkowski} or an explicit counterexample would both be very useful. 

Woolgar (\cite{W17}, p 20-21 in the published version) suggests that there may exist static configurations of multiple horizons with spherical topology and a toroidal infinity. So for $\widehat{\Sigma}=T^{2}$, it is possible that the topological condition in Theorem \ref{toroidal_uniqueness} is not removable. Finally, for Theorem \ref{Reverse Penrose Inequality}, the relation to static uniqueness makes a full realization of the Riemannian Penrose inequality even more desirable. On the other hand the nature of the inequality \eqref{reverse_penrose} suggests more subtle and complicated behavior for mass when the infinity is hyperbolic. Another possible partial result would be to establish \eqref{RPI} for small perturbations of Kottler space. Previously, Ambrozio \cite{A15} proved the Riemannian Penrose inequality with $\widehat{k}=+1$ for perturbations of AdS-Schwarzschild.

Lastly, a Minkowski inequality for higher-dimensional ALH static spaces is also an intriguing prospect. The difficulty here lies in finding suitable extensions of the arguments in Section 5 to higher dimensions. We include a more detailed discussion in Remark \ref{high_dimension}.
\subsection{Methods and Structure of the Paper}
Section 2 reviews elementary properties of static horizons and the asymptotic profile of ALH metrics.  

Section 3 introduces the quantity
\begin{equation*}
Q(t) = (|\Sigma_{t}|)^{-\frac{1}{2}} \left( \int_{\Sigma_{t}} VH d\sigma - 6 \int_{\Omega_{t}} V d\mbox{vol} + 4 \sum_{j=1}^{K} \frac{2 \pi \chi(\partial_{j} M)}{3|\partial_{j}M| + 2\pi \chi(\partial_{j}M)} \kappa_{j} |\partial_{j} M| \right),
\end{equation*}
which is shown to be monotone under the smooth inverse mean curvature flow $\Sigma_{t}$ of a surface $\Sigma$ that encloses horizons $\partial_{1} M, \dots, \partial_{K} M, \dots, \partial_{L} M$ as long as the horizons $\partial_{K+1} M, \dots, \partial_{L} M$ are not homeomorphic to $S^{2}$. This follows from the \textit{Heintze-Karcher inequality}, which says that for each $\Sigma_{t}$,

\begin{equation*}
    \int_{\Sigma_{t}} \frac{V}{H} d\sigma \geq \frac{3}{2} \int_{\Omega_{t}} V d\mbox{vol} + \sum_{j=1}^{L} c_{j} \int_{\partial_{j} M} \frac{\partial V}{\partial \nu} d\sigma.
\end{equation*}
For constants $c_{j}$ determined by the data of $\partial_{j} M$. Fogagnolo-Pinamonti \cite{FP22} determined optimal values for $c_{j}$, and amazingly their constants are determined only by the area and horizon topology in $3$-dimensional static vacuum spaces. We also see from these constants that the horizons $\partial_{K+1} M, \dots, \partial_{L} M$ make negligible contributions to the variation of $Q(t)$ given $\chi(\partial_{j} M) \leq 0$.

In Section 4, we extend the analysis in Section 3 to the weak solutions of IMCF introduced by Huisken-Ilmanen \cite{HI99}, showing in particular that $Q(t)$ remains monotone under the weak flow. A weak solution of IMCF includes ``jump times" wherein parts of the flow surface $\Sigma_{t}$ are replaced by minimal surfaces. Examining each term in $Q(t)$ makes it clear that this quantity should decrease across a jump. To rigorously justify this, we need a version of the Heintze-Karcher inequality that applies to almost every $\Sigma_{t}$ and an explicit growth formula for the total mean curvature. The former may be obtained via an approximation by mean-convex mean curvature flow as in \cite{H24}, whereas the latter is well-known for static spaces, see \cite{W17}, \cite{M18} \cite{HW24a}. It is also important to consider the presence of horizons $\partial_{K+1} M, \dots, \partial_{L} M$ not enclosed by the initial data $\Sigma_{t}$. Here, we apply the procedure described in \cite{HI99} of restarting weak IMCF as a flow surface $\Sigma_{t}$ approaches one of these horizons. This process causes $\Sigma_{t}$ to enclose additional horizons as time increases, and to preserve monotonicity we need that $\partial_{K+1} M, \dots, \partial_{L} M$ are non-spherical as in the previous section.

Section 5 addresses the regularity of IMCF in asymptotically locally hyperbolic 3-manifolds. We follow the analysis of Lee-Neves. In particular, the assumption $\min_{j \in \{ 1, \dots, J \}}\chi(\partial_{j} M) \leq \chi(\widehat{\Sigma})$ allows the use of Meek-Simon-Yau. Our new observation is that the $L^{2+\epsilon}$-norm of the mean curvature decays along a sequence of $\Sigma_{t_i}$. Together with an area bound, Allard's regularity then implies that $\Sigma_{t_i}$ eventually becomes star-shaped. For star-shaped solution of IMCF, standard arguments that works in all dimensions show that the solution becomes smooth. The regularity result was first demonstrated for IMCF in asymptotically hyperbolic 3-manifolds by Shi-Zhu \cite{SZ21}. We minimize the use of Hawking mass in the analysis so that the decay assumption is weaker and the additional decay of scalar curvature is not needed.  

In Section 6, we analyze the asymptotic behavior of smooth graphical solutions to IMCF, similar to Brendle-Hung-Wang and Ge-Wang-Wu-Xia, in order to evaluate the limit of $Q(t)$ and complete the proof of Theorem \ref{minkowski}.  

In Section 7, we prove the rigidity statement of Theorem \ref{minkowski}-- namely, equality is achieved in \eqref{minkowski} for $K=1$ if and only if $(M^{3},g, V)$ is isometric to Kottler space with $m > m_{\text{crit}}$. Recently, Borghini-Fogagnolo-Pinamonti \cite{BFP24} gave an equality statement for the Heintze-Karcher inequality, showing that if $\Omega^{n} \subset (M^{n},g)$ is bounded by a single horizon then $(\Omega^{n},g)$ splits as a warped product over $\partial M$. From here, it is easy to show that if $(M^{n},g,V)$ is static with $\Lambda=-n$ then $(\Omega^{n},g)$ is a piece $(\widehat{\Sigma}^{n-1} \times (\rho_{m},\rho_{0}), \widehat{g}_{m,\widehat{k}})$ of a Kottler metric, although in higher dimensions the base metric $\widehat{g}$ need only be Einstein. The equality statement in Theorem \ref{minkowski} comes more or less as a corollary of this, since $\Sigma_{t}$ must saturate the Heintze-Karcher inequality when $\Sigma$ achieves equality in \eqref{minkowski2}.

Finally in Section 8, we address the applications of Theorem \ref{minkowski} to black hole uniqueness. Theorems \ref{bh_uniqueness_spherical} and \ref{toroidal_uniqueness} immediately follow from the Minkowski inequality, but to prove Theorem \ref{Reverse Penrose Inequality}, we first need an upper bound on the Wang-Chru\'{s}ciel-Herzlich mass $m$ of ALH static systems in terms of the surface gravities $\kappa_{j}$ of the horizons $\partial_{j} M$, which arises from taking the Heintze-Karcher inequality over slices
$\widehat{\Sigma} \times \{ \rho \}$. Combining with the lower bound of surface gravity, the reverse Penrose inequality follows. To obtain Corollary \ref{alh_graph}, we extend the areal Minkowski inequality (Theorem 1.4 in \cite{GWWX13}) to outer-minimizing surfaces in $3$-dimensional Kottler space. In fact, we obtain a version of the areal Minkowski inequality for a general ALH static system through the same inverse mean curvature flow methods as before. This gives a lower bound on the Chru\'{s}ciel-Herzlich mass of ALH graphs in Einstein manifolds.

\begin{Notation}
    We write $g(X,Y) = \langle X,Y \rangle$ and let $D$ be the Levi-Civita connection of the ambient metric $g$. For hypersurface $\Sigma$ or $\Sigma_t$, let $\nu$ be its (outward) unit normal and  $\na$ (resp. $\Delta_\Sigma$) be the Levi-Civita connection (resp. Laplacian) of the induced metric, all with respect to $g$. The volume form in the integrals are denoted by $d\mbox{vol}$ for domains and $d\sigma$ for hypersurfaces respectively.
\end{Notation}

\textbf{Acknowledgements} The authors are grateful to Mattia Fogagnolo for comments on the rigidity statement of the Heintze-Karcher inequality and to Piotr Chru\'{s}ciel for clarifying what is currently known about black hole uniqueness with $\Lambda < 0$. We also thank the mathematics departments at Columbia University and National Yang Ming Chiao Tung University for continued support. The second author is supported by Ms. Wei-Lun Ko and NSTC grants 112-2115-M-A49-009-MY2 and 114-2115-M-A49-016-MY5.


\section{Some Preliminaries on ALH Static Systems}
\subsection{Geometry of Static Horizons}
The geometric conditions that the static equations impose over the zero set of $V$ are crucial in the construction of a monotone quantity under inverse mean curvature flow. Therefore, we review some well-known facts about static horizons which the experienced reader may skip.

\begin{Proposition}\label{no area-minimizing in interior}
Let $(M^{3},g,V)$ be an ALH static system. Then for each $j \in \{ 1, \dots, J \}$:

\begin{enumerate}
\item $\partial_{j} M$ is totally geodesic in $(M^{3},g)$.
\item $| d V |_{g}$ is equal to a constant $\kappa_{j} > 0$ on $\partial_{j} M$, known as the surface gravity of $\partial_{j} M$.
\item  $\partial_{j} M$ is outer-minimizing in $(M^{3},g)$-- that is, $|\Sigma| \geq |\partial_{j}M|$ for any surface $\Sigma$ enclosing $\partial_{j} M$. In fact, no closed area-minimizing surfaces lie in the interior of $(M^{3},g)$.
\end{enumerate}
\end{Proposition}
\begin{proof}
Since $V =0$ on $\partial_{j} M$, we have that $\Delta_g V(x) =0$ and hence $\nabla^{2} V(x)= 0$ at each $x \in \partial_{j} M$ by equation \eqref{static_equations}. The ambient and intrinsic Hessians on $\partial_{j} M$ are related by

\begin{equation*}
    D^{2} V = \nabla^{2}_{\partial_{j} M } V + \frac{\partial V}{\partial \nu} A,
\end{equation*}
where $A$ denotes the second fundamental form of $\partial_{j} M \subset (M^{3},g)$. Since $\partial M$ is a regular level set of $V$, we have that $\frac{\partial V}{\partial \nu} > 0$ and hence $A \equiv 0$ on $\partial_{j} M$. For item 2, we have

\begin{equation*}
    0 = D^{2} V(X, \nu) = X\left( \frac{\partial V}{\partial \nu} \right) - A(X,\nabla_{\partial_{j} M} V) = X \left( \frac{\partial V}{\partial \nu} \right) 
\end{equation*}
for any smooth tangent vector field $X \in \Gamma(T(\partial_{j} M))$ on $\partial_{j} M$, and thus the normal derivative is constant on $\partial_{j} M$. Lastly we address item 3. Let $\Sigma \subset (M^{3},g)$ be the minimizing hull of $\partial_{j} M$, i.e. the least-area surface enclosing $\partial_{j} M$, and suppose $\Sigma \neq \partial_{j} M$. It follows from the strong maximum principle that $\Sigma$ is disjoint from $\partial_{j} M$. Let $X: \Sigma \times (0,\epsilon) \rightarrow (M^{3},g)$ be a one-parameter family of immersions satisfying 

\begin{equation} \label{V_flow}
\frac{\partial}{\partial s} X (x,s) = V \nu(x,s). \\ 
\end{equation}
Such immersions are obtained by taking the normal exponential map of $\Sigma$ with respect to the metric $V^{-2} g$. One may calculate under \eqref{V_flow}, c.f. Proposition 3.2 in \cite{B13}, that

\begin{equation*}
\frac{d}{ds} \left( \frac{H}{V} \right) = - |A|^{2}.
\end{equation*}
Since $H(x) = 0$ for each $x \in \Sigma$, we get $H \leq 0$ on $\Sigma_{s}=X(\Sigma,s)$. So from the variation formula for the area,

\begin{equation*}
|\Sigma_{s}| - |\Sigma| = \int_{0}^{s} \int_{\Sigma_{s}} HV d\sigma ds \leq 0.
\end{equation*}
On the other hand, $\Sigma_{s}$ properly encloses $\Sigma$, contradicting the area-minimizing property of $\Sigma$. Altogether, we conclude that $\Sigma = \partial_{j} M$. Note that this argument rules the existence of any closed area-minimizing surface in $(M^{3},g)$.
\end{proof}

\subsection{Geometry of the Exterior Region}\label{subsection geometry of the exterior region}
Theorem \ref{minkowski} requires a clear asymptotic profile for the geometry of ALH $(M^{3},g,V)$, which we address in this subsection.

Recall the definition of asymptotically locally hyperbolic metric \eqref{ah1} in the introduction. We remark that  the sectional curvature of $g$ is $-1 + O(\rho^{-\alpha})$. It is convenient to work with the coordinate $r$ defined by 
\begin{align}\label{drhodr}
\frac{d\rho}{dr} = \sqrt{\rho^2 + \widehat k},
\end{align}
which can be solved explicitly
\begin{align*}
\rho &= 
\begin{cases} \sinh r,  &\widehat k=1\\
 \frac{e^r}{2},  &\widehat k=0\\
\cosh r,  &\widehat k=-1
\end{cases}
\end{align*}
If we set $\rho = \lambda(r)$, then
\begin{align}
\bar g = dr^2 + \lambda^2(r) \widehat g,
\end{align}
and since $\rho = \frac{1}{2} e^r + O(e^{-r})$, the decay assumption \eqref{ah1} becomes
\begin{align}\label{AH_r}
|q|_{\bar g} + |\bar\na q|_{\bar g} +  |\bar\na^2 q|_{\bar g} + |\bar\na^3 q|_{\bar g} = O(e^{-\alpha r}).
\end{align} 

The advantage of coordinate $r$ is that $\pl_r$ is asymptotically a unit vector.

\begin{rmk}
Let $f_i$ be a local frame of $\widehat\Sigma$ and define the frame $\varepsilon_0 = \pl_r, \varepsilon_i = \lambda^{-1} f_i$ on $(1,\infty) \times \widehat\Sigma$. Then \eqref{AH_r} is equivalent to
\begin{align*}
|q_{\alpha\beta}|, |\varepsilon_\gamma (q_{\alpha\beta})|, |\varepsilon_{\delta} (\varepsilon_{\gamma} (q_{\alpha\beta}))|, |\varepsilon_{\eta} (\varepsilon_{\delta} (\varepsilon_{\gamma} (q_{\alpha\beta})))| = O(e^{-\alpha r})  
\end{align*}
for $0 \le \alpha,\beta,\gamma,\delta,\eta \le 2$ where $q_{\alpha\beta} = q(\varepsilon_\alpha,\varepsilon_\beta)$. See (1.5) of \cite{MTX17} for example. 
\end{rmk}
Except in Subsection \ref{subsection star-shaped eventually}, the asymptotic analysis in this paper works in all dimensions. So we begin with a reformulation of the asymptotics. Let $(\widehat\Sigma, \widehat g)$ be an $(n-1)$-dimensional Riemannian manifold. Consider the metrics $\bar g = dr^2 + \lambda^2(r) \widehat g$ and $g = \bar g+ q$ defined on $(1,\infty) \times \widehat\Sigma$. Two cases are considered in this paper:
\begin{enumerate}
\item[(a)] $\lambda(r) = e^r$ and $q = O_2(e^{-kr})$ for $0< k \le 2$.
\item[(b)] $(\widehat{\Sigma},\widehat{g})$ is a space form with sectional curvature $\widehat k \in \{ 1,0,-1\}$, $\lambda(r)$ is defined in \eqref{drhodr} and $q = O_2(e^{-k r})$ for $k > 2$.
\end{enumerate}
In both cases, we have $\lambda''=\lambda$ and the sectional curvature of $g$ is $-1 + O(e^{-kr})$. 

\begin{rmk}\label{important remark}
We mostly work with asymptotics (b). The exception is Section \ref{subsection_star-shaped} where the analysis works for weak asymptotics (a). The only place that requires the decay of the third derivatives of $q$ and $\widehat g$ being a space form is Proposition \ref{asymp_C2}: The evolution equation of $h_i^j$ involves the curvature tensor $\mbox{Riem}(g)$ and its covariant derivative\footnote{In other words, most of the analysis works assuming the decay of two derivatives of $q$ and $(\widehat\Sigma, \widehat g)$ is Einstein.}.
\end{rmk}

We first present two results on the asymptotic geometry of $(1,\infty) \times \widehat\Sigma$ and then compute the evolution equation of $\langle \pl_r,\nu \rangle$ for IMCF.
\begin{prop}\label{DYrZ}
For tangent vectors $Y$ and $Z$ on $(1,\infty) \times \widehat\Sigma$ we have
\[ \langle D_Y \pl_r, Z \rangle =  \frac{\lambda'}{\lambda} \lt( \langle Y,Z \rangle - Y(r) \langle \pl_r, Z \rangle \rt) + \varepsilon(Y,Z) \]
for some $(0,2)$-tensor $\varepsilon = O_1(e^{-k r})$.
\end{prop}
\begin{proof}
First of all, $\lambda \pl_r$ is a conformal Killing vector field for $\bar g$ with conformal factor $\lambda'$. 
Note that $q = O_1(e^{-kr})$ implies the difference of Christoffel symbols $\Gamma - \bar \Gamma =O(e^{-kr})$ and we get $\lambda' \langle Y,Z \rangle + O(e^{(1-k)r}) = \langle D_Y (\lambda \pl_r),Z \rangle = \lambda \langle D_Y \pl_r,Z \rangle + \lambda' Y(r) \langle \pl_r,Z \rangle$. 
\end{proof}

\begin{prop}\label{hessian_r}
 We have $D^2 r = \frac{\lambda'}{\lambda}\lt( g - dr^2 \rt)+ O(e^{-kr})$ and $\Delta_g r = (n-1)\frac{\lambda'}{\lambda} + O(e^{-kr})$.
\end{prop}
\begin{proof}
Let $\bar D$ be the Levi-Civita connection with respect to $\bar g$ and compute
\begin{align*}
\bar D^2 r (\pl_r,\pl_r) &= 0 \\
\bar D^2 r (\pl_r,\pl_{\theta^i}) &=0\\
\bar D^2 r (\pl_{\theta^i},\pl_{\theta^j}) &= \lambda\lambda' \widehat g_{ij} = \frac{\lambda'}{\lambda}(\bar g - dr^2)(\pl_{\theta^i},\pl_{\theta^j}) 
\end{align*}
to get $\bar D^2 r = \frac{\lambda'}{\lambda} (\bar g - dr^2)$. The assertion follows from $\Gamma - \bar \Gamma = O(e^{-kr})$.
\end{proof}

\begin{prop}\label{evo}
The evolution equation of $\langle \pl_r,\nu \rangle$ along IMCF is given by
\begin{align*}
\pl_t \langle \pl_r,\nu \rangle = \frac{1}{H^2} \Big( &\Delta_\Sigma \langle \pl_r ,\nu \rangle + 2 \frac{\lambda'}{\lambda} \na r \cdot \na \langle \pl_r,\nu \rangle \\
&\quad +(n-1) \lt( \frac{\lambda'}{\lambda}\rt)^2 \langle \pl_r,\nu \rangle - 2 \frac{\lambda'}{\lambda} H \frac{\pl r}{\pl\nu}\langle \pl_r,\nu \rangle + |A|^2 \langle \pl_r,\nu\rangle \\
&\quad - |\na r|^2 \langle \pl_r,\nu \rangle + (\frac{\lambda'}{\lambda})^2 \lt( 1 - (\frac{\pl r}{\pl\nu})^2 \rt) \langle \pl_r,\nu \rangle\\
&\quad + A * O(e^{-kr}) + O(e^{-kr}) \Big) 
\end{align*}
\end{prop}
\begin{proof}
By \eqref{DYrZ}, we have
\begin{align}\label{nabla_Drnu}
\pl_i \langle \pl_r,\nu \rangle = -\frac{\lambda'}{\lambda} \pl_i r \langle \pl_r,\nu \rangle + \varepsilon(\pl_i,\nu) + h_i^j \langle \pl_r,\pl_j \rangle.
\end{align}
By Proposition \ref{DYrZ} and Codazzi equations, we have 
\begin{align*}
\Delta_\Sigma \langle \pl_r,\nu \rangle &= -\frac{\lambda'}{\lambda} \Delta_\Sigma r \langle \pl_r,\nu \rangle + \lt( (\frac{\lambda'}{\lambda})^2 -1 \rt) |\na r|^2 \langle \pl_r,\nu \rangle  - \frac{\lambda'}{\lambda} \na r \cdot \na \langle \pl_r,\nu \rangle \\
&\quad +  \gamma^{ij} D_{\pl_j} \varepsilon(\pl_i,\nu) - H \varepsilon(\nu,\nu) + h^{ij} \varepsilon(\pl_i,\pl_j) \\
&\quad + \langle \pl_r,\na H \rangle + R(\pl_r^T,\nu)  + \frac{\lambda'}{\lambda} H  - \frac{\lambda'}{\lambda} h^{ij} \pl_i r \langle \pl_r,\pl_j \rangle + h^{ij}\varepsilon(\pl_i,\pl_j)  - |A|^2 \langle \pl_r,\nu \rangle
\end{align*}
where superscript $T$ stands for the tangential projection to $\Sigma_t$.

By Proposition \ref{hessian_r}, \[ \Delta_\Sigma r = \Delta_g r - D^2 r(\nu,\nu) - H \frac{\pl r}{\pl\nu} = \frac{\lambda'}{\lambda} \lt(  n-1 + (\frac{\pl r}{\pl\nu})^2 - 1 \rt) -H \frac{\pl r}{\pl\nu} + O(e^{-kr}). \]
By \eqref{nabla_Drnu}, \[ \na r \cdot \na \langle \pl_r,\nu \rangle = - \frac{\lambda'}{\lambda} |\na r|^2 \langle \pl_r,\nu \rangle + \varepsilon(\na r,\nu) + h^{ij} \pl_i r \langle \pl_r,\pl_j \rangle.\]
Moreover, note that $A * \varepsilon$ terms are $A * O(e^{-kr})$, pure $\varepsilon$ terms are $O(e^{-kr})$, and $R(\pl_r^T,\nu) = O(e^{-kr})$. We thus arrive at
\begin{align*}
\Delta_\Sigma \langle \pl_r,\nu \rangle &= \frac{\lambda'}{\lambda} H+ \langle \pl_r,\na H \rangle + \frac{\lambda'}{\lambda} H \frac{\pl r}{\pl\nu} \langle \pl_r,\nu \rangle \\
&\quad -(n-1) (\frac{\lambda'}{\lambda})^2 \langle \pl_r,\nu \rangle - |A|^2 \langle \pl_r,\nu \rangle\\
&\quad - |\na r|^2 \langle \pl_r,\nu \rangle + (\frac{\lambda'}{\lambda})^2 \lt( 1 - (\frac{\pl r}{\pl\nu})^2 \rt) \langle \pl_r,\nu \rangle \\
&\quad -2 \frac{\lambda'}{\lambda} \na r \cdot \na \langle \pl_r,\nu \rangle + A * O(e^{-kr}) + O(e^{-kr}).
\end{align*}
The assertion follows from
\begin{align*}
\pl_t \langle \pl_r,\nu \rangle &= \frac{1}{H} \lt( \frac{\lambda'}{\lambda} - \frac{\lambda'}{\lambda} \frac{\pl r}{\pl\nu} \langle \pl_r,\nu \rangle +  \varepsilon(\nu,\nu) \rt) - \langle \pl_r, \na H^{-1} \rangle.
\end{align*}  
\end{proof}

\section{The Heintze-Karcher Inequality and the Inverse Mean Curvature Flow}
As a first step to Theorem \ref{minkowski}, we introduce a monotone quantity for the inverse mean curvature flow in static spaces with nondegenerate horizon boundary and with cosmological constant $\Lambda=-3$. For static metrics with $\Lambda=0$, monotonicity of the analogous quantity is a purely local phenomenon--see \cite{W18}, \cite{M18} and the Appendix of \cite{HW24a}. By contrast, monotonicity for $\Lambda$-negative static metrics depends on a Heintze-Karcher inequality which only applies in sub-static manifolds that are either complete or bounded by horizons. We focus on the case of a horizon boundary-- for classifications of complete static spaces, see \cite{BGH84}, \cite{Q04}, \cite{W05}, \cite{HJM19}, \cite{GSW03}, \cite{W17}.

\begin{Theorem}[\cite{FP22}, Theorem 3.5]
Let $(M^{n},g,V)$ be a sub-static Riemannian manifold. Suppose that $(M^{n},g)$ has compact, non-empty $C^{\infty}$ boundary $\partial M= \bigsqcup_{j=1}^{J} \partial_{j} M$, $V$ smoothly extends to $0$ on $\partial M$, and $\partial M= \{ V = 0 \}$ is a regular level set of $V$. Let $\Omega \subset (M^{n},g)$ be a smooth domain with boundary $\partial \Omega= \partial_{1} M \bigsqcup \dots \bigsqcup \partial_{L} M \bigsqcup \Sigma$ for some $1 \leq L \leq J$ and some mean-convex hypersurface $\Sigma^{n-1}$. Then we have the inequality 
\begin{equation} \label{heintze_karcher}
    \int_{\Sigma} \frac{V}{H} d\sigma \geq \frac{n}{n-1} \int_{\Omega} V d\mbox{vol} + \sum_{j=1}^{L} c_{j} \int_{\partial_{j} M} \frac{\partial V}{\partial \nu} d\sigma
\end{equation}
on the hypersurface $\Sigma$, where $c_{j}$ is the constant
\begin{equation} \label{c_j}
    c_{j}= \frac{ \int_{\partial_{j} M} \frac{\partial V}{\partial \nu} d\sigma}{\int_{\partial_{j} M} \frac{\partial V}{\partial \nu} \left( \frac{\Delta_g V}{V} - \frac{D^{2}V(\nu,\nu)}{V} \right) d\sigma} > 0.
\end{equation}
\end{Theorem}
Heintze and Karcher \cite{HK78} originally established this inequality for mean-convex domains of $\mathbb{R}^{n}$. Later, Brendle extended the Heintze-Karcher inequality to sub-static warped product manifolds in his breakthrough work \cite{B13}. Brendle's version of the Heintze-Karcher inequality was subsequently generalized to static manifolds with a single horizon by Wang-Wang in \cite{WW17} and to sub-static manifolds with multiple horizons by Li-Xia in \cite{LX19}. In \cite{FP22}, Fogagnolo and Pinamonti improved the constants of Li-Xia's inequality and relaxed the assumption from \cite{WW17} on the boundary components-- in particular, the constant $c_{j}$ in \eqref{c_j} is well-defined and positive whenever each $\partial_{j} M$ is homologous to a mean-convex hypersurface. When $n=3$ and $(M^{3},g,V)$ solves the static equations \eqref{static_equations}, the constant \eqref{c_j} reduces to

\begin{eqnarray}\label{c_j2}
   c_{j}&=& \frac{ \int_{\partial_{j} M} \frac{\partial V}{\partial \nu} d\sigma}{\int_{\partial_{j} M} \frac{\partial V}{\partial \nu} \left( \frac{\Delta_g V}{V} - \frac{D^{2}V(\nu,\nu)}{V} \right) d\sigma} \nonumber \\
   &=& \frac{\kappa_{j}|\partial_{j} M|}{\kappa_{j} \int_{\partial_{j} M} -\text{Ric}(\nu,\nu) d\sigma} = \frac{|\partial_{j} M|}{\int_{\partial_{j} M} \left( K_{\partial_{j} M} - \frac{1}{2} R_{g} \right) d\sigma} \label{c_j2} \\
   &=& \frac{|\partial_{j} M|}{3|\partial_{j} M| + 2\pi \chi(\partial_{j} M)}. \nonumber
\end{eqnarray}
In \cite{BHW12}, Brendle-Hung-Wang discovered an interaction between the Heintze-Karcher inequality and inverse mean curvature flow (IMCF) in AdS-Schwarzschild space. Recall that a one-parameter family of $C^{\infty}$ mean-convex immersions $X: \Sigma^{n-1} \times [0,T) \rightarrow (M^{n},g)$ move by IMCF if

\begin{equation} \label{IMCF}
\frac{\partial}{\partial t} X(p,t) = \frac{1}{H} \nu (x,t) \hspace{2cm} (x,t) \in \Sigma^{n-1} \times [0,T),
\end{equation}
where $\nu$ is the outer unit normal and $H > 0$ the mean curvature of $\Sigma_{t}=X_{t}(\Sigma)$. The first variation for the bulk integral $\int_{\Omega_{t}} V d\mbox{vol}$ under IMCF is given by the left-hand side of \eqref{heintze_karcher}, and a monotone quantity under IMCF may be deduced from this. Because the constant \eqref{c_j2} is determined by the area and topology of the horizons, the corresponding monotone quantity under IMCF in a general static background is determined by horizon area, topology, and surface gravity.
\begin{Proposition} \label{monotonicity}
Let $(M^{3},g,V)$ be a solution of the static vacuum equations \eqref{static_equations} with compact boundary $\partial M= \partial_{1} M \bigsqcup \partial_{2} M \bigsqcup \dots \bigsqcup \partial_{J} M$. Suppose that $V$ smoothly extends to $0$ on $\partial M$. Let $\Omega \subset (M^{3},g)$ be a smooth domain in $(M^{3},g)$ such that

\begin{equation} \label{boundary_assumption}
\partial \Omega= \partial_{1} M \bigsqcup \dots \bigsqcup \partial_{K} M \dots \bigsqcup \partial_{L} M \bigsqcup \Sigma, \hspace{2cm} 1\leq K \leq L,    
\end{equation}
for some connected mean-convex surface $\Sigma$. Now suppose that $\chi(\partial_{j}M) \leq 0$ for each $j \in \{ K+1, \dots, L \}$. Then the quantity
\begin{equation} \label{monotone_quantity}
Q(t)= \left( |\Sigma_{t}| \right)^{-\frac{1}{2}} \left( \int_{\Sigma_{t}} VH d\sigma - 6 \int_{\Omega_{t}} V d\mbox{vol} + 4 \sum_{j=1}^{K}  \frac{2 \pi \chi (\partial_{j} M)}{3|\partial_{j} M| +2\pi \chi(\partial_{j} M)} \kappa_{j} |\partial_{j} M|  \right)    
\end{equation}
is monotonically non-increasing along the IMCF $\Sigma_{t} \subset (M^{3},g)$ with initial condition $\Sigma_{0}=\Sigma$. Furthermore, if $Q(t)$ is constant along IMCF then each $\Sigma_{t}$ is totally umbilical.
\end{Proposition} 
In the next section, we will show monotonicity of the functional $Q(t)$ in \eqref{monotone_quantity} under weak solutions of IMCF starting from a domain $\Omega_{0}$ bounded by horizons $\partial_{1} M, \dots \partial_{K} M$. If there are additional horizons $\partial_{K+1} M, \dots \partial_{L} M$ exterior to $\Omega_{0}$, then following the approaches in \cite{HI99} and \cite{W18} we construct ``jumps" wherein weak IMCF is restarted from domains $\Omega'$ containing $\Omega_{0}$ and bounded by these additional horizons. If the horizons $\partial_{K+1} M, \dots, \partial_{L} M$ each have non-positive Euler characteristic, then their contributions to the variation formula for \eqref{monotone_quantity} are negligible, implying that $Q(t)$ retains monotonicity after the jump times.
\begin{proof}
Invoking the evolution equations for the mean curvature and volume element,
\begin{eqnarray*}
\frac{\partial}{\partial t} H &=& \frac{1}{H^{2}} \Delta_{\Sigma_{t}} H - \frac{2}{H^{3}} |\nabla H|^{2} - (|A|^{2} + \text{Ric}(\nu,\nu))\frac{1}{H}, \\
\frac{\partial}{\partial t} (d\sigma) &=& d\sigma,
\end{eqnarray*}
c.f. Theorem 3.2 in \cite{HI99} with $f=H^{-1}$, we compute
\begin{eqnarray}
    \frac{d}{d t} \int_{\Sigma_{t}} V H d\sigma &=& \int_{\Sigma_{t}} \left( \frac{\partial V}{\partial t} \right) H + V \left( \frac{\partial H}{\partial t} \right) + VH d\sigma \nonumber \\
    &=& \int_{\Sigma_{t}} \frac{\partial V}{\partial \nu} - V \Delta_{\Sigma_{t}} \frac{1}{H} - \left( |A|^{2} + \text{Ric}(\nu,\nu) \right) \frac{V}{H} + VH d\sigma \nonumber \\
    &=& \int_{\Sigma_{t}} 2 \frac{\partial V}{\partial \nu} - \frac{1}{H} \left( \Delta_{\Sigma_{t}} V + H \frac{\partial V}{\partial \nu} + V\text{Ric}(\nu,\nu) \right) - \frac{V}{H}|A|^{2} + VH d\sigma \label{integral_VH} \\
    &\leq& 2 \int_{\Sigma_{t}} \frac{\partial V}{\partial \nu} d\sigma + \frac{1}{2} \int_{\Sigma_{t}} VH d\sigma. \nonumber
\end{eqnarray}
In the last line, we used the identity $\Delta_{\Sigma_{t}} V + H \frac{\partial V}{\partial \nu} + V\text{Ric}(\nu,\nu)=0$ implied by the static equations and the inequality $H^{2} \leq 2 |A|^{2}$. Now, by the divergence theorem

\begin{equation*}
    \int_{\Sigma_{t}} \frac{\partial V}{\partial \nu} d\sigma = \int_{\Omega_{t}} \Delta_g V d\mbox{vol} + \sum_{j=1}^{L} \int_{\partial_{j} M} \frac{\partial V}{\partial \nu} d\sigma = 3 \int_{\Omega_{t}} V d\mbox{vol} + \sum_{j=1}^{L} \kappa_{j} |\partial_{j} M|,
\end{equation*}
where $\kappa_{j}= \frac{\partial V}{\partial \nu}|_{\partial_{j} M}$. Thus \eqref{integral_VH} becomes

\begin{equation*}
    \frac{d}{dt} \int_{\Sigma_{t}} VH d\sigma \leq \frac{1}{2} \int_{\Sigma_{t}} VH d\sigma + 6 \int_{\Omega_{t}} V d\mbox{vol} + 2 \sum_{j=1}^{L} \kappa_{j} |\partial_{j} M|.
\end{equation*}
For the bulk term, we apply the Heintze-Karcher inequality \eqref{heintze_karcher} with $c_{j}$ from \eqref{c_j2} to $\Sigma_{t}=\partial \Omega_{t}$ and find

\begin{equation*}
    \frac{d}{dt} \left( -6 \int_{\Omega_{t}} V d\mbox{vol} \right) = -6 \int_{\Sigma_{t}} \frac{V}{H} d\sigma \leq -9 \int_{\Omega_{t}} V d\mbox{vol} - 6 \sum_{j=1}^{L} \frac{|\partial_{j} M|}{3|\partial_{j} M| + 2\pi \chi(\partial_{j} M)} \kappa_{j} |\partial_{j} M|. 
\end{equation*}
Combining these, we obtain
\begin{eqnarray}
    & & \frac{d}{dt} \left( \int_{\Sigma_{t}} VH d\sigma - 6 \int_{\Omega_{t}} V d\mbox{vol} +  4 \sum_{j=1}^{K} \frac{2 \pi \chi (\partial_{j} M)}{3|\partial_{j} M| +2\pi \chi(\partial_{j} M)} \kappa_{j} |\partial_{j} M|  \right) \nonumber \\
    &\leq & \frac{1}{2} \int_{\Sigma_{t}} VH d\sigma - 3 \int_{\Omega_{t}} V d\mbox{vol} + 2 \sum_{j=1}^{L} \kappa_{j} |\partial_{j} M| - 6 \sum_{j=1}^{L} \frac{|\partial_{j} M|}{3|\partial_{j} M| + 2\pi \chi(\partial_{j} M)} \kappa_{j} |\partial_{j} M|. \nonumber \\
    &=& \frac{1}{2} \left( \int_{\Sigma_{t}} VH d\sigma - 6 \int_{\Omega_{t}} V d\mbox{vol} + 4 \sum_{j=1}^{L} \frac{2 \pi \chi (\partial_{j} M)}{3|\partial_{j} M| +2\pi \chi(\partial_{j} M)} \kappa_{j} |\partial_{j} M|  \right) \label{Q_variation}
\end{eqnarray}
Applying the assumption $\chi(\partial_{j} M) \leq 0$ for $j \in \{ K+1, \dots, L\}$ to the last line, we get

\begin{eqnarray*}
     & & \frac{d}{dt} \left( \int_{\Sigma_{t}} VH d\sigma - 6 \int_{\Omega_{t}} V d\mbox{vol} +  4 \sum_{j=1}^{K} \frac{2 \pi \chi (\partial_{j} M)}{3|\partial_{j} M| +2\pi \chi(\partial_{j} M)} \kappa_{j} |\partial_{j} M|  \right) \\
    &\leq& \frac{1}{2} \left( \int_{\Sigma_{t}} VH d\sigma - 6 \int_{\Omega_{t}} V d\mbox{vol} +  4 \sum_{j=1}^{K} \frac{2 \pi \chi (\partial_{j} M)}{3|\partial_{j} M| +2\pi \chi(\partial_{j} M)} \kappa_{j} |\partial_{j} M|  \right)
\end{eqnarray*}
Finally by the area growth formula $|\Sigma_{t}|=e^{t} |\Sigma_{0}|$ under IMCF, we altogether find that

\begin{equation*}
    \frac{d}{dt} Q(t) \leq 0.
\end{equation*}
Also note that equality is achieved in \eqref{integral_VH} only if $|A|^{2}=2 H^{2}$.
\end{proof}

\section{Monotonicity for Weak Inverse Mean Curvature Flow}
To obtain the Minkowski inequality for outer-minimizing surfaces, we must prove the functional $Q(t)$ introduced in the previous section is monotone under proper weak solutions of IMCF. Our monotonicity argument here is based off approaches of Huisken-Ilmanen \cite{HI99}, of Wei in \cite{W18}, and of the first author in \cite{H24}.

\subsection{The Heintze-Karcher Inequality for Weak Mean Curvature}
Before introducing weak IMCF, we first show that the Heintze-Karcher inequality \eqref{heintze_karcher} extends to $C^{1,\alpha}$ hypersurfaces with bounded non-negative, and integrable weak mean curvature. This allows one to estimate $\int_{\Omega_{t}} V d\mbox{vol}$ under weak IMCF. Our approach is approximation by mean-convex mean curvature flow. This approximation has been used in \cite{HI08} for $M=\mathbb{R}^{n}$ and in \cite{Z16} for more general backgrounds.

\begin{Theorem}[Approximation by Mean-Convex Mean Curvature Flow, \cite{Z16} Lemma 4.4] \label{mcf_approximation}
Let $(M^{n},g)$ be a Riemannian manifold, and let $\Sigma_{0}=X_{0}(\Sigma^{n-1}) \subset M^{n}$ be an immersion of class $C^{1}$ with bounded and non-negative measurable weak mean curvature $H$. Then $\Sigma_{0}$ is of class $C^{1,\alpha} \cap W^{2,p}$, and there exists a solution $X: \Sigma^{n-1} \times (0,\epsilon_{0}) \rightarrow (M^{n},g)$ to mean curvature flow (MCF),

\begin{eqnarray} \label{mcf}
    \frac{\partial}{\partial \epsilon} X(x,\epsilon) = -H \nu(x,\epsilon), \hspace{2cm} (x,\epsilon) \in \Sigma^{n-1} \times (0, \epsilon_{0}),
\end{eqnarray}
converging to $\Sigma_{0}$ locally in $C^{1,\alpha} \cap W^{2,p}$, $\alpha \in (0,1), p \in (1,\infty)$ and with  $H(x,\epsilon) \rightarrow H(x)$ strongly in $L^{p}$ for $p \in [1,\infty)$ as $\epsilon \rightarrow 0$. Furthermore, for every $\epsilon \in (0,\epsilon_{0})$ we have $\min_{\Sigma_{\epsilon}} H > 0$.
\end{Theorem}
Closely following the arguments in Section 4 of \cite{H24}, we study the reciprocal mean curvature $H^{-1}_{\epsilon}$ of $\Sigma_{\epsilon}$ under MCF as $\epsilon \searrow 0$. $|| H_{\epsilon}^{-1}||_{L^{1}(\Sigma_{\epsilon})}$ is monotonically decreasing in $\epsilon$ under MCF, and this suggests $L^{1}$ convergence as $\epsilon \searrow 0$. However, we must handle the case where $H^{-1}$ is integrable and $\text{essinf}_{\Sigma_{0}} H = 0$ delicately.
\begin{Proposition} \label{approx_liminf}
Let $\Sigma_{0}$, $\Sigma_{\epsilon}$ be as in Theorem \ref{mcf_approximation}, and suppose that the reciporical weak mean curvature $H(x,0)^{-1}$ of $\Sigma_{0}$ is $L^{1}$ integrable. Then there is a subsequence $\epsilon_{j} \searrow 0$ of times such that the inverse mean curvature $H(x,\epsilon)^{-1}$ of $\Sigma_{\epsilon_{j}}$ converges to $H(x,\epsilon)^{-1}$ strongly in $L^{1}$ as $\epsilon_{j} \searrow 0$.
\end{Proposition}
\begin{proof}
Let $U$ be a small neighborhood of $\Sigma_{0} \subset M^{n}$, and define

\begin{equation*}
    \lambda = \min_{x \in \overline{U}}  \big\{ \text{Ric}_{g} (E, E) | E \in T_{x}M^{n} \hspace{0.2cm} \text{and} \hspace{0.2cm} || E||_{g}=1 \big\}.
\end{equation*}
For each $k \in \mathbb{N}$, we consider the quantity

\begin{equation} \label{beta}
    \psi_{k}(\epsilon) =\int_{\Sigma_{\epsilon}} \frac{1}{e^{-\lambda\epsilon} H + k^{-1}} d\sigma.
\end{equation}
under \eqref{mcf}. Since $\min_{\Sigma_{\epsilon}} H > 0$, the variation formula for mean curvature under MCF in $U$ yields

\begin{eqnarray*}
    \frac{\partial}{\partial \epsilon} H &=& \Delta_{\Sigma_{\epsilon}} H + \left( |h|^{2} + \text{Ric}(\nu,\nu) \right) H \label{mcf_variation}  \\
    &\geq & \Delta_{\Sigma_{\epsilon}} H+ \left( |h|^{2} + \lambda \right) H, \nonumber
\end{eqnarray*}
and so

\begin{equation*}
    \frac{\partial}{\partial \epsilon} \left( e^{-\lambda \epsilon} H \right) \geq \Delta_{\Sigma_{\epsilon}} e^{-\lambda \epsilon} H.
\end{equation*}
Using this, we compute the evolution of \eqref{beta} under MCF:
\begin{eqnarray} \label{decreasing}
   \frac{\partial}{\partial \epsilon} \psi_{k}(\epsilon) &=& \int_{\Sigma_{\epsilon}} -\frac{1}{\left( e^{-\lambda \epsilon} H + k^{-1} \right)^{2}} \frac{\partial}{\partial \epsilon} \left( e^{-\lambda \epsilon} H \right) - \frac{1}{e^{-\lambda \epsilon} H + k^{-1}} H^{2} d\sigma \nonumber \\
   &\leq& -\int_{\Sigma_{\epsilon}} \frac{2e^{-2 \lambda \epsilon}}{\left( e^{-\lambda \epsilon} H + k^{-1} \right)^{3}} |\nabla H|^{2} d\sigma \leq 0,
\end{eqnarray}
and so \eqref{beta} is monotonically decreasing in $\epsilon$. By the $L^{p}$ convergence $H(x,\epsilon) \rightarrow H(x,0)$ and the bound $\psi_{k}(\epsilon) \leq k$, we have that
\begin{equation*}
   \psi_{k}(\epsilon_{j}) \rightarrow  \psi_{k}(0)
\end{equation*}
along a subsequence of times $\epsilon_{j}$ converging to $0$ by the dominated convergence theorem. Combining this with \eqref{decreasing} yields

\begin{equation} \label{regularized_inverse}
    \int_{\Sigma_{\epsilon}} \frac{1}{e^{-\lambda \epsilon} H + k^{-1}} d\sigma \leq  \int_{\Sigma_{0}} \frac{1}{H + k^{-1}} d\sigma, \hspace{1cm} k \in \mathbb{N}, \hspace{1cm} \epsilon \in (0,\epsilon_{0}).
\end{equation}
Now, for a fixed $\epsilon \in (0,\epsilon_{0})$ we take the $k \searrow \infty$ limit of each side of \eqref{regularized_inverse}. By monotone convergence, we get

\begin{equation*}
    \int_{\Sigma_{\epsilon}} \frac{1}{H} d\sigma \leq e^{\lambda \epsilon} \int_{\Sigma_{0}} \frac{1}{H} d\sigma.
\end{equation*}
Hence

\begin{eqnarray*}
    \limsup_{\epsilon \searrow 0} \int_{\Sigma_{\epsilon}} \frac{1}{H} d\sigma \leq \int_{\Sigma_{0}} \frac{1}{H} d\sigma,
\end{eqnarray*}
and once again by $L^{p}$ convergence of $H_{\epsilon}$ we have that that $H_{\epsilon_{j}}^{-1} \rightarrow H^{-1}$ in $L^{1}$ along a subsequence $\epsilon_{j}$.

\end{proof}

\begin{Theorem}[Heintze-Karcher Inequality for Weak Mean Curvature] \label{brendle_weak}
Let $(M^{n},g,V)$ be a sub-static Riemannian manifold. Suppose that $(M^{n},g)$ has compact, non-empty $C^{\infty}$ boundary $\partial M= \bigsqcup_{j=1}^{J} \partial_{j} M$, $V$ smoothly extends to $0$ on $\partial M$, and $\partial M= \{ V = 0 \}$ is a regular level set of $V$. Let $\Omega \subset (M^{n},g)$ be a smooth domain with boundary $\partial \Omega= \partial_{1} M \bigsqcup \dots \bigsqcup \partial_{L} M \bigsqcup \Sigma$ for some $1 \leq L \leq J$ and some $C^{1}$ hypersurface $\Sigma^{n-1}$ with bounded, non-negative, and integrable weak mean curvature $H$. Then we have the inequality 

\begin{equation*}
    \int_{\Sigma} \frac{V}{H} d \sigma \geq \frac{n}{n-1} \int_{\Omega} V d\mbox{vol} + \sum_{j=1}^{L} c_{j} \int_{\partial_{j} M} \frac{\partial V}{\partial \nu} d\sigma,
\end{equation*}
for the constants $c_{j}$ defined in \eqref{c_j}.
\end{Theorem}
\begin{proof}
 We assume $H^{-1}$ is integrable, since otherwise the statement is trivial. Let $X: \Sigma \times (0,\epsilon_{0}) \rightarrow (M^{n},g)$ be MCF from Theorem \ref{mcf_approximation}, and let $\Omega_{\epsilon} \subset (M^{n},g)$ be such that $\partial \Omega_{\epsilon}= \partial_{1} M \bigsqcup \dots \bigsqcup \partial_{K} M \bigsqcup \Sigma_{\epsilon}$. We consider the subsequence $\epsilon_{j} \searrow 0$ from Proposition \ref{approx_liminf}. By $C^{0}$ convergence of $\Sigma_{\epsilon_{j}}$ and $L^{1}$ convergence of $H_{\epsilon_{j}}^{-1}$, we have

\begin{eqnarray*}
   \lim_{j} \int_{\Omega_{\epsilon_{j}}} V d\mbox{vol} &=& \int_{\Omega_{0}} V d\mbox{vol}, \\
   \lim_{j} \int_{\Sigma_{\epsilon_{j}}} \frac{V}{H} d\sigma &=& \int_{\Sigma_{0}} \frac{V}{H} d\sigma.
\end{eqnarray*}
Since $\Sigma_{\epsilon}$ is smooth with $H>0$, we apply the inequality \eqref{heintze_karcher} over $\Sigma_{\epsilon_{j}}$ to make the conclusion. 
\end{proof}

\subsection{Monotonicity}
We are now ready to prove that the quantity $Q(t)$ remains monotone under weak IMCF, and we begin with a brief review of weak solution theory. Let $(M^{n},g)$ be a complete Riemannian manifold, and $\Omega^{n} \subset (M^{n},g)$ a smooth domain. A \textit{weak IMCF with initial condition $\Omega_{0}$} in $(M^{n},g)$ is a function $u \in C^{0,1}_{\text{loc}}(M^{n})$ such that

\begin{equation} \label{IC}
\{ u < 0 \} = \Omega_{0},
\end{equation}
and for any $t > 0$ the sub-level-sets $\Omega_{t}=\{ u < t\} \subset M^{n}$ satisfy
\begin{equation} \label{variational}
J_{u}\left( \Omega_{t} \right) \leq J_{u}(F) \hspace{1cm} \forall F \supset \Omega_{0},
\end{equation}
where $F$ is a domain of finite perimeter $|\partial^{*} F|$ in $M$ and $J_{u}$ is the functional
\begin{equation} \label{J_u}
    J_{u} \left( F \right) = |\partial^{*} F| - \int_{F} |D u| d\mbox{vol}.
\end{equation}
The relationship between the variational formulation \eqref{variational} and classical IMCF \eqref{IMCF} is as follows: for smooth $u$ with non-vanishing gradient, \eqref{IC}-\eqref{variational} are equivalent to the degenerate elliptic Dirichlet problem

\begin{eqnarray} \label{elliptic}
    \text{div} \left( \frac{D u}{|D u| } \right) &=& |D u|, \hspace{1cm} \text{in} \hspace{1cm} \left( M^{n} \setminus \overline{\Omega}_{0}, g \right), \\
    u &=& 0, \hspace{1.5cm} \text{on} \hspace{1cm} \Sigma_{0}=\partial \Omega_{0}, \nonumber
\end{eqnarray}
and the level sets $\Sigma_{t}= \{ u = t \}$ of solutions to \eqref{elliptic} are readily checked to form a solution to smooth IMCF \eqref{IMCF}. Weak  IMCF has been extremely effective in solving geometric problems on non-compact manifolds, as it was famously introduced by Huisken and Ilmanen in \cite{HI99} to prove the Riemannian Penrose inequality for asymptotically flat manifolds with non-negative scalar curvature.

Lemma 3.1 in \cite{LN13} confirms that there exists a unique weak IMCF $u \in C^{0,1}_{\text{loc}} (M^{n})$ with compact level sets for any smooth domain $\Omega_{0}$ within an asymptotically locally hyperbolic manifold $(M^{3},g)$ (note that by choosing an arbitrary fill-in of $\partial M$, we may take $(M^{3},g)$ to be complete). We summarize the key a priori properties of this solution established \cite{HI99} which are relevant to the monotonicity argument.

\begin{Theorem}[Properties of weak IMCF] \label{properties}
Let $(M^{3},g)$ be an asymptotically locally hyperbolic Riemannian manifold, let $\Omega_{0} \subset (M^{3},g)$ be a bounded domain with $\partial \Omega_{0}=\Sigma_{0}$, and let $u \in C^{0,1}_{\text{loc}}(M^{3})$ be the proper weak IMCF with initial condition $\Omega_{0}$. Then the following hold, with the corresponding reference from \cite{HI99} provided:

\begin{enumerate}[label= (\roman*)]
    \item (Theorem 1.3, Regularity) For each $t >0$, the sets $\Sigma_{t}= \partial \{ u < t \}$ and $\Sigma^{+}_{t}=\partial \{ u > t \}$ are outer-minimizing, resp. strictly outer-minimizing, $C^{1,\alpha}$ hypersurfaces.
    \item ((1.12) and Lemma 5.1, Mean-convexity) For a.e. $t >0$, $\Sigma_{t}$ and $\Sigma^{+}_{t}$ have essentially bounded weak mean curvature $H_{\Sigma_{t}} (x) = |D u(x)| >0$ for $\mathcal{H}^{2}$ a.e. $x \in \Sigma_{t}$.
    \item (Lemma 4.2, Connectedness) If $\Sigma_{0}$ is a connected hypersurface, then $\Sigma_{t}$ is a connected hypersurface for each $t >0$.
    \item ((1.10), Convergence) For each $t>0$, $\alpha \in (0,1)$, we have the convergences

    \begin{equation} \label{convergence}
        \Sigma^{+}_{s} \rightarrow \Sigma_{t} \hspace{1cm} \text{as} \hspace{1cm} s \nearrow t, \hspace{2cm} \Sigma_{s} \rightarrow \Sigma^{+}_{t} \hspace{1cm} \text{as} \hspace{1cm} s \searrow t
    \end{equation}
    in $C^{1,\alpha}$.

    \item (Lemma 5.6, Area Growth) The $\mathcal{H}^{2}$ measures of $\Sigma_{t}$ and $\Sigma^{+}_{t}$ are given by

    \begin{equation} \label{area_growth}
        |\Sigma_{t}| = |\Sigma^{+}_{t}|= e^{t}|\Sigma^{+}_{0}|.
    \end{equation}
    In particular, if $\Sigma_{0}=\partial M$ is outer-minimizing then $|\Sigma_{t}| = e^{t} |\Sigma_{0}|$.
  
\end{enumerate}
\end{Theorem}
We call $t_{0} \in [0,\infty)$ a ``jump time" if $\Sigma_{t_{0}} \neq \Sigma_{t_{0}}^{+}$. Intuitively, the reason that the quantity $Q(t)$ in \eqref{monotone_quantity} remains monotone under the weak flow is because $H_{\Sigma^{+}_{t_{0}} \setminus \Sigma_{t_{0}}} \equiv 0$ as a consequence of the outer-minimizing property, and so the first term in \eqref{monotone_quantity} decreases during a jump. On the other hand, the bulk integral $\int_{\Omega_{t}} V d\mbox{vol}$ clearly increases across jumps, while the area remains the same by item (v). The growth of the first term of \eqref{monotone_quantity} in static manifolds is already understood rigorously, c.f. \cite{W18}, \cite{M18}, \cite{HW24a}, and \cite{H24a}.

\begin{Lemma}[Growth of the Total Mean Curvature] \label{mc}
Let $(M^{3},g,V)$ be an ALH static system, and let $\Omega_{0} \subset (M^{3},g)$ be a smooth domain with $\partial \Omega_{0}= \partial_{1} M \bigsqcup \dots \bigsqcup \partial_{L} M \bigsqcup \Sigma_{0}$ for some smooth surface $\Sigma_{0}$ and some $L \leq J$. Taking fill-ins $W_{L+1}, \dots W_{J}$ of the boundary components $\partial_{L+1} M, \dots, \partial_{J} M$, consider the proper weak solution $u \in C^{0,1}_{\text{loc}}(M^{3})$ of IMCF with initial condition $\Omega_{0}$. Let $t_{0}= \sup\{ t | \Sigma_{t} \cap \left(\bigsqcup_{j=L+1}^{J} W_{j} \right) = \varnothing \}$. Then for any $0 < t_{1} < t_{2} \leq t_{0}$ we have
\begin{equation} \label{key_inequality}
    \int_{\Sigma_{t_{2}}} VH d\sigma \leq \int_{\Sigma_{t_{1}}} VH d\sigma + \int_{t_{1}}^{t_{2}} \int_{\Sigma_{s}} \left( \frac{1}{2} VH + 2 \frac{\partial V}{\partial \nu} \right) d\sigma ds .
\end{equation}
\end{Lemma}

\begin{proof}
Since $(M^{3},g)$ is asymptotically locally hyperbolic, there exists a sub-solution to IMCF with initial condition $\{ \rho=r\}$ for $r$ sufficiently large, see Lemma 3.1 in \cite{LN13}. Therefore, weak IMCF $u \in C^{0,1}_{\text{loc}}(M^{3})$ arises as the $C^{0}_{\text{loc}}$ limit of a subsequence of solutions $u_{\epsilon_{i}}$ to the regularized problem
\begin{eqnarray*}
\sqrt{|D u_{\epsilon}|^{2} + \epsilon^{2}} &=& \text{div} \left( \frac{D u_{\epsilon}}{\sqrt{|D u_{\epsilon}|^{2} + \epsilon^{2}}} \right) \hspace{2cm} \text{in} \hspace{1cm} F_{N}, \\
u &=& 0 \hspace{5.2cm} \text{on} \hspace{1cm} \Sigma_{0}, \\ 
u &=& N \hspace{5.1cm} \text{on} \hspace{1cm} \partial F_{N} \setminus \Sigma_{0},
\end{eqnarray*}
where $F_{N}= \{ v < N \}$ is defined via this sub-solution. Thus, we may directly apply Proposition 2.2 and Theorem 2.3 from \cite{H24a}, for example, as the proofs only use the sub-staticity condition

\begin{equation*}
V \text{Ric} - D^{2} V + (\Delta_g V) g \geq 0
\end{equation*}
for the potential function $V$.
\end{proof}

The Heintze-Karcher inequality for weak mean curvature accounts for the volumetric term $\int_{\Omega_{t}} V d\mbox{vol}$ under the weak IMCF.

\begin{Proposition}[Growth of the Bulk Integral] \label{bulk}
Let $(M^{3},g,V)$, $\Omega_{0}$, $u$, and $t_{0}$ be as in the previous lemma. Then for every $0 \leq t_{1} < t_{2} \leq t_{0}$, we have for the static potential $V$ that

\begin{equation} \label{volume_growth}
   \int_{\Omega_{t_{2}}} V d\mbox{vol}  \geq \int_{\Omega_{t_{1}}} V d\mbox{vol} + \int_{t_{1}}^{t_{2}} \left( \frac{3}{2} \int_{\Omega_{s}} V d\mbox{vol}  + \sum_{j=1}^{L} \frac{|\partial_{j} M|}{3|\partial_{j} M| + 2\pi \chi(\partial_{j} M)} \kappa_{j} |\partial_{j} M| \right) ds
\end{equation}
where $\kappa_{j}=\frac{\partial V}{\partial \nu}|_{\partial_{j} M}$.
\end{Proposition}
\begin{proof}
Since $u \in C^{0,1}_{\text{loc}}(M^{3})$, $D u(x)$ is defined for 
 $\mathcal{H}^{3}$ a.e. $x \in M^{3} \setminus \Omega_{0}$. From this, we define the measurable function $h: M^{3} \setminus \Omega_{0} \rightarrow \mathbb{R}$ by

\begin{equation*}
    h(x) = \begin{cases} 0 & D u(x) =0, \\
   V(x)|D u|^{-1} & D u(x) \neq 0.
   \end{cases} 
\end{equation*}
Applying the co-area formula for Lipschitz functions from Theorem 2.7.3 and Remark 2.7.4 in \cite{S18}, we obtain

\begin{eqnarray*}
    \int_{\Omega_{t_{2}}} V d\mbox{vol} - \int_{\Omega_{t_{1}}} V d\mbox{vol}  \geq \int_{\Omega_{t_{2}} \setminus \Omega_{t_{1}}} h |D u| d\mbox{vol} = \int_{t_{1}}^{t_{2}} \int_{\Sigma_{s}} h d\sigma ds = \int_{t_{1}}^{t_{2}} \int_{\Sigma_{s}} \frac{V}{H} d\sigma ds.
\end{eqnarray*}
Here, we used the fact that $\Sigma_{t} = \Sigma^{+}_{t} = \{ u = t\}$ for a.e. $t \in (0,\infty)$. Likewise, the last equation follows from property (ii) in Theorem \ref{properties}. Now, applying the Heintze-Karcher inequality for weak mean curvature with  the $n=3$ constant $c_{j}$ from \eqref{c_j2}, we get

\begin{eqnarray*}
    \int_{\Omega_{t_{2}}} V d\mbox{vol} - \int_{\Omega_{t_{1}}} V d\mbox{vol} &\geq& \int_{t_{1}}^{t_{2}} \left( \frac{3}{2} \int_{\Omega_{s}} V d\mbox{vol} + \sum_{j=1}^{L} c_{j} \int_{\partial_{j} M} \frac{\partial V}{\partial \nu} \right) d\sigma ds \\
    &=& \int_{t_{1}}^{t_{2}} \left( \frac{3}{2} \int_{\Omega_{s}} V d\mbox{vol}  + \sum_{j=1}^{L} \frac{|\partial_{j} M|}{3|\partial_{j} M| + 2\pi \chi(\partial_{j} M)} \kappa_{j} |\partial_{j} M| \right) ds.
\end{eqnarray*}

\end{proof}

\begin{Theorem}
Let $(M^{3},g,V)$ be an ALH static system, and let $\Omega_{0} \subset (M^{3},g)$ be a smooth domain with
\begin{equation*}
    \partial \Omega_{0}=\partial_{1} M \bigsqcup \dots \bigsqcup \partial_{K} M \bigsqcup \dots \bigsqcup \partial_{L} M \bigsqcup \Sigma_{0}, \hspace{2cm} 1 \leq K \leq L \leq J.
\end{equation*}
for some outer-minimizing surface $\Sigma_{0}$. Suppose also that the components $\partial_{j} M$ have genus at least $1$ for $j \in \{ K+1, \dots, L \}$. Taking fill-ins, consider the proper weak solution $u \in C^{0,1}_{\text{loc}}(M^{3})$ of IMCF with initial condition $\Omega_{0}$. Then the quantity

\begin{equation*}
Q(t)= \left( |\Sigma_{t}| \right)^{-\frac{1}{2}} \left( \int_{\Sigma_{t}} VH d\sigma - 6 \int_{\Omega_{t}} V d\mbox{vol} + 4 \sum_{j=1}^{K} \frac{2\pi \chi(\partial_{j} M)}{3|\partial_{j} M| + 2\pi \chi(\partial_{j} M)} \kappa_{j} |\partial_{j} M| \right)
\end{equation*}
is monotonically non-increasing under weak IMCF with initial condition $\Omega_{0}$ for each $t \in [0,t_{0})$, where $t_{0}$ is defined in Lemma \ref{mc}.
\end{Theorem}
\begin{proof}
For each $t \in (0,t_{0})$ we have that $\partial \Omega_{t}= \Sigma_{t} \bigsqcup \partial_{1} M \bigsqcup \dots \bigsqcup \partial_{K} M$. After applying the divergence theorem, \eqref{key_inequality} becomes

\begin{eqnarray*}
    \int_{\Sigma_{t_{2}}} VH d\sigma &\leq& \int_{\Sigma_{t_{1}}} VH d\sigma + \int_{t_{1}}^{t_{2}} \int_{\Sigma_{s}} \left( \frac{1}{2} VH + 2 \frac{\partial V}{\partial \nu} \right) d\sigma ds \\
    &=& \int_{\Sigma_{t_{1}}} VH d\sigma + \int_{t_{1}}^{t_{2}} \left(  \frac{1}{2} \int_{\Sigma_{s}} VH d\sigma + 6 \int_{\Omega_{s}} V d\mbox{vol}  + 2 \sum_{j=1}^{L} \kappa_{j} |\partial_{j} M| \right) ds.
\end{eqnarray*}
Combining this with the lower bound on $\int_{\Omega_{t}} V d\mbox{vol}$ from \eqref{volume_growth}, we find for the quantity

\begin{equation*}
    q(t)= \int_{\Sigma_{t}} VH d\sigma - 6 \int_{\Omega_{t}} V d\mbox{vol}  + 4 \sum_{j=1}^{K} \frac{2\pi \chi(\partial_{j} M)}{3|\partial_{j} M| + 2\pi \chi(\partial_{j} M)} \kappa_{j} |\partial_{j} M|
\end{equation*}
that

\begin{eqnarray} \label{q}
    q(t_{2}) - q(t_{1}) &\leq& \int_{t_{1}}^{t_{2}}\left( \frac{1}{2} \int_{\Sigma_{s}} V H d\sigma + 6 \int_{\Omega_{s}} V d\mbox{vol}   + 2 \sum_{j=1}^{L} \kappa_{j} |\partial_{j} M| \right) ds \nonumber \\
    & & - \int_{t_{1}}^{t_{2}} \left( 9 \int_{\Omega_{s}} V d\mbox{vol} + 6 \sum_{j=1}^{L} \frac{|\partial_{j} M|}{3|\partial_{j} M| + 2\pi \chi(\partial_{j} M)} \kappa_{j} |\partial_{j} M| \right) ds \nonumber \\
    &=& \int_{t_{1}}^{t_{2}} \left( \frac{1}{2} \int_{\Sigma_{s}} VH d\sigma -3 \int_{\Omega_{s}} V d\mbox{vol} + 2 \sum_{j=1}^{L} \frac{2\pi \chi(\partial_{j} M)}{3|\partial_{j} M| + 2\pi \chi(\partial_{j} M) } \kappa_{j}|\partial_{j}M| \right) ds.  \nonumber \\
    &\leq& \frac{1}{2} \int_{t_{1}}^{t_{2}} q(s) ds.
\end{eqnarray}
for any $0 < t_{1} < t_{2} \leq t_{0}$, where we have used $\chi(\partial_{j}M) \leq 0$ for $j \in \{ K+1, \cdots, L \}$. Therefore, Gronwall's Lemma applied to $Q(t)=|\Sigma_{t}|^{-\frac{1}{2}} q(t)$ that $Q(t_{2})= |\Sigma_{t_{2}}|^{-\frac{1}{2}} q(t_{2}) \leq Q(t_{1})$ for $t_{2} > t_{1} >0$. If $\Sigma_{0}$ is outer-minimizing, then by Theorem \ref{properties} (v) we may also take $t_{1}=0$.
\end{proof}
To conclude the section, we introduce artificially constructed jumps if the initial data does not enclose all horizons of $(M^{3},g)$ as in \cite{HI99}.  
\begin{Theorem}[Monotonicity for Weak IMCF with Obstacles]\label{monotonicity obstacle}
Let \\$(M^{3},g,V)$ be an ALH static system. Let $\Omega_{0} \subset (M^{3},g)$ be a domain such that

\begin{equation*}
    \partial \Omega_{0} = \partial_{1} M \bigsqcup \dots \bigsqcup \partial_{K} M \bigsqcup \Sigma, \hspace{2cm} 1 \leq K \leq J,
\end{equation*}
for some outer-minimizing surface $\Sigma$. Then there exists a flow of domains $\{ \Omega_{t} \}_{t \in (0,\infty)}$ in $(M^{3},g)$ satisfying the following:

\begin{enumerate}
    \item $\partial \Omega_{t} = \partial_{1} M \bigsqcup \dots \bigsqcup \partial_{L} M \bigsqcup \Sigma_{t}$ for $K \leq L \leq J$ and for some connected, outer-minimizing surface $\Sigma_{t}$.
    \item $\{ \Omega_{t} \}_{t \in [0,\infty)}$ solves weak IMCF except possibly at a finite number of times $t_{1}, t_{2}, \dots t_{M}$.
    \item For each of the times $t_{1}, \dots, t_{M}$, we have for 
    \begin{equation*}
     Q(t)= \left( |\Sigma_{t}| \right)^{-\frac{1}{2}} \left( \int_{\Sigma_{t}} VH d\sigma - 6 \int_{\Omega_{t}} V d\mbox{vol} + 4 \sum_{j=1}^{K} \frac{2\pi \chi(\partial_{j} M)}{3|\partial_{j} M| + 2\pi \chi(\partial_{j} M)} \kappa_{j} |\partial_{j} M| \right)
    \end{equation*}
    that $\liminf_{t \nearrow t_{i}} Q(t) \geq \limsup_{t \searrow t_{i}} Q(t)$.
\end{enumerate}
As a consequence, if the boundary components $\partial_{K+1} M, \dots, \partial_{J} M$ satisfy $\chi(\partial_{j} M) \leq 0$, then $Q(t)$ is monotonically non-increasing along this flow. 
\end{Theorem}
\begin{proof}
We apply the construction described in \cite{HI99}, Section 6, see also Section 4.2 of \cite{W17}. In particular, at the jump times $t_{i}$, $\Sigma_{t_{i}}$ is replaced by the minimizing hull $\Sigma_{t_{i}}'$ of $\Sigma_{t_{i}}$ and at least one of the horizons $\partial_{k+1} M, \dots, \partial_{L} M$. Note that $\Sigma_{t_{i}}'$ is connected and outer-minimizing and hence the resulting domain $\Omega_{t_{i}}$ satisfies item (1). Moreover, by the outer-minimizing properties $\Sigma_{t_{i}}$ of $\Sigma_{t_{i}}'$ and the nesting $\Omega_{t_{i}} \subset \Omega_{t_{i}}'$, we have

\begin{eqnarray*}
    \int_{\Sigma_{t_{i}}} VH d\sigma &\geq& \int_{\Sigma_{t_{i}}'} VH d\sigma, \\
    -\int_{\Omega_{t_{i}}} V d\mbox{vol} &\geq& - \int_{\Omega_{t_{i}}'} V d\mbox{vol}, \\
    |\Sigma_{t_{i}}|^{-\frac{1}{2}} &\geq& |\Sigma_{t_{i}}'|^{-\frac{1}{2}}.
\end{eqnarray*}
Taking the weak IMCF $\Sigma_{t}$ of $\Sigma_{t_{i}}'$, we obtain item (3) by the convergences described in Theorem \ref{properties} (iv). Finally, given that $\partial \Omega_{t}= \Sigma_{t} \bigsqcup \partial_{1} M \bigsqcup \dots \bigsqcup \partial_{K} M \dots \bigsqcup \partial_{L} M$, we have that $Q(t)$ is monotonically non-increasing on $(t_{i}, t_{i+1})$ given the topological assumption according to the previous theorem.
\end{proof}

\section{Regularity of IMCF in Asymptotically Locally Hyperbolic $3$-Manifolds}\label{section regularity}
Lee-Neves \cite[Section 3]{LN13} proved that the weak solutions to IMCF exist for all time in $n$-dimensional ALH\footnote{In this section, the static potential $V$ plays no role.} manifolds for $C^{2}$ initial data $\Sigma_{0}$. A consequence of their construction is that $\Sigma_t$ lies in the exterior region for sufficiently large $t$ and 
\begin{align}\label{C^0_rho}
c_1 e^{\frac{t}{n-1}} \le \rho \le c_2 e^{\frac{t}{n-1}}
\end{align}
on $\Sigma_t$ for some positive constants $c_1$ and $c_2$. All the estimates in this section will depend on $c_1$ and $c_2$.

The goal of this section is to show that weak solutions in 3-dimensional ALH manifold eventually become smooth. This is proved in Shi-Zhu \cite{SZ21} when $(\widehat\Sigma,\widehat g)$ is the unit sphere $S^2$ under stronger decay conditions for $q$.

Item (3) of Proposition \ref{no area-minimizing in interior} and assumption (T1) imply that assumptions (1), (2) of Theorem \ref{main_regularity} hold.

\begin{Theorem}\label{main_regularity}
Let $(M,g)$ be a 3-dimensional ALH manifold, namely that it satisfies \eqref{ah1} of Definition \ref{alh_static}. If $\widehat\Sigma \neq S^2$, assume in addition 
\begin{enumerate}
\item $\pl M = \cup_{j=1}^J \pl_j M$ consists of compact minimal surfaces and $M$ contains no other closed area-minimizing surfaces, AND
\item  The genus of some boundary component, say $\pl_1 M$, is greater or equal to the genus $\mathfrak{g}$ of $\widehat\Sigma$.
\end{enumerate}

  Let $\{ \Sigma_t\}_{t \ge 0}$ be a weak solution to IMCF starting from a smooth surface $\Sigma_0$ in $M$. Then $\Sigma_t$ is a smooth solution to \eqref{IMCF} for $t$ sufficiently large. 
\end{Theorem}

\begin{Remark}
Assumption (1) can be found in Lemma 3.3 of \cite{LN13}. Note the different meaning of their ``exterior region". Assumption (2) can be found in Theorem 1.1 of \cite{LN13}. They enable one to compare the Euler characteristic of $\widehat\Sigma$ and $\Sigma_t$ in \eqref{Lem3.3}. Note that \eqref{Lem3.3} automatically holds when $\widehat\Sigma = S^2$. 
\end{Remark}

\subsection{Weak Solutions in 3 Dimension become Star-Shaped  Eventually}\label{subsection star-shaped eventually}
We first recall the setup in Lee-Neves. Consider a conformal compactification $\tilde g = \rho^{-2} g$ of the exterior region of $(M,g)$. Set $s = \rho^{-1}$ and we have
\begin{align}\label{tilde g}
\tilde g = ds^2 + \widehat g + \tilde q
\end{align}
on $(0,s_1) \times \widehat \Sigma$ for sufficiently small $s_1$, where
\begin{align}\label{AH tilde g}
|\tilde q|_{\tilde g} + s |\tilde\na \tilde q|_{\tilde g} + s^2|\tilde\na^2 \tilde q|_{\tilde g} = O_2(s^2).
\end{align}
Note that the error term is $O(s^2)$ instead of $O(s^\alpha)$ because $\rho^{-1}$ is a ``rough" conformal factor. Equation \eqref{C^0_rho} translates into
\begin{align}\label{C^0_s}
c_1 e^{-t/2} \le s \le c_2 e^{-t/2}
\end{align}
for sufficiently large $t$ (with different $c_1$ and $c_2$). 

\begin{lem}\cite[Lemma 3.7]{LN13}
Let $\tilde\Sigma_t$ denote the surface $\Sigma_t$ endowed with the metric induced from $\tilde g$. Then the area of $\tilde\Sigma_t$ is uniformly bounded in time.
\end{lem}

For the rest of this subsection, we omit the volume form $d\sigma$ of the surface when there is no confusion. We also remind the readers that the ambient metric $g$ satisfies \eqref{ah2} with $\alpha > 2$.
\begin{lem}\cite[Lemma 3.8]{LN13}\label{upperlower} $\int_{\Sigma_t} H^2-4 $ is bounded above and below.
\end{lem}
\begin{proof}
Our argument of the upper bound is different from Lee-Neves. For a smooth flow
\begin{align*}
\pl_t \int_{\Sigma_t} (H^2 -4) \,d\sigma &= \int_{\Sigma_t} -\Delta_\Sigma \frac{1}{H} \cdot 2H - 2|A|^2 - 2Ric(\nu,\nu) + H^2 -4 \\
&= \int_{\Sigma_t} -\frac{2}{H^2} |\na H|^2 + H^2 - 2|A|^2 -2 Ric(\nu,\nu) -4 \\
&\le \int_{\Sigma_t} -2 Ric(\nu,\nu) - 4. 
\end{align*}
The computation can be passed to the weak flow as in Section 5 of \cite{HI99}. Recalling $Ric(\nu,\nu) = -2 + O(e^{-\alpha r}) = -2 + O(e^{-\frac{\alpha}{2} t})$ and $ |\Sigma_t| = e^t |\Sigma_0|$, we get
\[ \int_{\Sigma_{t}} (H^2 -4) d\sigma \le  \int_{\Sigma_{t_0}} (H^2-4) d\sigma  + O(e^{1-\frac{\alpha}{2}t}). \]

The proof of lower bound proceeds as in Lee-Neves. Using the Gauss equation and conformal invariance of the integrals $\int_\Sigma |\mathring{A}|^2 + 2K$, we get
\begin{align*}
\int_{\Sigma_t} H^2 + 2R - 4 Ric(\nu,\nu) = \int_{\tilde\Sigma_t} \tilde H^2 + 2 \tilde R - 4 \widetilde{Ric}(\tilde\nu,\tilde\nu).
\end{align*}
Therefore, 
\begin{align}\label{(9) in LN13}
\int_{\Sigma_t} H^2 -4 = \int_{\Sigma_t} 4 Ric(\nu,\nu) + 8 - 2 (R+6) + \int_{\tilde\Sigma_t} \tilde H
^2 - C 
\end{align}
for some constant $C$ independent of $t$. Since the first integral on the right-hand side is $O(e^{(1-\frac{\alpha}{2})t})$, the lower bound follows.
\end{proof}

The next two lemma is exactly the same as in Lee-Neves.
\begin{lem}\cite[Lemma 3.9]{LN13}\label{L2graphical}
We have
\[ \int_{\tilde\Sigma_t} |\tilde\na^T s|^2 = O(e^{-t/2}). \]
\end{lem}

\begin{lem}\cite[Lemma 3.10]{LN13}\label{area}
\[  |\tilde\Sigma_t| = |\widehat\Sigma_t| + O(e^{-t/2}). \]
\end{lem}

\begin{lem}\cite[Lemma 3.11]{LN13}\label{seq of surface}
There exists a sequence of times $t_i$ such that
\[ \lim_{i \rw \infty} \int_{\Sigma_{t_i}} |\mathring A|^2 =0. \]
\end{lem}
\begin{proof}
Fix $\xi$. As in Lee-Neves, (5.22) of \cite{HI99} and Lemma \ref{upperlower} imply
\begin{align*}
\int_{\Sigma_\xi} H^2 -4 &\ge \int_{\Sigma_\eta} H^2-4 + \int_\xi^\eta \int_{\Sigma_t} 2|\mathring A|^2 + 2 (Ric(\nu,\nu) + 2) \\
&\ge -C + \int_\xi^\eta \lt( \int_{\Sigma_t} 2 |\mathring A|^2 \rt) dt + \int_\xi^\eta O(e^{(1-\frac{\alpha}{2})t }) dt
\end{align*}
as $\eta \rw \infty$. Consequently $\int_{\Sigma_t} |\mathring A|^2$ cannot be bounded from below by any positive constant for all $t$.
\end{proof}

In Lee-Neves, the last three lemmas are used to  control the Euler characteristic of $\Sigma_t$ in the next proposition. In the analysis of Hawking mass, which behaves well under the weak solution, this topological information suffices. Nevertheless, we will see momentarily the last three lemmas are pivotal in the regularity of IMCF.

\begin{prop}\cite[Lemma 3.12]{LN13}\label{EC}
Using the same sequence as in Lemma \ref{seq of surface}, we have $\chi(\Sigma_{t_i}) \ge \chi(\widehat\Sigma)$ for sufficiently large $i$. 
\end{prop}
\begin{proof}
As in Lee-Neves, we have
\begin{align}\label{EC_comparison}
8\pi \chi(\tilde\Sigma_t) + \int_{\tilde\Sigma_t} 2 |\mathring{\tilde A}|^2 \ge 8\pi \chi(\widehat\Sigma) + O(e^{-t/4}) +  \int_{\tilde\Sigma_t} \tilde H^2. 
\end{align}
The assertion follows from the previous lemma.
\end{proof}
Next we follow Lemma 2.1 of Shi-Zhu \cite{SZ21} to prove exponential decay of $\int_{\Sigma_t} |\mathring A|^2$ using the lower bound of Hawking mass. 
\begin{lem}
Let $\Sigma_{t_i}$ be the sequence of surfaces as in Proposition \ref{EC}. Then
\begin{align*}
\int_{\Sigma_{t_i}} |\mathring A|^2 = O(e^{(1-\frac{\alpha}{2})t_i}).
\end{align*}
\end{lem}
\begin{proof}
Define the Hawking mass
\begin{align}
m_H(\Sigma) := \sqrt{\frac{|\Sigma|}{16\pi}} \lt( 1 - \mathfrak{g} - \frac{1}{16\pi} \int_\Sigma H^2-4 \rt)
\end{align}
where $\mathfrak{g}$ is the genus of $\widehat\Sigma$. Fix $\xi \ge 0$. By Theorem 3.2 of \cite{LN13}, for $\eta > \xi$
\begin{multline*}
m_H(\Sigma_\eta) - m_H(\Sigma_\xi) \ge \frac{1}{2} (16\pi)^{-3/2} \int_{\xi}^\eta |\Sigma_t|^{1/2} \bigg[ 8\pi (\hat\chi - \chi(\Sigma_t)) \\
+ \int_{\Sigma_t} 2(R+6) + |\mathring A|^2 + 4H^{-2}|\na H|^2 \bigg] dt, 
\end{multline*}
where $\hat\chi = \chi(\widehat\Sigma) = 2-2\mathfrak{g}$. For $t$ sufficiently large, $\Sigma_t = \pl \{ u < t \}$ encloses $\pl M$ and Lemma 3.3 of \cite{LN13} implies
\begin{align}\label{Lem3.3}
 \hat\chi \ge \chi(\Sigma_t).
\end{align}
Therefore, for $\xi,\eta$ sufficiently large
\begin{align*}
m_H(\Sigma_\eta) - m_H(\Sigma_\xi) \ge (16\pi)^{-3/2} \int_\xi^\eta |\Sigma_t|^{1/2} \lt( \int_{\Sigma_t} R+6 \rt) dt \ge  -C \int_\xi^\eta e^{\frac{1}{2} (3-\alpha) t} dt.
\end{align*}
Fixing $\xi$ and sending $\eta=t$ to infinity, we get
\begin{align*}
m_H(\Sigma_t) \ge - O(e^{\frac{1}{2}(3-\alpha)t})
\end{align*}
for $t$ sufficiently large. By Gauss equation, 
\begin{align*}
H^2-4 = 4K + 2|\mathring A|^2 + 4(Rc(\nu,\nu) +2 ) - 2(R+6). 
\end{align*}
By Proposition \ref{EC} and \eqref{Lem3.3}, we have $\hat\chi = \chi(\Sigma_{t_i})$ and
\begin{align*}
- O(e^{\frac{1}{2}(3-\alpha)t_i}) \le m_H(\Sigma_{t_i}) = \sqrt{\frac{|\Sigma_{t_i}|}{16\pi}} \cdot \frac{1}{8\pi} \int_{\Sigma_{t_i}} \lt( |\mathring A|^2 + O(e^{-\frac{\alpha}{2}t_i})\rt).
\end{align*}
It follows that $\int_{\Sigma_{t_i}} |\mathring A|^2 = O(e^{(1 -\frac{\alpha}{2}) t_i})$.
\end{proof}

Next we observe that the $L^2$ norm of the full second fundamental form and $L^{2+\epsilon}$ norm of the mean curvature decays along $\tilde\Sigma_{t_i}$.
\begin{lem}\label{A2H2+}
Let $\Sigma_{t_i}$ be the sequence of surfaces as in Proposition \ref{EC}. Then 
\begin{align}\label{A2}
\int_{\tilde\Sigma_{t_i}} |\tilde A|^2 = O(e^{(1-\frac{\alpha}{2})t_i})
\end{align}
and
\begin{align}\label{H2+}
\int_{\tilde \Sigma_{t_i}} \tilde H^{2+\epsilon} = O(e^{(1+\frac{\epsilon}{2}-\frac{\alpha}{2})t_i}).
\end{align} 
\end{lem}
\begin{proof}
By \eqref{EC_comparison} and the conformal invariance of $\int |\mathring A|^2$, we get  $\int_{\tilde\Sigma_{t_i}} \tilde H^2 \le O(e^{(1-\frac{\alpha}{2})t})$, assuming $\alpha-2 \ll \frac{1}{4}$. Together with the previous lemma, this proves \eqref{A2}. 

By the formula of the mean curvature under conformal change,
\[ H = s \tilde H + 2 \tilde\nu(s),\]
and the fact that $H$ is uniformly bounded \cite[Theorem 3.1]{HI99}, we have $\tilde H = O(s^{-1}) = O(e^{\frac{t}{2}})$ and \eqref{H2+} follows.
\end{proof}

The last ingredient is a localized version of the area bound, Lemma \ref{area}.
\begin{lem}
For any open set $\hat U$ of $\widehat\Sigma$ with smooth boundary, we have
\begin{align*}
\lt| \tilde\Sigma_t \cap \lt( (0,s_1) \times \hat U \rt) \rt| = |\hat U| + O(e^{-t/2}).
\end{align*}
\end{lem}
\begin{proof}
Let $\tilde M_t$ be the region in $(0,s_1) \times \widehat\Sigma$ so that $\pl \tilde M_t$ is the union of $\tilde\Sigma_t \cap \lt( (0,s_1) \times \hat U \rt)$, $\{ 0\} \times \hat U$, and a lateral surface $\mathcal{C} = \pl \tilde M_t \cap \lt( (0,s_1) \times \pl \hat U \rt)$. See Figure \ref{figure:star-shape}. By the divergence theorem, we have
\begin{align*}
\int_{\tilde\Sigma_t \cap (0,s_1) \times \hat U} \tilde\nu (s) + \int_{\{ 0\} \times \hat U} \tilde\nu(s) + \int_{\mathcal{C}} \tilde\nu(s) = \int_{\tilde M_t} \Delta_{\tilde g} s
\end{align*}
where $\tilde\nu$ denotes the outward unit normal of $\pl \tilde M_t$. By \eqref{AH tilde g}, $\tilde\nu(s) = O(s^{\alpha-1})$ on $\mathcal{C}$, $\Delta_{\tilde g} s = O(s^{\alpha-1})$ in $\tilde M_t$, $\mbox{area}(\mathcal{C}) = O(s)$, and $\mbox{vol}(\tilde M_t) = O(s)$. Moreover, $\tilde\nu(s)$ is exactly $-1$ on $\{ 0\} \times \hat U$. Recall that $s \sim e^{-t/2}$ and we get
\begin{align*}
\int_{\tilde\Sigma_t \cap (0,s_1) \times \hat U} \tilde\nu (s) = |\hat U| + O(e^{- \alpha \cdot \frac{t}{2}}).
\end{align*}
The rest of the proof is identical to that of Lemma 3.10 of Lee-Neves with $\tilde\Sigma_t$ replaced by $\tilde\Sigma_t \cap (0,s_1) \times \hat U$.  
\end{proof}
We are ready to prove the main result of this subsection.
\begin{prop}
Let $\Sigma_{t_i}$ be the sequence of surfaces as in Proposition \ref{EC}. Then for any $\eta > 0$, there is an $i_0$ such that we have
\begin{align*}
\langle \pl_r,\nu \rangle \ge 1-\eta
\end{align*}
on $\Sigma_{t_i}$ for $i \ge i_0$.
\end{prop}
\begin{proof}
Fix a small constant $\delta > 0$ such that for any $\hat x \in \widehat\Sigma$, $\widehat g$ is close to the Euclidean metric on $B_\delta (\hat x)$ and $\tilde g$ is also close to the (3-dimensional) Euclidean metric on $(0,\delta) \times \widehat\Sigma$ (Recall that $\tilde g$ is close to a product metric). By the previous lemma and \eqref{H2+}, we have
\begin{align*}
|\tilde\Sigma_t \cap (0,\delta) \times B_\delta(\hat x)| = |B_\delta(\hat x)| + O(e^{-t/2})
\end{align*}
and
\begin{align*}
\int_{\tilde\Sigma_{t_i}} \tilde H^{2+\epsilon} \rw 0 \mbox{ as } i \rw \infty
\end{align*}
for $0 < \epsilon < \alpha -2$. For $\delta$ sufficiently small and $i$ sufficiently large, Allard's regularity theorem, Theorem \ref{Allard_Riemannian}\footnote{More precisely, the varifold $V = \Sigma_{t_i}$ is a $C^1$ surface. Hence the multiplicity $\theta =1$ everywhere and $\mu(B_\rho(x)) = |\Sigma_{t_i} \cap B_\rho(x)|$. We can choose $\rho_0$ small, independent of $x$, so that the metric $\tilde g$ is close to Euclidean metric on $B_{\rho_0}(x)$ so that the first inequality of \eqref{Allard assumption 2} holds.}, implies that $\tilde\Sigma_{t_i} \cap (0,\delta) \times B_\delta (\hat x)$ can be written as the graph of a $C^{1,\alpha}$ function $u_i$ with $\| u_i\|_{C^{1,\alpha}}$ uniformly bounded independent of $i$.

Now suppose $\langle \tilde\nu,\pl_s \rangle < 1 - 2\eta$ at some point $x \in \tilde\Sigma_{t_i}$. By the $C^{1,\alpha}$ bound, we have $\langle \tilde\nu,\pl_s \rangle \le 1 - \eta$ on $B_{c\delta}(x) \subset \tilde\Sigma_{t_i}$ and the area lower bound $|B_{c\rho}(x)| > c'$ for constants $c$ and $c'$ depending only on $\| u_i\|_{C^{1,\alpha}}$. Since $\langle \tilde\nu,\pl_s \rangle^2 + |\tilde\na^T s|^2 = 1 + O(s^\alpha)$, we get
\begin{align*}
\int_{\tilde\Sigma_{t_i}} |\tilde\na^T s|^2 \ge \int_{B_{c\rho}(x)} \eta + O(s^\alpha) \ge c' \eta + O(e^{-\alpha \frac{t_i}{2}}), 
\end{align*} contradicting Lemma \ref{L2graphical}. Finally, from
\begin{align*}
\pl_s = \frac{\rho^2}{-1 + \frac{\widehat k}{2} \rho^{-2}} \pl_\rho = \frac{\rho^2}{-1 + \frac{\widehat k}{2} \rho^{-2}} \frac{1}{\sqrt{\rho^2 + \widehat k}} \pl_r,
\end{align*}
we see the relation 
\begin{align*}
\langle \tilde\nu,\pl_s \rangle = \tilde g (\tilde\nu,\pl_s) = \rho^{-2} g (-\rho\nu, \frac{\rho^2}{-1 + \frac{\widehat k}{2} \rho^{-2}} \frac{1}{\sqrt{\rho^2 + \widehat k}} \pl_r) = \frac{1}{(1 - \frac{\widehat k}{2}\rho^{-2})\sqrt{1 +\widehat k \rho^{-2}}}  \langle \pl_r,\nu \rangle
\end{align*}
and conclude that $\langle \pl_r,\nu \rangle \ge 1 - \eta$ on $\Sigma_{t_i}$ for $i$ sufficiently large.
\end{proof}

\begin{figure}\label{figure:star-shape}
    \centering
 \textbf{Star-Shape of IMCF in Asymptotically Locally Hyperbolic 3-Manifolds}
    \vspace{1cm}

\begin{tikzpicture}
\begin{scope}[xshift=90, domain=-3:3]
\draw[line width=2.5pt] (0,0) ellipse (3 and 1);
\draw[line width=2.5pt, domain=3:5] plot(\x, {3*(\x-3)^(0.5)} );
\draw[line width=2.5pt, domain=-5:-3] plot(\x, {3*(-3 - \x)^(0.5)});

\begin{scope}[scale=0.7, yshift=-13]
\draw[color=blue!50, line width=2.5pt, domain=-0.5:0.5, name path=A] plot (\x,{-(1- 0.1111*((\x)^(2)))^(0.5)});
\draw[line width=2.5pt, color=blue!50, domain=-0.6:-0.5, name path=B] plot(\x, {5*(-0.5-\x)^(0.5) - 0.98}) ;
\draw[line width=2.5pt, color=blue!50, domain=0.5:0.6, name path=C] plot(\x, {5*(\x-0.5)^(0.5) - 0.98}) ;
\draw[color=blue!50, line width=2.5pt, domain=-0.6:-0.5, name path=D] plot (\x,{-(1- 0.1111*((\x)^(2)))^(0.5) + 1.58});
\draw[color=blue!50, line width=2.5pt, domain=-0.5:0.5, name path=E] plot (\x,{-(1- 0.1111*((\x)^(2)))^(0.5) + 1.58});
\draw[color=blue!50, line width=2.5pt, domain=0.5:0.6, name path=F] plot (\x,{-(1- 0.1111*((\x)^(2)))^(0.5) + 1.58});
\tikzfillbetween[of=B and D]{blue!50, opacity=0.3};
\tikzfillbetween[of=A and E]{blue!50, opacity=0.3};
\tikzfillbetween[of=C and F]{blue!50, opacity=0.3};
\end{scope}
\filldraw[color=blue, fill=blue](0,-1) circle (0.1);
  \end{scope}
\draw (6.2,0) node[anchor=west]{\large{$\{ s=0 \} =\widehat{\Sigma}$}};
\draw (6.7,2) node[anchor=west]{\large{$(M^{3},\widetilde{g})$}};
\draw (3,0.5) node[color=blue]{$B_{\delta}(\widehat{x}) \times (0,\delta)$} ;
\draw (3.1,-1.2) node[color=blue, anchor=north] {\large{$\widehat{x}$}} ;
\draw [red, line width=2.5pt] plot [smooth] coordinates {(0.1,0.5) (0.5,0) (2,0.3) (2.2,-0.1) (3,-0.2)  (3.4,-0.65) (4.2,-0.4) (5,-0.2) (6.2,0.4) };
\draw [red, line width=2.5pt, dashed] plot [smooth] coordinates {(0.1,0.5) (0.5,1.2) (2,1.6) (3,2) (4, 1.3) (5,1.5) (5.6,1.1) (6.15, 0.6) (6.2,0.4)};
\draw (0.1,0.4) node[color=red, anchor=east] {\large{$\widetilde{\Sigma}_{t_{i}}$}};

\begin{scope}[xshift=86, yshift=-80]
\draw(0,0) node{\Huge{$\Downarrow$}};
\end{scope}

\begin{scope}[xshift=46, yshift=-250]
\draw[line width=2.5pt] (-1,0) -- (4,0);
\draw [red, line width=2.5pt] plot [smooth] coordinates {(0,4) (1.2,3.7) (1.8,3)  (2.3,1.8) (3,2) };
\draw[red] (1.2,2) node{$\text{graph}(u_i)$};
\draw (1.5,0.5) node{$\tilde M_t$};
\draw (-0.3,1) node{$\mathcal{C}$};
\draw (3.3,1) node{$\mathcal{C}$};
\draw[color=blue!50, line width=2.5pt, fill=blue, opacity=0.1] (0,0) rectangle (3,4.5);
\draw[color=blue!50, line width=2.5pt] (0,0) rectangle (3,4.5);
\draw[color=red] (3,3) node[anchor=west] {$\partial_{\alpha} u_i \rightarrow 0$} ;
\end{scope}
\end{tikzpicture}
 \caption{Given $L^{2+\epsilon}$ decay of $H$ and $L^{2}$ decay of $|\nabla^{T} s|$, $| A|$, Allard's regularity gives that $\Sigma_{t_{i}}$ can be written as the graph in the tube $B_{\delta}(\widehat{x}) \times (0,\delta)$ of a function converging to a constant in $C^{1,\alpha}$.}
    \label{graph_figure}
\end{figure}
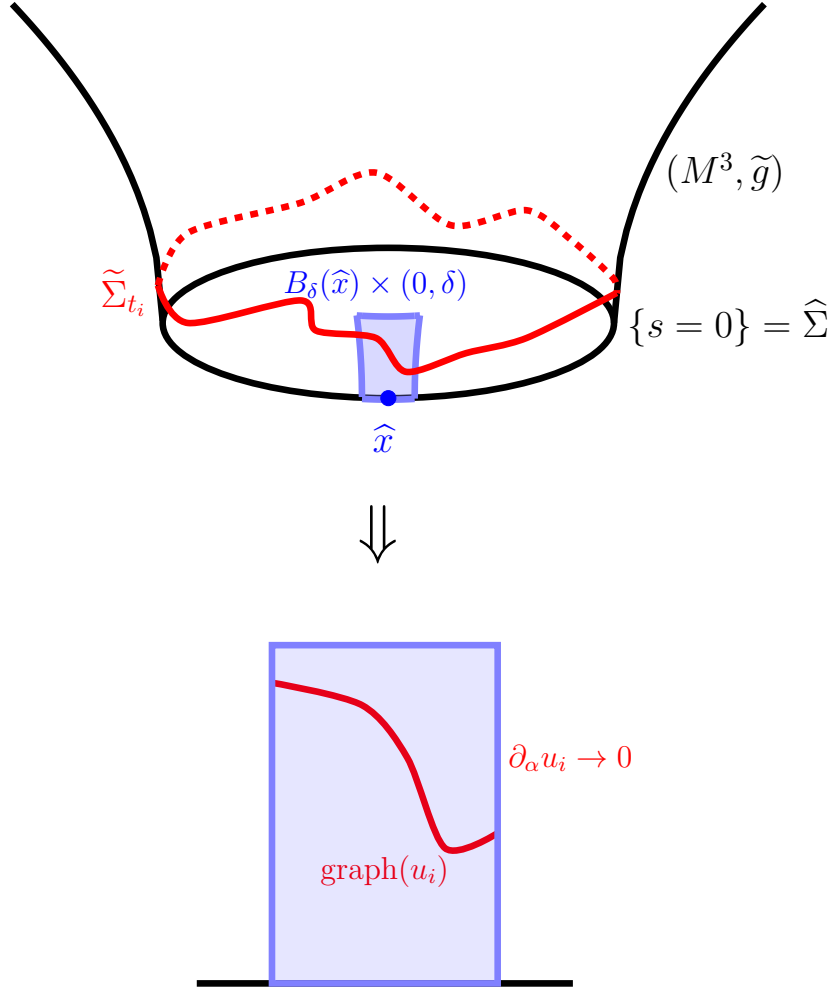

\subsection{Smooth Star-Shaped Solutions in $n$ Dimensions}\label{subsection_star-shaped}
We first review the approximation scheme of Huisken-Ilmanen \cite{HI08} in proving regularity for weak solutions. Consider the solution of mean curvature flow $\Sigma_{t_i,\tau}, 0 < \tau \le \epsilon_i$ evolving from $\Sigma_{t_i}$ obtained in the previous subsection. By Theorem \ref{mcf_approximation}, $\Sigma_{t_i,\tau}$ is smooth for $\tau >0$ and converges to $\Sigma_{t_i}$ in $C^{1,\beta}$ as $\tau \rw 0$; in particular, $\langle \pl_r, \nu \rangle$ is close to $1$ on each $\Sigma_{t_i,\tau}$. Moreover, by strong maximum principle and the assumption of no compact minimal surfaces, each $\Sigma_{t_i,\tau}$ has positive mean curvature. Next, consider \underline{smooth} solution of IMCF $\Sigma_{t_i,\tau,t-t_i}, 0 \le t-t_i \le \epsilon'_i$ with $\Sigma_{t_i,\tau,0} = \Sigma_{t_i,\tau}$. We aim to prove uniform estimates that can pass, as $\tau \rw 0$, to the original weak solution $\Sigma_t, t > t_i$.

In this subsection we work with asymptotics (a) defined in Subsection \ref{subsection geometry of the exterior region}.  

A hypersurface $\Sigma \subset (1,\infty) \times \widehat\Sigma$ is said to be {\it star-shaped} if $\langle \nu, \pl_r \rangle \ge c > 0$ on $\Sigma$. Let $F: \widehat \Sigma \times [T,T_1] \rw (1,\infty) \times \widehat\Sigma$ be a smooth solution to IMCF starting from a star-shaped hypersurface and satisfying
\begin{align}\label{C^0_r}
c_1 e^{\frac{t}{n-1}} \le e^r \le c_2 e^{\frac{t}{n-1}}
\end{align}
on $\Sigma_t = F(\widehat\Sigma \times \{ t\})$. We assume that $T$ is large so that the effect of $q$ is negligible via \eqref{C^0_r}.

Noting that $1 - |\na r|^2 - (\frac{\pl r}{\pl \nu})^2 = O(e^{-kr})$, the evolution equation simplifies to
\begin{align*}
\pl_t \langle \pl_r,\nu \rangle &= \frac{1}{H^2} \Big( \Delta_\Sigma \langle \pl_r ,\nu \rangle + 2 \na r \cdot \na \langle \pl_r,\nu \rangle \\
&\qquad\qquad  + (n-1) \langle \pl_r,\nu \rangle - 2 H \frac{\pl r}{\pl\nu}\langle \pl_r,\nu \rangle + |A|^2 \langle \pl_r,\nu\rangle + A * O(e^{-kr}) + O(e^{-kr}) \Big) . 
\end{align*}
 
\begin{prop}\label{starshape}
Let $\Sigma_t$ be a smooth solution to IMCF on $(1,\infty) \times \widehat\Sigma$. Then for any small $\delta > 0$, there exists a large time $T$ such that $\langle \nu,\pl_r \rangle \ge c \ge \frac{1}{2}$ on $\Sigma_T$ implies $\langle \nu,\pl_r \rangle > c-\delta$ on $\Sigma_t$ for $t \ge T$. 
\end{prop}
\begin{proof}
Suppose $\langle \pl_r,\nu \rangle$ hits $c - \delta$ for the first time at $T_2$. Let $\beta = (1 - L e^{-lt}) \langle \pl_r,\nu\rangle$ with $l = \frac{k}{n-1}$ and $L$ a large constant. Choosing $T$ sufficiently large, we have
\begin{align}\label{temp_alpha_lb}
\beta \ge c - 2\delta 
\end{align}
on $[T,T_2]$. By Cauchy-Schwarz, we have
\begin{align*}
\pl_t \beta \ge \frac{1}{H^2} \lt( \Delta_\Sigma \beta + 2 \na r \cdot \na \alpha + I \cdot \beta - O(e^{-kr}) \rt) 
\end{align*}
where
\begin{align*}
I &= (n-1) - 2  H \frac{\pl r}{\pl\nu} + |A|^2 (1 - O(e^{-kr})) + \frac{Lle^{-lt}}{1 - e^{-lt}} H^2.
\end{align*}

For $T,L$ sufficiently large, we have
\begin{align*}
I &\ge (n-1)  - 2  H \frac{\pl r}{\pl\nu} + \lt( \frac{1}{n-1} + \frac{L l}{2} e^{-lt} \rt)H^2 \\
&= \lt( \sqrt{ \frac{n-1}{1 + \frac{Lk}{2} e^{-lt}} }  \frac{\pl r}{\pl \nu} - \sqrt{\frac{1+ \frac{Lk}{2} e^{-lt}}{n-1}} H \rt)^2 + (n-1) \lt( 1 - \frac{1}{1 + \frac{Lk}{2} e^{-lt}} (\frac{\pl r}{\pl\nu})^2 \rt) \\
&\ge (n-1) \cdot \frac{Lk}{4} e^{-lt}
\end{align*}
where we used $\frac{\pl r}{\pl \nu} = \langle \pl_r,\nu \rangle + O(e^{-kr}) \le 1 + O(e^{-kr})$. By \eqref{temp_alpha_lb}, $I \cdot \beta - O(e^{-kr}) > 0$ for $T,L$ sufficiently large and we get
\begin{align*}
\pl_t \beta > \frac{1}{H^2} \lt( \Delta_\Sigma \beta + 2 \na r \cdot \na\beta \rt).
\end{align*}
Maximum principle then implies $\min \beta$ is increasing and hence  
\[ (1 - L e^{-lT_1})(c-\delta) \ge (1 - L e^{-lT}) c,\] which is a contradiction if $T$ is sufficiently large. We conclude that $\langle \pl_r ,\nu \rangle$ never falls below $c-\delta$.
\end{proof}

We now come to the main theorem of this subsection, which is adapted from Lemma 4.1 of \cite{LW17}.
\begin{Theorem}
Let $\Sigma_t$ be a smooth solution to IMCF in $(1,\infty) \times \widehat\Sigma$. Suppose $\langle \pl_r,\nu \rangle > \frac{2\sqrt{2}}{3}$ on $\Sigma_t$. Then there exists a large time $T$ such that
\begin{align*}
H \ge C \min \{  (t-T)^{1/2}, 1 \}
\end{align*}
on $\Sigma_t$ for $t > T$.
\end{Theorem}
\begin{proof}
Let $X = e^r \pl_r$. As in Proposition \ref{evo}, one derives  
\begin{align}\label{dXnu}
\pl_i \langle X,\nu \rangle = \langle X, h_i^j \pl_j \rangle + \varepsilon_1(\pl_i,\nu)
\end{align}
for some $(0,2)$ tensor $\varepsilon_1 = O(e^{(1-k)r})$ and the evolution equation
\begin{align*}
\pl_t \langle X,\nu \rangle &= \frac{1}{H^2} \lt( \Delta_\Sigma \langle X,\nu \rangle + |A|^2 \langle X,\nu \rangle + A * O(e^{(1-k)r}) + O(e^{(1-k)r}) \rt)
\end{align*}
where $R(X^T,\nu)$ term is absorbed into $O(e^{(1-k)r})$. By Cauchy-Schwarz, we have $A * O(e^{(1-k)r}) \ge - |A|^2 \langle X,\nu \rangle e^{-kr} - O(e^{(1-k)r}) $ and hence
\begin{align*}
\pl_t \langle X,\nu \rangle^2 &\ge \frac{1}{H^2} \lt( \Delta_\Sigma \langle X,\nu \rangle^2 - 2 |\na \langle X,\nu \rangle|^2 + 2(1 - O(e^{-kr})) |A|^2 \langle X,\nu \rangle^2 - O(e^{(2-k)r}) \rt).
\end{align*}
By \eqref{dXnu}, we have $|\na \langle X,\nu \rangle|^2 \le 2 |A|^2 |X^T|^2 + O(e^{(2-2k)r})$. Moreover the assumption $\langle \pl_r,\nu \rangle > \frac{2\sqrt{2}}{3}$ implies that $|X^T|^2 < \frac{1}{8} \langle X,\nu \rangle^2$ for $t$ sufficiently large. Therefore, we have
\begin{align*}
\pl_t \langle X,\nu \rangle^2 \ge \frac{1}{H^2} \lt( \Delta_\Sigma \langle X,\nu \rangle^2 + |A|^2 \langle X,\nu \rangle^2 - O(e^{(2-k)r}) \rt)
\end{align*}
and, combining with the evolution equation of $H$, \begin{align*}
\pl_t \lt( H \langle X,\nu \rangle^2	\rt) &\ge \frac{1}{H^2} \lt( \Delta_\Sigma H \cdot \langle X,\nu \rangle^2  +  H \Delta_\Sigma \langle X,\nu \rangle^2\rt) - \frac{2}{H^3} |\na H|^2 \langle X,\nu \rangle^2 \\
&\qquad + \frac{1}{H} \lt( - R(\nu,\nu)\langle X,\nu \rangle^2 - O(e^{(2-k)r}) \rt) \\
&= \frac{1}{H^2} \Delta_\Sigma (H \langle X,\nu \rangle^2) - 2 \frac{\na H}{H^3} \cdot \na \lt( H\langle X,\nu \rangle^2 \rt) + (n-1-O(e^{-kr})) \frac{\langle X,\nu \rangle^2 }{H}
\end{align*}
for $t$ sufficiently large, say $t \ge T$.
Let $u = \lt( H \langle X,\nu \rangle^2 \rt)^{-1}$. We conclude that
\begin{align*}
\pl_t u \le \mbox{div} (H^{-2}\na u) -2 H^{-2} u^{-1} |\na u|^2 - (n-1 - O(e^{-kr})) u^3 \langle X,\nu \rangle^4 
\end{align*}
for $t \ge T$. 

To simplify notation, set $\Sigma'_t = \Sigma_{T + t}$.
For an arbitrary $t_0 > 0$, let $v = (t-t_0)^{\frac{1}{2}} u$ and we get
\begin{align*}
\pl_t v &\le \mbox{div} (H^{-2} \na v) - (t-t_0)^{\frac{1}{2}} \cdot 2 H^{-2} v^{-1}|\na v|^2 \\
&\quad + \frac{1}{2} (t-t_0)^{-1} v - \lt( n-1-O(e^{-kr}) \rt) v^3 \langle X,\nu \rangle^4.
\end{align*}

Let $v_K = \max (v-K,0)$ for $K \ge 0$ and let $A(K) = \{ x \in \Sigma'_t, v(x,t) > K \}$. Multiplying by $v_K$ and integrating by parts, we get 
\begin{align*}
\frac{d}{dt} \int_{\Sigma'_t} v_K^2 d\sigma &\le (t-t_0)^{-1} \int_{A(K)} vv_K d\sigma + \int_{A(K)} v_K^2 d\sigma \\
&\quad - 2\lt( n-1 - O(e^{-kr}) \rt) (t-t_0)^{-1} \int_{A(K)} v^3 v_K \langle X,\nu \rangle^4 d\sigma.
\end{align*}
By Proposition \ref{starshape}, we have $\langle X,\nu \rangle \ge c e^{\frac{t+T}{n-1}} \le c e^{\frac{t}{n-1}}$. Using $v > K$ on $A(K)$ and $v \ge v_K$, we obtain
\begin{align*}
\frac{d}{dt} \int_{\Sigma_t} v_K^2 d\sigma &\le (t-t_0)^{-1} \int_{A(K)} vv_K d\sigma + \int_{A(K)} v_K^2 d\sigma \\
&\quad - 2  \lt( n-1 - O(e^{-kr}) \rt) (t-t_0)^{-1} K^2 c^4 e^{\frac{4}{n-1} t} \int_{A(K)} v v_K d\sigma \\
&\le 0
\end{align*}
on the interval $[t_0,t_1]$ provided
\begin{align}\label{conditionK}
K^2 \ge K_0^2 = \frac{1}{n-1 - O(e^{-kr})} c^{-4} e^{-\frac{4}{n-1}t_0} \max (t_1-t_0,1). 
\end{align} 
By definition $v_K(x,t_0) \equiv 0$ and we get $v_K(x,t) \equiv 0$ for all $t \in [t_0,t_1]$ if \eqref{conditionK} holds.

Separate into small and large times:
\begin{itemize}
\item If $t_1 \le 2$, choose $t_0 = \frac{t_1}{2} \le 1$ and then \[ K_0^2 = \frac{1}{n-1 - O(e^{-kr})} c^{-4} e^{-\frac{4}{n-1}t_0}. \]
From the definition of $v_K$, we see that
\begin{align*}
\sup_{\Sigma'_{t_1}} u \le C t_1^{-\frac{1}{2}} e^{-\frac{2}{n-1} t_1} 
\end{align*}
since $t_0 = \frac{t_1}{2} \le 1$.
\item If $t_1 \ge 2$, choose $t_0 = t_1 -1 \ge 1$ and then
\begin{align*}
K_0^2 = \frac{1}{n-1 - O(e^{-kr})} c^{-4} e^{-\frac{4}{n-1}t_0}.
\end{align*}
We see that
\begin{align*}
\sup_{\Sigma'_{t_1}} u \le C e^{-\frac{2}{n-1} t_1}.
\end{align*}
\end{itemize}
Therefore, in any case we have
\[ \sup_{\Sigma'_t} u \le C e^{-\frac{2}{n-1}t} \max (t^{\frac{1}{2}}, 1). \]
From the definition $u = (H \langle X,\nu \rangle^2)^{-1}$ and $\langle X,\nu \rangle \ge c e^{\frac{t}{n-1}}$, we obtain
\begin{align*}
\min_{\Sigma_t} H \ge C \min ((t-T)^{\frac{1}{2}}, 1).
\end{align*} 
for $t > T$.
\end{proof}

\begin{prop}\cite[Proposition 3.6]{BHW12}
Suppose the mean curvature is bounded from above and below
\[ 0 < H_0 \le H \le H_1 \] by positive constants $H_0$ and $H_1$. Then the second fundamental form is bounded.
\end{prop}

\begin{proof}[Proof of Theorem \ref{main_regularity}]
By the above proposition and Krylov's regularity theorem, the smooth solution to IMCF $\Sigma_{t_i,\tau,t-t_i}$ exists for all time $0 \le t-t_i < \infty$ and is uniformly bounded independent of $\tau$. Let $\tau \rw 0$ and we get a solution of IMCF that starts from $\Sigma_{t_i}$ and is smooth for positive time. By the uniqueness of IMCF, this flow coincides with the original weak flow.
\end{proof}

\section{The Asymptotic Behavior of the Flow as $t \rw \infty$}\label{asymptotic behavior}
In this section, we evaluate the limit of $Q(t)$ defined in Section 3, following \cite{BHW12} or Section 5 of \cite{GWWX13}. The asymptotics we work with are the higher dimensional analogue of \eqref{ah1}.

\noindent{\bf Assumption.}
Let $(\widehat\Sigma, \widehat g)$ be an $(n-1)$-dimensional closed \underline{oriented} space form with sectional curvature $ \widehat k \in \{ 1,0,-1\}$. Then the metric
\[ \bar g = \frac{1}{\rho^2 + \widehat k} d\rho^2 + \rho^2 \widehat g \]
defined on $(1,\infty) \times \widehat\Sigma$ has constant sectional curvature $-1$. Let $g$ be a metric on $(1,\infty) \times \widehat\Sigma$ satisfying
\begin{align*}
g = \bar g + q
\end{align*}
with $q = O_3(\rho^{-\alpha})$ for some $\alpha > 2$. 

Consider smooth graphical solutions $F: \widehat\Sigma \times [T,\infty) \rw (1,\infty) \times \widehat\Sigma$ to IMCF given by the function $r(t,\theta), \theta \in \widehat\Sigma$. We may assume there exist positive constants $c$ and $C$ such that 
\begin{align}\label{assumption for asymptotic analysis}
H \ge c ,  \langle \pl_r,\nu \rangle \ge c, |A| \le C,
\end{align}
which will be improved momentarily, on $\Sigma_t$ in addition to
\[ c_1 e^{\frac{t}{n-1}} \le e^r \le c_2 e^{\frac{t}{n-1}}. \] 

First we analyze the asymptotic behavior of $\langle \pl_r,\nu \rangle$ and $H$.

\begin{prop}\label{asymp_C1} We have $\langle \pl_r,\nu \rangle = 1 + O(e^{-\frac{2}{n-1}t})$. 
\end{prop}
\begin{proof}
Since the function $|D\vphi|^2_{S^{n-1}}$ used in \cite{BHW12} is equal to $\langle \pl_r,\nu \rangle^{-2} -1$ there, we work with $\langle \pl_r,\nu \rangle^{-2} -1$. From \eqref{assumption for asymptotic analysis} and $|\na r|^2 = 1 - (\frac{\pl r}{\pl\nu})^2 + O(e^{-\alpha r})$, the evolution equation simplifies to
\begin{align*}
\pl_t \langle \pl_r,\nu \rangle = \frac{1}{H^2} \Big( &\Delta_\Sigma \langle \pl_r ,\nu \rangle + 2 \frac{\lambda'}{\lambda} \na r \cdot \na \langle \pl_r,\nu \rangle \\
&\quad +(n-1) \lt( \frac{\lambda'}{\lambda}\rt)^2 \langle \pl_r,\nu \rangle - 2 \frac{\lambda'}{\lambda} H \frac{\pl r}{\pl\nu}\langle \pl_r,\nu \rangle + |A|^2 \langle \pl_r,\nu\rangle \\
&\quad + \frac{\widehat k}{\lambda^2} \lt( 1 - (\frac{\pl r}{\pl\nu})^2 \rt) \langle \pl_r,\nu \rangle + O(e^{-\alpha r}) \Big). 
\end{align*}
Since $\frac{\pl r}{\pl\nu} = \langle \pl_r,\nu \rangle + O(e^{-\alpha r})$, $\frac{\pl r}{\pl\nu}$ can be replaced by $\langle \pl_r,\nu \rangle$. Direct computation yields 
\begin{align*}
\pl_t \langle \pl_r,\nu \rangle^{-2} &= \frac{1}{H^2} \lt( \Delta_\Sigma \langle \pl_r, \nu \rangle^{-2}  -6 \langle \pl_r,\nu \rangle^{-4} |\na \langle \pl_r,\nu \rangle|^2 - 4 \langle \pl_r,\nu \rangle^{-3} \frac{\na\lambda}{\lambda} \cdot \na \langle \pl_r,\nu \rangle \rt)\\
&\quad - \frac{2}{H^2} \langle \pl_r,\nu \rangle^{-3} \lt( -2 \frac{\lambda'}{\lambda} H \langle \pl_r,\nu \rangle^2 + (n-1) \lt( \frac{\lambda'}{\lambda}\rt)^2 \langle \pl_r,\nu \rangle + |A|^2 \langle \pl_r,\nu \rangle \rt) \\
&\quad - \frac{2\widehat k}{H^2} \langle \pl_r,\nu \rangle^{-2} \cdot \frac{1}{\lambda^2} \lt( 1 - \langle \pl_r,\nu \rangle^2 \rt) + O(e^{-\alpha r}).
\end{align*}
The second line is
\begin{align*}
- \frac{2}{n-1} \lt( 1 -  \frac{n-1}{H} \frac{\lambda'}{\lambda} \langle \pl_r,\nu \rangle^{-1}  \rt)^2 -\frac{2}{n-1} \lt( \langle \pl_r,\nu \rangle^{-2}-1 \rt) - \frac{2}{H^2} \langle \pl_r,\nu \rangle^{-3} |\mathring A|^2 
\end{align*}
and the third line is
\[ ( \langle \pl_r,\nu \rangle^{-2} -1) \cdot O(e^{-2r}) + O(e^{-\alpha r}). \]
Let $\beta = \langle \pl_r,\nu \rangle^{-2} -1$ and we get
\[ \pl_t \beta \le \frac{1}{H^2} \Delta_\Sigma \beta + G \cdot \na \beta + \lt( -\frac{2}{n-1} + O(e^{-2r}) \rt)\beta + O(e^{-\alpha r}). \] 
Recall $\alpha > 2$, and the maximum principle implies
\begin{align*}
\max_{\Sigma_t} \beta \le O( e^{\frac{-2}{n-1}t}),
\end{align*}
where we used $e^r \sim e^{\frac{t}{n-1}}$. Finally, since $\langle \pl_r,\nu\rangle^{-2}-1 \ge -O(e^{-\alpha r})$, we obtain
\[ \langle \pl_r,\nu \rangle^{-2} = 1 + O(e^{\frac{-2}{n-1}t}). \]
\end{proof}

\begin{prop}\cite[Proposition 3.2]{BHW12}\label{upperH}
We have $H \le n-1 + O(e^{-\frac{2}{n-1}t})$
\end{prop}

\begin{prop}
We have $H = n-1 + O(t e^{-\frac{2}{n-1}t})$.
\end{prop}
\begin{proof}
In view of the previous proposition, it suffices to show that $H \le n-1 + O(te^{-\frac{2}{n-1}t})$. We will denote vectors, which change from line to line, whose forms are not relevant by $G$. Using $Ric(\nu,\nu) = -(n-1) + O(e^{-\alpha r})$, we obtain
\begin{align*}
\pl_t (H \langle \pl_r,\nu \rangle) &= \frac{1}{H^2} \lt( \Delta_\Sigma H \langle \pl_r,\nu \rangle + H \Delta_\Sigma \langle \pl_r,\nu \rangle \rt) - \frac{2|\na H|^2}{H^3}\langle \pl_r,\nu \rangle + \frac{n-1}{H}\langle \pl_r,\nu \rangle \\
&\quad + \frac{2}{H} \frac{\lambda'}{\lambda}\na r \cdot \na \langle \pl_r,\nu \rangle + \frac{n-1}{H} \lt( \frac{\lambda'	}{\lambda}\rt)^2 \langle \pl_r,\nu \rangle - 2 \frac{\lambda'}{\lambda} \langle \pl_r,\nu \rangle^2 + O(e^{-\alpha r}). 
\end{align*}
Proposition \ref{asymp_C1} now implies $|\na r|^2 = 1 - (\frac{\pl r}{\pl\nu})^2 + O(e^{-\alpha r})= O(e^{-\frac{2}{n-1}t})$ and hence   
$|\na \langle \pl_r,\nu \rangle|^2 = O( e^{-\frac{2}{n-1}t})$ by \eqref{nabla_Drnu}. From
\[ \na (H \langle \pl_r,\nu \rangle) = \langle \pl_r,\nu \rangle \na H + H \na \langle \pl_r,\nu \rangle,\]
we see that $|\na H|^2 = G \cdot \na (H\langle \pl_r,\nu \rangle) + O( e^{-\frac{2}{n-1}t})$. Putting these together with $\frac{\lambda'}{\lambda} = 1 + O(e^{-2r})$, we obtain
\begin{align*}
\pl_t (H \langle \pl_r,\nu \rangle) = \frac{1}{H^2}\Delta_\Sigma (H \langle \pl_r,\nu \rangle) + G \cdot \na (H\langle \pl_r,\nu \rangle) + \frac{2(n-1)}{H} -2 +O(e^{-\frac{2}{n-1}t}).
\end{align*}

Let $\chi = \frac{1}{H\langle \pl_r,\nu \rangle}$ and we get
\begin{align*}
\pl_t \chi = \frac{1}{H^2} \Delta_\Sigma \chi + G \cdot \na\chi + \lt( 2 + O(e^{-\frac{2}{n-1}t}) \rt)\chi^2 - \lt( 2(n-1) + O(e^{-\frac{2}{n-1}t}) \rt) \chi^3.
\end{align*}
Applying the maximum principle as in \cite[Proposition 4.1]{BHW12}, we obtain $H \le n-1 + O(te^{-\frac{2}{n-1}t})$. 
\end{proof}
\begin{prop}\cite[Proposition 4.2]{BHW12}\label{asymp_C2}
We have $|h_i^j - \delta_i^j| = O(t^2 e^{-\frac{2}{n-1}t})$.
\end{prop}

We turn to the parametric form of $\Sigma_t$. We first compute the geometric data assuming mild conditions on the radial function.
\begin{prop}
Let $\Sigma = (r(\theta),\theta)$ be a radial graph in $(1,\infty) \times \widehat\Sigma$ given by the function $r: \widehat\Sigma \rw \R$.  Assume
\begin{align}\label{Dr bound}
|D r|_{\widehat g} = O(e^r), \quad |D^2 r|_{\widehat g} = O(e^{2r})
\end{align}
where $D$ denotes the covariant derivative\footnote{In this and the next 2 propositions, we use $\na$ to denote the Levi-Civita connection of $g$ to facilitate the comparison with \cite{BHW12}} of $\widehat g$. Let $\varphi = \Phi(r(\theta))$ where $\Phi(r)$ is a positive function with $\Phi' = \frac{1}{\lambda}$. Let $\vphi_i = \pl_i \vphi$ and $\vphi_{ij} = D_iD_j \vphi$. 
Then the induced metric, second fundamental form, and the mean curvature of $\Sigma$ satisfy
\begin{align*}
\gamma_{ij} &= \lambda^2 (\widehat g_{ij} + \varphi_i\varphi_j + O(e^{-\alpha r})) \\
h_{ij} &= \frac{\lambda}{\varrho} \lt( \lambda' (\widehat g_{ij} + \varphi_i \varphi_j) - \varphi_{ij} \rt) + O(e^{(2-\alpha)r}) \\
H &=  \frac{(n-1)\lambda'}{\lambda \varrho} - \frac{\tilde\sigma^{ij}}{\lambda\varrho} \varphi_{ij} + O(e^{-\alpha r})
\end{align*}
where $\tilde\sigma^{ij} = \widehat g^{ij} - \frac{\varphi^i\varphi^j}{1+|D\varphi|^2_{\hat g}}$, $\varrho = \sqrt{1 + |D\varphi|^2_{\hat g}} + O(e^{-\alpha r})$ and we raise the indices by $\widehat g^{ij}$.
\end{prop}
\begin{proof}
 A basis of tangent vectors of $\Sigma$ is of the form $r_i\pl_r + \pl_{\theta^i}$. We compute
\begin{align*}
\gamma_{ij} &= r_i r_j (1 + q_{rr}) + r_i q_{rj} + r_j q_{ri} + \lambda^2 \widehat g_{ij} + q_{ij}\\
&= \lambda^2 (\widehat g_{ij} + \varphi_i \varphi_j + O(e^{-\alpha r})).
\end{align*}

The unit normal vector is given by 
\begin{align}\label{nu}
\nu = \frac{1}{\varrho} \lt( \pl_r + B^k \pl_k \rt)
\end{align}
where $B^k = - \frac{1}{\lambda^2} r^k + O(e^{(-\alpha-1)r})$ and 
\begin{align*}
\varrho^2 = 1 + q_{rr} + 2B^k q_{rk} + B^kB^l (\lambda^2 \widehat g_{kl} + q_{kl}) = 1 + |D \varphi|^2_{\hat g} + O(e^{-\alpha r}).
\end{align*}
The second fundamental form is $h_{ij} = - \langle \na_{r_i \pl_r + \pl_i} r_j \pl_r + \pl_j, \nu \rangle$. We have
\begin{align}
\langle \na_{r_i \pl_r + \pl_i} r_j \pl_r + \pl_j, \pl_r + B^k\pl_k \rangle &= \pl_i \pl_j r (1 + q_{rr} + B^k q_{rk}) \label{hij_1}\\
&\quad + r_j \langle \na_{r_i \pl_r + \pl_i} \pl_r, \pl_r + B^k\pl_k \rangle + \langle \na_{r_i \pl_r + \pl_i} \pl_j, \pl_r + B^k\pl_k \rangle.
\end{align}
Note that $q_{rr} + B^k q_{rk} = O(e^{-\alpha r})$. For the second line, we compute
\begin{align*}
\langle \na_{r_i \pl_r + \pl_i} \pl_r, \pl_r + B^k\pl_k \rangle &= r_i \lt[ \frac{1}{2} \pl_r q_{rr} + B^k (\pl_r q_{rk} - \frac{1}{2} \pl_k q_{rr}) \rt] + \pl_i q_{rr} \\
&\quad + \frac{1}{2} B^k \lt( \pl_i q_{rk} - \pl_k q_{ir} + \underline{ 2 \lambda\lambda' \widehat g_{ik} } + \pl_r q_{ik} \rt)\\
&= - \frac{\lambda'}{\lambda} r_i + O(e^{(1-\alpha ) r})
\end{align*}
and
\begin{align*}
\langle \na_{r_i \pl_r + \pl_i} \pl_j, \pl_r + B^k\pl_k \rangle &= r_i \lt[ \frac{1}{2} \pl_j q_{rr} + \frac{B^k}{2} ( \underline{ 2 \lambda\lambda' \widehat g_{jk} } + \pl_j q_{rk} - \pl_k q_{rj}) \rt] \\
&\quad + \frac{1}{2} \lt( \pl_i q_{jr} + \pl_j q_{ir} - \underline{2 \lambda\lambda' \widehat g_{ij}} - \pl_r q_{ij} \rt) \\
&\quad + B^k \lt[ \lambda^2 \cdot \frac{1}{2} (\pl_i \widehat g_{jk} + \pl_j \widehat g_{ik} - \pl_k \widehat g_{ij}) + \frac{1}{2} (\pl_i q_{jk} + \pl_j q_{ik} - \pl_k q_{ij}) \rt]\\
&= - \frac{\lambda'}{\lambda} r_i r_j - \lambda \lambda' \widehat g_{ij} - r^k \cdot \frac{1}{2} (\pl_i \widehat g_{jk} + \pl_j \widehat g_{ik} - \pl_k \widehat g_{ij}) + O(e^{(2-\alpha) r}).
\end{align*}
Putting these together, we get
\begin{align*}
h_{ij} = \frac{1}{\varrho} \lt( - r_{ij} + 2 \frac{\lambda'}{\lambda} r_i r_j + \lambda \lambda' \widehat g_{ij} + O(e^{(2-\alpha)r }) \rt) = \frac{\lambda}{\varrho} \lt( \lambda'(\widehat g_{ij} + \varphi_i\varphi_j) - \varphi_{ij} \rt) + O(e^{(2-\alpha) r}). 
\end{align*}

Finally, we compute the mean curvature
\begin{align*}
H = (\gamma^{-1})^{ij} h_{ij} = \frac{1}{\lambda^2} \lt( \widehat g^{ij} - \frac{\varphi^i \varphi^j}{1 + |D\varphi|^2_{\hat g}} + O(e^{-\alpha r}) \rt) h_{ij} = \frac{(n-1)\lambda'}{\lambda \varrho} - \frac{\tilde\sigma^{ij}}{\lambda\varrho} \varphi_{ij} + O(e^{-\alpha r}).
\end{align*} 
\end{proof}

The above proposition can be applied to star-shaped solutions to IMCF satisfying \eqref{assumption for asymptotic analysis}; namely, the assumption \eqref{Dr bound} follows from $\langle \pl_r,\nu \rangle \ge c$ and $|A| \le C$. Indeed, $\langle \pl_r,\nu \rangle \ge c$ implies that the angle of tangents of $\Sigma$ and $\pl_r$ is bounded above by a constant less than $1$. Let $e$ be a unit tangent vector on $\widehat\Sigma$. We have
\[ \frac{e(r)(1+q_{rr}) + q_{re})}{\sqrt{e(r)^2(1+q_{rr}) + 2 e(r)q_{re} + \lambda^2)}} \le \eta < 1. \]
It follows from quadratic formula that $|Dr|_{\widehat g} = O(e^r)$. Next, $|A| \le C$ implies $\gamma^{jk} r_{ik}$ in \eqref{hij_1} is bounded and it follows that $|D^2 r|_{\widehat g} = O(e^{2r})$.  

In fact, Proposition \ref{asymp_C1} and Proposition \ref{asymp_C2} imply the better bounds 
\begin{align}\label{improve1}
|D\vphi|_{\widehat g} = O( e^{-\frac{1}{n-1}t}) 
\end{align} 
and
\begin{align}\label{improve2}
|D^2 \vphi|_{\widehat g} = O(t^2 e^{-\frac{1}{n-1}t}).
\end{align}

We are ready to prove the main result of this section, cf Corollary 5.4 of \cite{BHW12} and Lemma 5.3 of \cite{GWWX13}. \begin{prop}\label{tildeQ}
Let $V$ be a smooth function on $(1,\infty) \times \widehat\Sigma$ satisfying 
\begin{align}
\Delta_{g} V &\ge n V \label{subharmonic}
\end{align}
and \begin{align}
V = \sqrt{\rho^2 + \widehat k} + o_1(\rho^{-1}). \label{V_expansion}
\end{align}
Then the quantity
\begin{align}
\tilde Q(t) = |\Sigma_t|^{-\frac{n-2}{n-1}} \lt( \int_{\Sigma_t} VH d\sigma -n(n-1) \int_{\Omega_t} V d\mbox{vol} \rt)
\end{align} has $\liminf_{t\rw\infty} \tilde Q(t) \ge (n-1) \widehat k |\widehat\Sigma|^{\frac{1}{n-1}}$.
\end{prop}

\begin{proof}
By \eqref{subharmonic} and divergence theorem, we have
\begin{align}
\int_{\Sigma_t} VH d\sigma - n(n-1) \int_\Omega V d\mbox{vol} &\ge \int_{\Sigma_t} VH d\sigma - (n-1) \int_\Sigma \frac{\pl V}{\pl\nu} d\sigma + O(1) \\
&\ge \int_{\Sigma_t} V (H-n+1) d\sigma + (n-1) \int_{\Sigma_t} \lt( V -\frac{\pl V}{\pl\nu} \rt) d\sigma + O(1) \label{temp}.
\end{align}

By \eqref{V_expansion}, we have $|\na_g V| = \rho + o(\rho^{-1})$ and hence 
\begin{align*}
V - \frac{\pl V}{\pl \nu} \ge V - |\na_g V| = \frac{\widehat k}{2} \lambda^{-1} + o(e^{-r}). 
\end{align*}
From the expression of $\gamma$ and the improved bound \eqref{improve1}, we have
\begin{align}\label{volume}
d\sigma = \lambda^{n-1} d\widehat\Sigma \lt( 1 + O(e^{-\frac{2}{n-1}t}) \rt)
\end{align}
where $d\widehat\Sigma$ denotes the volume form of $(\widehat\Sigma, \hat g)$. Hence the second integral in \eqref{temp} is
\begin{align}\label{integral2}
\int_{\widehat\Sigma} \frac{n-1}{2} \widehat k \lambda^{n-2} d\widehat\Sigma + o(e^{\frac{n-2}{n-1}t}). 
\end{align}

On the other hand, from the improved bounds \eqref{improve1} and \eqref{improve2} we get
\[ \tilde\sigma^{ij} \vphi_{ij} = \Delta_{\widehat g} \vphi + O(t^2 e^{-\frac{3}{n-1}t}). \]
We may assume $\alpha-2 \ll 1$ and this implies
\begin{align*}
H &= \frac{(n-1)\lambda'}{\lambda \varrho} - \frac{\tilde\sigma^{ij}}{\lambda\varrho} \varphi_{ij} + O(e^{-\alpha r})\\
&= \frac{(n-1)\lambda'}{\lambda \varrho} - \frac{1}{\lambda\varrho} \Delta_{\widehat g} \vphi + O(e^{-\frac{\alpha}{n-1}t}) \\
&= n-1 + \frac{n-1}{2\lambda^2}\widehat k - \frac{n-1}{2} |D\vphi|^2_{\widehat g} - \frac{1}{\lambda} \Delta_{\widehat g} \vphi + O(e^{-\frac{\alpha}{n-1}t})
\end{align*}
where we used $\lambda' = \lambda + \frac{\widehat k}{2}\lambda^{-1} + O(e^{-\frac{2}{n-1}t})$ and $\frac{1}{\varrho} = 1 - \frac{1}{2}|D\vphi|^2_{\widehat g} + O(e^{-\frac{\alpha}{n-1}t})$ in the last equality.
Since $V = \lambda (1+ O(e^{-\frac{2}{n-1}t}))$, we obtain
\begin{align*}
&\int_{\Sigma_t} V(H-n+1) d\sigma\\
&= \int_{\widehat\Sigma} \lt( \frac{n-1}{2} \widehat k\lambda^{n-2} - \frac{n-1}{2} \lambda^n |D\vphi|^2_{\widehat g} + (n-1) \lambda^{n-2} \langle D\lambda, D\vphi \rangle_{\widehat g} \rt) d\widehat\Sigma + O(e^{\frac{n-\alpha}{n-1}t}).
\end{align*}
By \eqref{improve1} and $D\lambda = \lambda \lambda' D\vphi$, we have $|D\lambda - \lambda^2 D\vphi|_{\widehat g} = O(t^{1/2} e^{-\frac{1}{n-1}t})$. This implies
\begin{align}\label{integral1}
\begin{split}
&\int_{\Sigma_t} V(H-n+1) d\sigma\\
&= \int_{\widehat\Sigma} \lt( \frac{n-1}{2} \widehat k\lambda^{n-2} + \frac{n-1}{2} \lambda^{n-4} |D\lambda|^2_{\widehat g}  \rt) d\widehat\Sigma + O(e^{\frac{n-\alpha}{n-1}t})
\end{split}
\end{align}

Adding \eqref{integral1} and \eqref{integral2}, we get
\begin{align*}
&\int_{\Sigma_t} VH d\sigma - n(n-1) \int_\Omega V d\mbox{vol} \\&\ge \int_{\widehat\Sigma} \lt( \frac{n-1}{2} \lambda^{n-4} |D\lambda|^2_{\widehat g} + (n-1) \widehat k \lambda^{n-2} \rt) d\widehat\Sigma + o(e^{\frac{n-2}{n-1}t}).  
\end{align*}
Moreover, by \eqref{volume}, $|\Sigma_t| = \int_{\widehat\Sigma} \lambda^{n-1} d\widehat\Sigma + O( e^{\frac{n-3}{n-1}t})$. When $\widehat k=1$, there is a finite Riemannian $\mathfrak{m}$-cover $p:S^{n-1} \rw \widehat\Sigma$. For every positive function $u$ on $\widehat\Sigma$, apply \cite[Proposition 5.1]{BHW12} to $u \circ p$ to get 
\begin{align*}
    \frac{\mathfrak{m}}{2} \int_{\widehat\Sigma} u^{n-4} |D u|^2_{\hat g} d\widehat\Sigma + \mathfrak{m} \int_{\widehat\Sigma} u^{n-2} d\widehat\Sigma
    \ge (\mathfrak{m} |\widehat\Sigma|)^{\frac{1}{n-1}} \lt( \mathfrak{m} \int_{\widehat\Sigma} u^{n-1} d\widehat\Sigma \rt)^{\frac{n-2}{n-1}}.
\end{align*} When $\widehat k=0,-1$, use H\"{o}lder inequality\footnote{The case $\widehat k=0$ is actually trivial} as in \cite[Lemma 5.3]{GWWX13}. 
We thus conclude that for $\widehat k = 1,0,-1$
\begin{align*}
\liminf_{t \rw \infty} |\Sigma_t|^{-\frac{n-2}{n-1}} \lt( \int_{\Sigma_t} VH d\sigma - n(n-1) \int_\Omega V d\mbox{vol} \rt) \ge (n-1) \widehat k |\widehat\Sigma|^{n-1} .
\end{align*}
\end{proof}

The next proposition will be used in the proof of Theorem \ref{area_minkowski}.
\begin{prop}\label{tildeQ2}
Let $V$ be a smooth function on $(1,\infty) \times \widehat\Sigma$ satisfying 
\eqref{V_expansion}. Then the quantity 
\begin{align}
    \tilde Q_2(t) = \lt( \frac{|\Sigma_t|}{\omega_{n-1}} \rt)^{-\frac{n-2}{n-1}} \lt( \int_{\Sigma_t} VH d\sigma - (n-1) \omega_{n-1} \lt( \frac{|\Sigma_t|}{\omega_{n-1}}\rt)^{\frac{n}{n-1}}\rt)
\end{align}
has $\liminf_{t\rw\infty} \tilde Q_2(t) \ge (n-1) \widehat k \omega_{n-1}^{\frac{1}{n-1}}$, where $\omega_{n-1} = |\widehat\Sigma|$.
\end{prop}
\begin{proof}
   We have $|\Sigma_t| = \int_{\Sigma_t} d\sigma = \int_{\widehat\Sigma} \lambda^{n-1} \lt( 1 + \frac{1}{2} |D\vphi|^2_{\hat g} + O(e^{-\alpha r}) \rt) d\widehat\Sigma$ and by H\"{o}lder inequality
   \[|\Sigma_t|^{\frac{n}{n-1}} \le |\widehat\Sigma|^{\frac{1}{n-1}}\int_{\widehat\Sigma} \lambda^n \lt( 1 + \frac{n}{2(n-1)} |D\vphi|^2_{\hat g} + O(e^{-\alpha r}) \rt) d\widehat\Sigma. \]
Here $d\hat\Sigma$ is the volume form of $(\hat\Sigma, \hat g)$. On the other hand, we write $\int_{\Sigma_t} VH d\sigma = \int_{\Sigma_t} V (H-n+1) d\sigma + (n-1) \int_{\Sigma_t} V d\sigma$ and by \eqref{V_expansion}
\[ \int_{\Sigma_t} V d\sigma = \int_{\widehat\Sigma} \lt( \lambda + \frac{\widehat k}{2} \lambda^{-1} + o(e^{-r}) \rt) \lambda^{n-1} \lt( 1 + \frac{1}{2} |D\vphi|^2_{\hat g} + O(e^{-\alpha r}) \rt) d\widehat\Sigma. \]
Putting these together with \eqref{integral1}, we get
\[ \int_{\Sigma_t} VH d\sigma - (n-1) \omega_{n-1} \lt( \frac{|\Sigma_t|}{\omega_{n-1}}\rt)^{\frac{n}{n-1}} \ge \int_{\widehat\Sigma} (n-1)\widehat k \lambda^{n-2} + \frac{n-2}{2} \lambda^{n-4} |D\lambda|^2_{\hat g} d\widehat\Sigma + o(e^{-\frac{n-2}{n-1} t}). \]
The assertion follows from (5.21) of \cite{GWWX13}. (When $\widehat k=1$, apply the trick in the proof of the previous proposition to Lemma 5.4 of \cite{GWWX13})
\end{proof}
\section{Proof of Theorem \ref{minkowski}}
We have all the ingredients for proving Theorem \ref{minkowski}. First of all, we can apply the results in Section \ref{section regularity} and Section \ref{asymptotic behavior} because the decay assumption on $g$ in Definition \ref{alh_static} is stronger. Next, item (3) in Proposition \ref{no area-minimizing in interior} and assumption \eqref{topological assumption for genus control in IMCF} imply that assumptions (1) and (2) of Theorem \ref{main_regularity} are satisfied. Therefore, IMCF eventually becomes smooth and Proposition \ref{tildeQ} applies. By Theorem \ref{monotonicity obstacle} and assumption \eqref{topological assumption for monotonicity}, $Q(t)$ is non-increasing along IMCF and we obtain
\[ Q(0) \ge \liminf_{t \rw \infty} Q(t) = \liminf_{t\rw\infty} \tilde Q(t) \ge 2 \omega_2 \widehat k. \] This proves \eqref{minkowski2}.

Next, we address the equality case, that is,  inequality \eqref{minkowski2} is saturated only by slices $\widehat{\Sigma} \times \{ \rho \}$ of the Kottler manifolds \eqref{ads_schwarzschild}. The authors characterized the equality case of the analogous Minkowski inequality for asymptotically flat static manifolds in \cite{HW24}, \cite{HW24a}, and this was quite nontrivial. In the asymptotically locally hyperbolic setting, the approach taken in \cite{HW24}, \cite{HW24a}-- namely, the characterization of ``quasi-spherical" static metrics-- can be completely circumvented using the equality statement for the substatic Heintze-Karcher inequality.  

\begin{Theorem}[\cite{BFP24}, Theorem 1.1] \label{HK_rigidity}
Let $(M^{n},g,V)$ be a sub-static Riemannian manifold with connected, non-degenerate horizon boundary $\partial M= \{ V=0 \}$, and let $\Omega \subset (M^{n},g)$ be a domain with $\partial \Omega =\partial M \bigsqcup \Sigma$ for some smooth \underline{connected} mean-convex hypersurface $\Sigma$. Then equality holds in the Heintze-Karcher inequality \eqref{heintze_karcher} if and only if $(\Omega^{n},g)$ is isometric to

\begin{equation} \label{warped_product}
\left( (\rho_{m},\rho_{0}) \times \partial M, \frac{1}{V(\rho)^{2}} d\rho^{2} + \rho^{2} g_{\partial M} \right).
\end{equation}
In particular, $\Sigma$ is a level set of $V$ and $\pl M$ is connected.
\end{Theorem}

\begin{rmk} \label{connected}
 The above theorem is in fact a special case of \cite[Theorem 1.1]{BFP24} as we assume $\Sigma$ is connected. The key step in their proof is to show that $g$ is conformal to a warped product, see (3.5) of \cite{BFP24}. This forces $\pl M$ to be connected since $\Sigma$ is.  
    \end{rmk}
From this, it is straightforward to show the following.
\begin{Lemma}
Suppose that, in addition to the hypotheses of the previous theorem, $(M^{n},g,V)$ solves the $n$-dimensional static vacuum Einstein equations

\begin{eqnarray}\label{n_static_eqs}
    D^{2} V &=& V \text{Ric} + (nV) g,  \\
    \Delta_g V &=& nV. \nonumber
\end{eqnarray}
Then $(\Omega^{n},g)$ is isometric to a region $\widehat{\Sigma}^{n-1} \times (\rho_{m},\rho_{0})$ of $n$-dimensional Kottler space

\begin{equation*}
    M^{n}= \widehat{\Sigma}^{n-1} \times (\rho_{m},\infty), \hspace{0.5cm} g_{m,\widehat{k}}= \left(\widehat{k} + \rho^{2} - \frac{2m}{\rho^{n-2}} \right)^{-1} d\rho^{2} + \rho^{2} \widehat{g}, \hspace{0.5cm} V_{m,\widehat{k}}(\rho) = \sqrt{\widehat{k} + \rho^{2} - \frac{2m}{\rho^{n-2}}}.
\end{equation*}
Here, $\widehat{g}$ is an Einstein metric on a closed manifold $\widehat{\Sigma}$ with Einstein constant $\widehat{k}$, that is

\begin{equation*}
    \widehat{\text{Ric}} = (n-2) \widehat{k} \widehat{g}.
\end{equation*}
\end{Lemma}
\begin{proof}
Calling $\partial M =\widehat{\Sigma}$ and $\widehat{g}=g_{\partial M}$, we first show that $\widehat{g}$ is Einstein. We let $*$ denote tensor contraction, e.g. $(T*X)(Y)=T(X,Y)$ for a tensor $T$ and vectors $X$ and $Y$, and $T|_{\widehat{\Sigma} \times \{ \rho \}}$ to denote the restriction of a tensor to the tangent space of $\widehat{\Sigma} \times \{ \rho \}$.

Using the variation for the induced metric on $\widehat{\Sigma} \times \{ \rho \}$, the second fundamental form $A$ of a slice $\widehat{\Sigma} \times \{ \rho \}$ of \eqref{warped_product} is given by

\begin{equation*}
    A= \ \frac{V}{2} \partial_{\rho} (\rho^{2} \widehat{g}) = V \rho \widehat{g}.
\end{equation*}
From this, we compute

\begin{eqnarray} \label{H_slice}
   \sum_{i=1}^{n-1} (A*E_{i}) \otimes (A * E_{i}) - HA &=& (V^{2}\rho^{2}) \sum_{i=1}^{n-1} (\widehat{g} * E_{i}) \otimes (\widehat{g} * E_{i}) - (n-1) \frac{V}{\rho} (V\rho)\widehat{g} \\
   &=& -(n-2) V^{2} \widehat{g}, \nonumber
\end{eqnarray}
where $E_{i}$ is an orthonormal basis of $T_{x} (\Sigma \times \{ \rho \})$. Moreover, by the static equations \eqref{n_static_eqs}, the (extrinsic) Ricci curvature in directions tangent to $\widehat{\Sigma} \times \{ \rho \}$ is given by

\begin{eqnarray} \label{ric_tang}
     \text{Ric}|_{\widehat{\Sigma} \times \{ \rho \}} &=& \frac{1}{V} D^{2} V|_{\widehat{\Sigma} \times \{ \rho \}} - n g|_{\widehat{\Sigma} \times \{ \rho \}} = \frac{1}{V} \frac{\partial V}{\partial \nu} A - n g|_{\widehat{\Sigma} \times \{ \rho \}} \\
     &=& \left( \rho V V' - n \rho^{2} \right) \widehat{g}. \nonumber
\end{eqnarray}
Finally, the variation for the second fundamental form, c.f. Theorem 3.2(iv) in \cite{HP99}, implies that the Riemann tensor $\text{Riem} * \nu * \nu$ contracted in the first and third arguments is given by
\begin{eqnarray} \label{riem_nn}
    \text{Riem}*\nu*\nu&=& \sum_{i=1}^{n-1} (A*E_{i}) \otimes (A*E_{i}) - V \frac{\partial}{\partial \rho} (A) =  V^{2} \widehat{g} - V(V'\rho + V) \widehat{g}  \\
    &=& - \rho VV' \widehat{g}. \nonumber
\end{eqnarray}
Inputting \eqref{H_slice}, \eqref{ric_tang}, and \eqref{riem_nn} into the Gauss equation

\begin{equation*}
    \text{Ric}|_{\widehat{\Sigma} \times \{ \rho \}} = \widehat{\text{Ric}} + \text{Riem}*\nu*\nu + \sum_{i=1}^{n-1} (A* E_{i}) \otimes (A*E_{i})  - HA,
\end{equation*}
yields
\begin{equation*}
\widehat{\text{Ric}}= \left(2\rho V'V + (n-2) V^{2} - n\rho^{2} \right) \widehat{g}.
\end{equation*}
Hence $\widehat{g}$ is an Einstein metric with Einstein constant
\begin{equation*}
    \widehat{k} = \frac{2\rho}{n-2} V' V + V^{2} - \frac{n}{n-2} \rho^{2}.
\end{equation*}
Using the initial condition $V(\rho_{m})=0$, we get that $V$ solves the initial value problem 


\begin{eqnarray*}
V'(\rho) &=& \frac{(n-2)}{2\rho} \left( (\widehat{k} + \frac{n}{n-2}\rho^{2}) V^{-1} - V \right), \\
V(\rho_{m}) &=& 0.
\end{eqnarray*}
The solution is

\begin{equation*}
    V(\rho)= \sqrt{ \widehat{k} + \rho^{2} - \frac{2m}{\rho^{n-2}}}, \hspace{1cm} \text{where} \hspace{1cm} m= \frac{1}{2} (\rho_{m}^{n} + \widehat{k} \rho_{m}^{n-2}).
\end{equation*}

\end{proof}
This characterization would also yield rigidity for a higher-dimensional Minkowski inequality, but for now we focus on the case $n=3$.
\begin{Theorem}
$\Sigma^{2} \subset (M^{3},g)$ achieves equality in \eqref{minkowski} if and only if $\Sigma = \widehat{\Sigma} \times \{ \rho \}$ is a slice in Kottler space \eqref{ads_schwarzschild} with mass $m > m_{\text{crit}}$.
\end{Theorem}

\begin{proof}
From the arguments in Section 4, it is clear that if $J > K$ then $Q(t)$ is necessarily non-constant under IMCF with obstacles, and so we must have $J=K$. From Section 5, $\Sigma_{t}$ is smooth, strictly mean-convex, and homeomorphic to $\widehat{\Sigma}$ for $t$ sufficiently large. Since $\Sigma_{t}$ saturates the Heintze-Karcher inequality, we must have that $K=1$ and that $(\Omega_{t},g)$ is isometric to a region of $\widehat{\Sigma} \times (\rho_{m},\rho_{0})$ of Kottler space with base $(\widehat{\Sigma}^{2},\widehat{g})$. In fact, the slices $\widehat{\Sigma}^{2} \times \{ \rho \} \subset (\Omega_{t},g)$ of Kottler are related to the IMCF of $\partial M$ by the change of variable $\rho= \rho_{m} e^{\frac{t}{2}}$ due to uniqueness of IMCF with given initial condition, and so taking the global solution with initial condition $\partial M$ yields an isometry to Kottler. 

\end{proof}
\begin{figure}\label{Topology of Horizons}
    \centering
 \textbf{Topological Censorship and Horizon Topologies}
    \vspace{1cm}


\begin{tikzpicture}
\begin{scope}[scale=0.7]
\begin{scope}[xscale=0.5, yscale=0.7, xshift=40]


\draw[blue, line width=2.5pt, domain=3:5, name path=C1] plot(\x, {(\x-5)^(2) -4 } );
\draw[blue, line width=2.5pt, domain=5:7, name path=c1] plot(\x, {(\x-5)^(2) -4 } );
\draw[blue, line width=2.5pt, domain=-4.4:-3, name path=A1] plot(\x, {-5*(-3 - \x)^(0.5)});
\draw[line width=2.5pt, domain=-6:-4.4] plot(\x, {-5*(-3 - \x)^(0.5)});
\draw[blue, line width=2.5pt, domain=13:14.4, name path=F1] plot(\x, {-5*(\x - 13)^(0.5)});
\draw[line width=2.5pt, domain=14.4:16] plot(\x, {-5*(\x - 13)^(0.5)});
\draw[blue] (14.6,-3) node{\Large{$\Omega$}};
\draw[blue] (14.4,-6) node[anchor=west]{\Large{$\Sigma$}};
\draw[blue, line width=2.5pt, domain=7:13, name path=D1]  plot(\x, {(1-(0.111*(\x-10)^(2)))^(0.5) } );
\draw[blue, line width=2.5pt, domain=0:3, name path=b1]  plot(\x, {(1-(0.111*\x^(2)))^(0.5) } );
\draw[blue, line width=2.5pt, domain=-3:3, name path=B1]  plot(\x, {(1-(0.111*(\x)^(2)))^(0.5) } );
\draw[blue, line width=2.5pt] (10,0) ellipse (3 and 1);
\draw[blue, line width=2.5pt] (0,0) ellipse (3 and 1);
\draw (0,1) node[anchor=south]{\Large{$g_{i}=g_{\infty}$}} ;
\draw (10.5,1) node[anchor=south]{\Large{$g_{j}=0, \hspace{0.1cm} j \neq i$}};
\draw[line width=2.5pt] (5,-11) node[anchor=north]{\Large{$\partial_{\infty} M= \widehat{\Sigma}$}};

\draw[line width=2.5pt] (16,-8.66) arc(0:180:11cm and 2cm); 
\draw[line width=2.5pt] (-6,-8.66) arc(180:360:11cm and 2cm); 
\draw[blue, line width=2.5pt] (14.4,-6) arc(0:180:9.4cm and 1.7cm); 
\draw[blue, line width=2.5pt, name path=A2] (-4.4,-6) arc(180:211.5:9.4cm and 1.7cm);
\draw[blue, line width=2.5pt, name path=B2] (-3,-6.87) arc(210:257:9.4cm and 1.7cm);
\draw[blue, line width=2.5pt, name path=F2] (14.4,-6) arc(360:330:9.4cm and 1.7cm);
\draw[blue, line width=2.5pt, name path=C2] (5, -7.7) arc(270:257.5:9.4cm and 1.7cm);
\draw[blue, line width=2.5pt, name path=c2] (5, -7.7) arc(270:282.5:9.4cm and 1.7cm);
\draw[blue, line width=2.5pt, name path=D2] (13.15,-6.87) arc(330:282.5:9.4cm and 1.7cm);
\tikzfillbetween[of=A1 and A2]{blue, opacity=0.1};
\tikzfillbetween[of=F1 and F2]{blue, opacity=0.1};
\tikzfillbetween[of=C1 and C2]{blue, opacity=0.1};
\tikzfillbetween[of=B1 and B2]{blue, opacity=0.1};
\tikzfillbetween[of=c1 and c2]{blue, opacity=0.1};
\tikzfillbetween[of=D1 and D2]{blue, opacity=0.1};
\end{scope}
\begin{scope}[xscale=0.3, yscale=0.5, yshift=-620, xshift=-30]


\draw[blue, line width=2.5pt, domain=3:5, name path=C1] plot(\x, {(\x-5)^(2) -4 } );
\draw[blue, line width=2.5pt, domain=5:7, name path=c1] plot(\x, {(\x-5)^(2) -4 } );
\draw[blue, line width=2.5pt, domain=-4.4:-3, name path=A1] plot(\x, {-5*(-3 - \x)^(0.5)});
\draw[line width=2.5pt, domain=-11:-4.4] plot(\x, {-5*(-3 - \x)^(0.5)});
\draw[blue, line width=2.5pt, domain=13:14.4, name path=F1] plot(\x, {-5*(\x - 13)^(0.5)});
\draw[line width=2.5pt, domain=14.4:16] plot(\x, {-5*(\x - 13)^(0.5)});
\draw[line width=2.5pt, domain=16:23.05] plot(\x, {0.4*(\x-18)^(2) -10.2});
\draw[line width=2.5pt] (29.05,0.001) arc(360:0:3cm and 1cm) ;
\draw[line width=2.5pt, domain=29.05:33]  plot(\x, {-7*(\x - 29.05)^(0.5) +0.001});
\draw (5,1) node[anchor=south]{\Large{$g_{i}=0, \hspace{0.2cm} i \leq K$}};
\draw (26.05,1) node[anchor=south]{\Large{$g_{J}=g_{\infty}$}};
\draw[line width=2.5pt] (11.1, -13.911) ellipse (21.9 and 3);
\draw[line width=2.5pt] (11,-17.5) node[anchor=north]{\Large{$\partial_{\infty} M=\widehat{\Sigma}$}};
\draw[blue] (14.4,-6) node[anchor=west] {\Large{$\Sigma$}};
\draw[blue] (14.8,-3) node{\Large{$\Omega$}};

\draw[blue, line width=2.5pt, domain=7:13, name path=D1]  plot(\x, {(1-(0.111*(\x-10)^(2)))^(0.5) } );
\draw[blue, line width=2.5pt, domain=0:3, name path=b1]  plot(\x, {(1-(0.111*\x^(2)))^(0.5) } );
\draw[blue, line width=2.5pt, domain=-3:3, name path=B1]  plot(\x, {(1-(0.111*(\x)^(2)))^(0.5) } );
\draw[blue, line width=2.5pt] (10,0) ellipse (3 and 1);
\draw[blue, line width=2.5pt] (0,0) ellipse (3 and 1);

\draw[blue, line width=2.5pt] (14.4,-6) arc(0:180:9.4cm and 1.7cm); 
\draw[blue, line width=2.5pt, name path=A2] (-4.4,-6) arc(180:211.5:9.4cm and 1.7cm);
\draw[blue, line width=2.5pt, name path=B2] (-3,-6.87) arc(210:257:9.4cm and 1.7cm);
\draw[blue, line width=2.5pt, name path=F2] (14.4,-6) arc(360:330:9.4cm and 1.7cm);
\draw[blue, line width=2.5pt, name path=C2] (5, -7.7) arc(270:257.5:9.4cm and 1.7cm);
\draw[blue, line width=2.5pt, name path=c2] (5, -7.7) arc(270:282.5:9.4cm and 1.7cm);
\draw[blue, line width=2.5pt, name path=D2] (13.15,-6.87) arc(330:282.5:9.4cm and 1.7cm);
\tikzfillbetween[of=A1 and A2]{blue, opacity=0.1};
\tikzfillbetween[of=F1 and F2]{blue, opacity=0.1};
\tikzfillbetween[of=C1 and C2]{blue, opacity=0.1};
\tikzfillbetween[of=B1 and B2]{blue, opacity=0.1};
\tikzfillbetween[of=c1 and c2]{blue, opacity=0.1};
\tikzfillbetween[of=D1 and D2]{blue, opacity=0.1};
\end{scope}
\end{scope}
\end{tikzpicture}

 \caption{From \cite{GSWW99}, the genera $g_{j}$ of outermost minimal surfaces $\partial_{j} M$ bounding an ALH manifold  $(M^{3},g)$ are related to the genus of $\partial_{\infty} M=\widehat{\Sigma}$ by
 \begin{equation} \label{genus inequality}
     \sum_{j=1}^{J} g_{j} \leq g_{\infty}.
 \end{equation}
Assumption (T1) in Theorem \ref{minkowski} requires $g_{i} \geq g_{\infty}$ for some horizon $\partial_{i} M$, so \eqref{genus inequality} then stipulates that $g_{i}=g_{\infty}$ and $g_{j}= 0$ for $j \neq i$. Furthermore, assumption (T2) that $g_{j} > 0$ for $j >K$ means the only horizon which can lie outside $\Omega$ is $\partial_{i} M$ itself, in which case the horizons that $\Sigma$ encloses are spheres. Although we do not assume these facts to prove Theorem \ref{minkowski}, they give motivation to remove assumption (T1) or (T2) from the theorem to obtain a further-reaching result in future work.
 }
    \label{topology}
\end{figure}
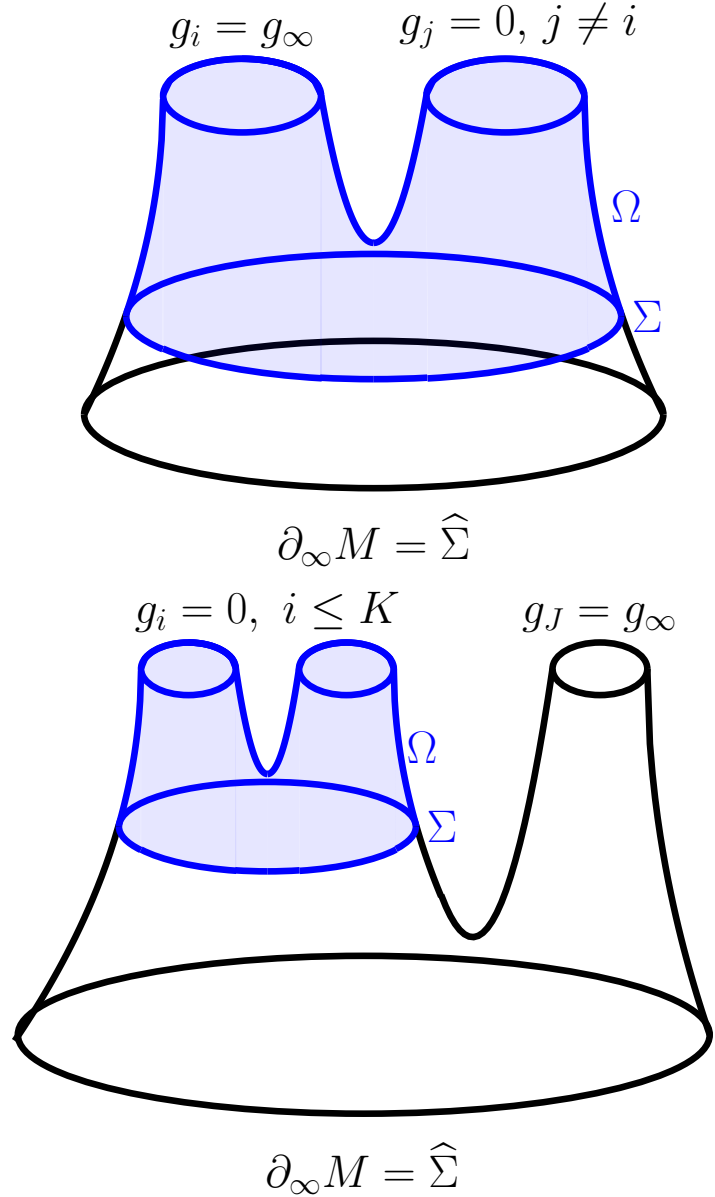
\begin{Remark} \label{high_dimension}
One possible future research direction is an extension of the Minkowski inequality to higher-dimensional ALH static spaces. The relevant quantity to consider for ALH static spaces in dimension $n$ is

\begin{eqnarray*}
Q(t)= \left( |\Sigma_{t}| \right)^{\frac{2-n}{n-1}} & & \left(  \int_{\Sigma_{t}} VH d\sigma - n(n-1) \int_{\Omega_{t}} V d\mbox{vol} \right. \\
& & \left. + 2 \frac{n-1}{n-2} \sum_{j=1}^{K} \frac{\int_{\partial_{j} M} R_{\partial_{j} M} d\sigma}{n(n-1)|\partial_{j} M| +  \int_{\partial_{j} M} R_{\partial_{j}M} d\sigma} \kappa_{j} |M_{j}| \right).
\end{eqnarray*}
Using the optimal Heintze-Karcher coefficients \eqref{c_j}, one may show that $Q(t)$ is monotone along weak IMCF for an initial domain $\Omega_{0}$ of the form \eqref{boundary_assumption}, provided $n < 8$ and the horizons not enclosed by $\Sigma_{0}$ have total scalar curvature $\int_{\partial_{j}M} R_{\partial_{j} M} d\sigma \leq 0$. The major obstacle is the asymptotic analysis. Specifically, the higher-dimensional Minkowski inequality would require an extension of Theorem \ref{main_regularity} on the regularity of weak IMCF in ALH $3$-manifolds. Many of the arguments of Section 5 are specific to $3$ dimensions, meaning this extension is likely difficult. It turns out that regularity of weak IMCF does hold higher-dimensional hyperbolic space, c.f. \cite{H24a}, but unfortunately the argument of \cite{H24a} does not easily translate to the ALH setting.
\end{Remark}

\section{Applications}
\subsection{Proof of Theorem \ref{bh_uniqueness_spherical} and Theorem \ref{toroidal_uniqueness}}
The black hole uniqueness theorems for spherical and toroidal infinities both come as immediate consequences of the rigidity statement in Theorem \ref{minkowski}. We begin with Theorem \ref{bh_uniqueness_spherical}. For $\Sigma = \pl_1 M$ (it is understood $\Omega = \emptyset$), \eqref{minkowski2} reduces to
\begin{align*}
    \frac{4\pi \chi(\pl_1 M) \kappa_1}{3|\pl_1 M| + 2\pi \chi(\pl_1 M)} \ge \lt( \frac{|\pl_1 M|}{w_2}\rt)^{-\frac{1}{2}}
\end{align*}
as $\hat k=1$. Recalling the denominator is positive (by \eqref{c_j} and \eqref{c_j2}), we see that $\chi(\pl_1 M) > 0$ and hence $\pl_1 M$ is homeomorphic to sphere. After some algebra ($w_2=4\pi$) we get
\begin{equation} \label{spherical_kappa}
    \kappa_{1} \geq \frac{1}{2} \left( 3 \left( \frac{|\partial_{1}M|}{4\pi} \right)^{\frac{1}{2}} + \left( \frac{|\partial_{1}M|}{4\pi} \right)^{-\frac{1}{2}} \right).
\end{equation}
The right-hand side of \eqref{spherical_kappa} is greater than or equal to $\sqrt{3}$ regardless of the value of $|\partial_{1} M|$, which gives us item (2) in Theorem \ref{bh_uniqueness_spherical}. Similarly, if $\kappa_{1}=\sqrt{3}$ then we neccessarily have saturation, implying that $(M^{3},g,V)$ is isometric to the AdS-Schwarzschild manifold with surface gravity $\kappa_{1}=\sqrt{3}$. For $\kappa_{1} > \sqrt{3}$, the lower bound \eqref{spherical_kappa} on $\kappa_{1}$ is equivalent to the two-sided area bound \eqref{area_bounds} in Theorem \ref{bh_uniqueness_spherical}.

Next, we consider flat ($\hat k=0$) toroidal infinity in Theorem \ref{toroidal_uniqueness}. Now \eqref{minkowski2} reduces to $\chi(\pl_i M) \ge 0$ for any component of the horizon. From the assumption that each $\partial_{j} M$ has genus at least one, we see that the equality saturates. Hence the horizon is connected ($J=1$)\footnote{We remark that the conclusion $J=1$ also follows from the assumption that $\chi(\partial_{j} M) \leq 0$ and \eqref{genus inequality}} and $(M^3,g,V)$ is isometric to Kottler space with $\hat k=0$. This completes the proof of Theorem \ref{toroidal_uniqueness}.

Let us close this subsection by explaining the Hawking-Page construction \cite{HP83} of Poincar\'e-Einstein metrics from ALH static system. Recall that $S^1(\lambda)$ denotes the circle with length $2\pi\lambda$.
\begin{Theorem}\label{Hawking-Page}
Let $(M^3,g,V)$ be an ALH static system with conformal infinity $(\hat\Sigma, \hat g)$. Suppose all components $\pl_i M$ have the same surface gravity $\kappa$. Consider the metric 
\begin{align*}
\mathrm g = V^2 d\theta^2 + g, 
\end{align*}
where $\theta \in [0,2\pi\lambda]$ is the coordinate of $S^1(\lambda)$. If $\lambda = \kappa^{-1}$, then $\mathrm g$ is a smooth 4-dimensional Poincar\'e-Einstein metric, $Ric(\mathrm g) = -3 \mathrm g$, whose conformal infinity is $S^1(\kappa^{-1}) \times\hat\Sigma$ equipped with product metric $d\theta^2 + \hat g$.
\end{Theorem}
\begin{proof}
By the static equations, $Ric(\mathrm g) = -3 \mathrm g$. From \eqref{tilde g} we see that $\mathrm g$ admits conformal compactification. It remains to show that $\mathrm g$ is smooth. The only problem is near $\pl M$. Let $r$ be the $g$-geodesic distance from a component, say $\pl_1 M$, of $\pl M$. In a tubular neighborhood of $S^1 \times \pl_1 M$, $\mathrm g = V^2 d\theta^2 + dr^2 + g_r$. Comparing $dr^2 + V^2 d\theta^2$ with the flat metric $d\mathbf{r}^2 + (\mathbf{r}/\lambda)^2 d (\lambda\boldsymbol{\theta})^2$, we see that if \[ \lambda^{-1} = \frac{\pl V}{\pl r}(r=0) = \kappa, \]
then $\mathrm g$ is free from conical singularity and hence $C^1$. Note that the topology near $r=0$ is now $D^2 \times \pl_1 M$. As $V$ extends smoothly to $\pl M$, $\mathrm g$ is smooth. 
\end{proof}

\subsection{The Reverse Penrose Inequality}
Note that the decay required for Chru\'{s}ciel-Herzlich mass to be well-defined, $\alpha > \frac{n}{2}$ is slower, equal, or faster than \eqref{ALH_intro} in dimension $3,4,$ or $n \ge 5$ respectively.

We will need the formula for the mass due to Herzlich, see Proposition 2.10 of \cite{MTX17} for example. 
\begin{Proposition}\label{Herzlich mass formula} Under the setup of Definition \ref{wang_mass}, we have
\begin{align*}
    \lim_{\rho \rw\infty} \int_{\widehat\Sigma \times\{ \rho\} } (\mbox{Ric}_g + (n-1) g )(\bar D f,\bar\nu) d\sigma_{\bar g} = -\frac{1}{(n-1)(n-2)\omega_{n-1}} m.
\end{align*}
\end{Proposition}
The following generalizes \cite[Corollary 4.4]{WW17}.
\begin{lem}\label{baby reverse Penrose} Let $(M^n,g,V)$ be an ALH static system
with decay $q = O_2(\rho^{-\alpha})$ and $V = f + O_1(\rho^{1-\alpha})$ for some $\alpha > \frac{n}{2}$. Then
\begin{align}\label{reverse Penrose}
m \le \frac{1}{(n-2)\omega_{n-1}} \sum_{j=1}^J \lt( 1- (n-1) c_j \rt) \kappa_j |\pl_j M|    
\end{align}
with $c_j$ given in \eqref{c_j}.
\end{lem}
\begin{proof}
    For $\Sigma = \hat\Sigma \times \{\rho \}$, we write \begin{align}\label{temp1}
    \begin{split}
         (n-1) \frac{V}{H} - \frac{\pl V}{\pl\nu} &= \lt( \frac{n-1}{H} - 1 \rt) V + \lt( \frac{H}{n-1} -1 \rt) \frac{\pl V}{\pl \nu} \\
         &\quad + \frac{1}{n-1} \lt( -\Delta_g V + D^2 V(\nu,\nu) + (n-1) V + \Delta_\Sigma V  \rt).
         \end{split}
    \end{align}
    Since $\nu = \rho \frac{\pl}{\pl\rho} + O(\rho^{-\alpha})$ on $\widehat\Sigma \times \{ \rho \}$, we have
    \[ \frac{\pl V}{\pl\nu} = \rho + O(\rho^{1-\alpha}) \]
    and \[ V \nu = \bar\na f + O(\rho^{1-\alpha}). \]
    Moreover, the mean curvature of $\widehat\Sigma \times \{ \rho\}$ is given by $H = n-1 + \mathcal{H}$ where $\mathcal{H} = O(\rho^{-\alpha})$; the volume form on $\hat\Sigma \times \{ \rho\}$ with respect to $g$ and $\bar g$ are related by $d\sigma = (1 + O(\rho^{-\alpha}))d\sigma_{\bar g}$. Integrating \eqref{temp1} over $\hat\Sigma \times \{ \rho\}$ and using the static equations \eqref{n_static_eqs} and Proposition \ref{Herzlich mass formula}, we obtain
    \[ \int_{\widehat\Sigma \times \{ \rho\} } (n-1)\frac{V}{H} - \frac{\pl V}{\pl\nu} d\sigma = O(\rho^{1-2\alpha + n-1}) - (n-2) \omega_{n-1}m. \]
    Recalling that $\alpha > \frac{n}{2}$ and  applying Heintze-Karcher inequality to $\widehat\Sigma \times \{ \rho\}$, we get
    \[ \sum_{j=1}^J \lt( -1 + (n-1) c_j \rt) \kappa_j |\pl_j M| \le -(n-2)\omega_{n-1} m. \]\end{proof}

\begin{proof}[Proof of Theorem \ref{Reverse Penrose Inequality}]
    As  $\chi(\partial M) < 0$, we may divide inequality \eqref{surface_grav_ineq} by $\chi(\pl  M)$ to obtain the surface gravity upper bound
\begin{equation} \label{surface_grav_hyperbolic}
    \kappa \leq -\frac{1}{4\pi \chi(\partial M)} \left( 3 w_{2} \left( \frac{|\partial M|}{w_{2}} \right)^{\frac{1}{2}} + 2\pi \chi(\partial M) \left( \frac{|\partial M|}{w_{2}} \right)^{-\frac{1}{2}} \right).
\end{equation}
For a 3-dimensional ALH static system with connected boundary, \eqref{reverse Penrose} reads  \begin{align*}
    m \le \frac{|\pl M|}{\omega_2} \frac{|\pl M| + 2\pi\chi(\pl M)}{3|\pl M| + 2\pi\chi(\pl M)} \kappa.
\end{align*}
Note that the numerator $|\pl M| + 2\pi\chi(\pl M) \ge 0$ by the assumption $m \ge 0$ and we can thus insert \eqref{surface_grav_hyperbolic} to get \begin{equation*}
    m \leq \frac{1}{2}  \frac{\chi(\hat\Sigma)}{\chi(\partial M)} \left( \frac{|\partial M|}{w_{2}} \right)^{\frac{3}{2}} - \frac{1}{2} \left( \frac{|\partial M|}{w_{2}} \right)^{\frac{1}{2}}\\
    \le \frac{1}{2}  \left( \frac{|\partial M|}{w_{2}} \right)^{\frac{3}{2}} - \frac{1}{2} \left( \frac{|\partial M|}{w_{2}} \right)^{\frac{1}{2}},
\end{equation*}
where we have used the relation $ w_{2}=-2\pi \chi(\hat\Sigma)$ arising from Gauss-Bonnet in the first inequality and the assumption $\chi(\partial M) \leq \chi(\widehat{\Sigma}) < 0$ in both inequalities. Also note that if equality is achieved then \eqref{surface_grav_hyperbolic} must be saturated, which once again triggers the rigidity statement.
\end{proof}
\subsection{The Areal Minkowski and Graphical Penrose Inequalities}
To prove Corollary \ref{alh_graph}, we need an alternative Minkowski-type inequality to the main one \eqref{minkowski2}. In \cite{DG15}, De Lima-Girao discovered a quantity under IMCF in hyperbolic space distinct from the one in \cite{BHW12} which is either monotonically non-increasing or is bounded below by $Q(t)$ for all times. Ge-Wang-Wu-Xia \cite{GWWX13} later determined the corresponding quantity in Kottler space, and this quantity can be further generalized to ALH static systems.
\begin{Theorem}[Areal Minkowski Inequality] \label{area_minkowski}
Let $(M^{3},g,V)$ and $\Omega$ be as in Theorem \ref{minkowski}, and assume also that $J=K$, i.e. that $\Omega$ is bounded by all components of $\partial M$. Then the surface $\Sigma$ satisfies the total mean curvature lower bound
\begin{equation} \label{area_mink}
    \frac{1}{4 w_{2}} \int_{\Sigma} V H d\sigma \geq \frac{1}{2} \left( \widehat{k} \left(\frac{|\Sigma|}{w_{2}} \right)^{\frac{1}{2}} + \left( \frac{|\Sigma|}{w_{2}} \right)^{\frac{3}{2}} \right) - \frac{1}{w_{2}} \sum_{j=1}^{J} \frac{|\partial_{j} M| + 2\pi\chi(\partial_{j} M)}{3|\partial_{j}M| + 2\pi\chi(\partial_{j}M) } \kappa_{j} |\partial_{j}M|,
\end{equation}
with equality if and only if $(M^{3},g,V)$ is isometric to Kottler space \eqref{ads_schwarzschild} and $\Sigma=\widehat{\Sigma} \times \{ \rho \}$ is a slice.
\end{Theorem}

\begin{proof}
First of all, we have for the smooth IMCF $\Sigma_{t}$ of $\Sigma$ that
\begin{eqnarray*}
    \frac{d}{dt} \left( w_{2}^{\frac{1}{2}}|\Sigma_{t}|^{\frac{3}{2}} - 2 \sum_{j=1}^{J} c_{j} \kappa_{j} |\partial_{j}M| - 3 \int_{\Omega_{t}} V d\mbox{vol} \right) &=& \frac{3}{2} w_{2}^{\frac{1}{2}} |\Sigma_{t}|^{\frac{3}{2}} - 3 \int_{\Sigma_{t}} \frac{V}{H} d\sigma \\
    &\leq& \frac{3}{2} \left( w_{2}^{\frac{1}{2}} |\Sigma_{t}|^{\frac{3}{2}} - 2 \sum_{j=1}^{J}c_{j} \kappa_{j} |\partial_{j} M| -3 \int_{\Omega_{t}} V d\mbox{vol} \right), 
\end{eqnarray*}
where $c_{j}$ are the coefficients \eqref{c_j} of the Heintze-Karcher inequality. As a result, if 
\begin{equation} \label{area_volume_ineq}
     w_{2}^{\frac{1}{2}}|\Sigma_{t_{0}}|^{\frac{3}{2}} - 2 \sum_{j=1}^{J} c_{j} \kappa_{j} |\partial_{j}M|  \leq 3 \int_{\Omega_{t_{0}}} V d\mbox{vol}
\end{equation}
for some $t_{0} \geq 0$, then this inequality is preserved for times greater than $t_{0}$. Take $t_{0} \in [0,T)$ to be the infimum of times at which \eqref{area_volume_ineq} holds. Then for times $t \in [0,t_{0})$, we have
\begin{align}\label{P_ev}
&\frac{d}{dt} \left( \int_{\Sigma_{t}} VH d\sigma -  w_{2}^{-\frac{1}{2}} |\Sigma_{t}|^{\frac{3}{2}} + 4 \sum_{j=1}^{J} (1-2c_{j}) \kappa_{j} |\partial_{j} M|  \right) \\ 
&\leq \frac{1}{2} \int_{\Sigma_{t}} VHd\sigma + 2 \int_{\Sigma_{t}} \frac{\partial V}{\partial \nu} d\sigma \nonumber - \frac{3}{2} w_{2}^{-\frac{1}{2}} |\Sigma_{t}|^{\frac{3}{2}} \nonumber \\
&= \frac{1}{2} \int_{\Sigma_{t}} VH d\sigma + 6 \int_{\Omega_{t}} V d\mbox{vol} + 2\sum_{j=1}^{J} \kappa_{j}|\partial_{j}M| - \frac{3}{2} w_{2}^{-\frac{1}{2}}|\Sigma_{t}|^{\frac{3}{2}} \nonumber \\
&\leq \frac{1}{2} \left( \int_{\Sigma_{t}} VH d\sigma - w_{2}^{-\frac{1}{2}} |\Sigma_{t}|^{\frac{3}{2}} \right. \nonumber \left. \hspace{0.4cm} + 4 \sum_{j=1}^{J} (1-2c_{j}) \kappa_{j} |\partial_{j} M| \right), \nonumber    
\end{align} 
where we have used the differential inequality \eqref{integral_VH} for $\int_{\Sigma_{t}} VH d\sigma$ and \eqref{area_volume_ineq}. As a result, the quantity

\begin{equation}
    P(t)=|\Sigma_{t}|^{-\frac{1}{2}} \left( \int_{\Sigma_{t}} VHd\sigma - 2 w_{2}^{\frac{1}{2}} |\Sigma_{t}|^{\frac{3}{2}} + 4 \sum_{j=1}^{J} (1-2c_{j}) \kappa_{j} |\partial_{j} M| \right)
\end{equation}
is monotonically non-increasing under IMCF for times $t \in (0,t_{0})$ and is trivially bounded below by $Q(t)$ for times $t \in (t_{0},T)$. Furthermore, these properties directly carry over to the weak IMCF of $\Sigma_{t}$ by Lemma \ref{mc} and Proposition \ref{bulk}. To determine the asymptotic behavior of $P(t)$ under weak IMCF, we consider $2$ cases separately:\\
(I) If $t_{0}$ exists, i.e. if \eqref{area_volume_ineq} is achieved in finite time, then by Theorem \ref{minkowski},
\begin{equation} \label{P_0}
    P(0) \geq P(t_{0}) \geq Q(t_{0}) \geq 2 \widehat{k} w_{2}^{\frac{1}{2}}.
\end{equation}
(II) If \eqref{area_volume_ineq} is not achieved for any $t \in (0,\infty)$, then $P(t)$ is monotone for all times, and by Proposition \ref{tildeQ2},
\begin{equation*}
  P(0) \geq \liminf_{t \rightarrow \infty} P(t) \geq 2\widehat{k} w_{2}^{\frac{1}{2}}.
\end{equation*}
Substituting
\begin{equation} \label{cj_factor}
1- 2c_{j}= 1 - \frac{2|\partial_{j}M|}{3|\partial_{j}M| + 2\pi \chi(\partial_{j}M)} = \frac{|\partial_{j}M| + 2\pi \chi(\partial_{j}M)}{3|\partial_{j}M| + 2\pi \chi(\partial_{j}M)}   
\end{equation}
yields \eqref{area_mink}. Also note that if equality is achieved then the Heintze-Karcher inequality must be saturated 
\end{proof}
\begin{Remark}
The assumption that $K=J$ arises from the fact that the Minkowski inequality \eqref{minkowski2}, which stipulates that $\max_{j\in \{ K+1,\dots, J \}} \chi(\partial_{j} M) \leq 0$, is required in the proof, particularly in \eqref{P_0}. On the other hand, the factors \eqref{cj_factor} which appear in the evolution formula for $P(t)$ may be positive even if $\chi(\partial_{j} M) \leq 0$. Thus the contributions of $\partial_{K+1} M, \dots, \partial_{J} M$ to the variation formula \eqref{P_ev} are not negligible, and so we assume no additional horizons.
\end{Remark}

Inequality \eqref{area_mink} does not give new information about the surface gravity of $\partial_{1} M$. Instead, its utility is made apparent by the mass formulas from \cite{DG15}, \cite{GWWX13} for ALH graphs in Riemannian Kottler manifolds. One obtains the same mass formula if the Kottler background is replaced by a warped product with an ALH static system as its base. From here, we can readily apply \eqref{area_mink}.

\begin{Corollary}[Graphical Penrose-Type Inequality in Riemannian Warped Products]
Let $(N^{3},h,u)$ be an ALH static system satisfying hypothesis (T1) and with Chru\'{s}ciel-Herzlich mass $m_{0}$. Consider the Riemannian warped product manifold
\begin{equation} \label{general_warped_product}
    (N^{3} \times \mathbb{R}, u^{2}d\tau^{2} + h), 
\end{equation}
and let $M^{3} \subset  (N^{3} \times \mathbb{R}, u^{2}d\tau^{2} + h) $ be an ALH hypersurface satisfying the following:

\begin{enumerate}
    \item $\langle \frac{\partial}{\partial \tau}, \eta \rangle > 0$, where $\eta$ is the unit normal of $M^{3}$.
    \item $\partial M$ lies in a slice $\{ \tau=\tau_{0} \}$ of \eqref{general_warped_product}, and $\overline{M^{3}}$ intersects this slice orthogonally. 
    \item $\Sigma=\partial M \subset (\{ \tau=\tau_{0}\},h)$ is a connected, outer-minimizing surface enclosing all horizons $\partial_{1}N, \dots, \partial_{J} N$ of $N^{3}$. 
\end{enumerate}
Then assuming $R_{g} + 6$ is $L^{1}$ integrable on $M^{3} \subset  (N^{3} \times \mathbb{R}, u^{2}d\tau^{2} + h) $, the Chru\'{s}ciel-Herzlich mass $m$ of $(M^{3},g)$ has the following lower bound:

\begin{eqnarray} \label{mass_inequality}
    m &\geq& \frac{1}{4w_{2}} \int_{M^{3}} \langle \frac{\partial}{\partial \tau}, \eta \rangle (R_{g}+6) d\mbox{vol} + \frac{1}{2} \left(  \widehat{k} \left( \frac{|\Sigma|}{w_{2}} \right)^{\frac{1}{2}} + \left( \frac{|\Sigma|}{w_{2}} \right)^{\frac{3}{2}}  \right) \\
    & &+ \left(m_{0} - \frac{1}{w_{2}} \sum_{j=1}^{J} \frac{|\partial_{j}N| + 2\pi \chi(\partial_{j}N)}{3|\partial_{j}N| + 2\pi \chi(\partial_{j}N)}\kappa_{j} |\partial_{j} N| \right). \nonumber
\end{eqnarray}
As a result, if $R_{g} \geq -6$ on $M^{3}$ and $(N^{3},h,u)$ is Kottler space then the Riemannian Penrose inequality

\begin{equation*}
    m \geq \frac{1}{2} \left(  \widehat{k} \left( \frac{|\Sigma|}{w_{2}} \right)^{\frac{1}{2}} +  \left( \frac{|\Sigma|}{w_{2}} \right)^{\frac{3}{2}} \right)
\end{equation*}
holds for $(M^{3},g)$.
\end{Corollary}
\begin{Remark}
The static equations \eqref{static_equations} imply that the warped product \eqref{general_warped_product} is Einstein. Also, point (2) implies that $\partial M$ is a totally geodesic surface within $(M^{3},g)$.
\end{Remark}

\begin{Remark}
For the last term in \eqref{mass_inequality}, we have
\begin{equation*}
    m_{0} - \frac{1}{w_{2}} \sum_{j=1}^{K} \frac{|\partial_{j}N| + 2\pi \chi(\partial_{j}N)}{3|\partial_{j}N| + 2\pi \chi(\partial_{j}N)}\kappa_{j} |\partial_{j} N| = m_{0} - \frac{1}{w_{2}} \sum_{j=1}^{J} (1-2c_{j}) \kappa_{j} |\partial_{j}N| \leq 0.
\end{equation*}
as a consequence of Lemma \ref{baby reverse Penrose}. Therefore, \eqref{mass_inequality} does not necessarily imply the graphical Penrose inequality within a general Einstein warped product. In fact, we suspect that rigidity holds in \eqref{baby reverse Penrose}, which would mean that \eqref{mass_inequality} only produces the Penrose inequality for Kottler. 
\end{Remark}
\begin{proof}
\cite{GWWX13} equation (1.13), see also \cite{DG12}, provides the mass formula

\begin{equation*} 
    m = m_{0} + \frac{1}{4w_{2}} \int_{M^{3}} \langle \frac{\partial}{\partial \tau}, \eta \rangle (R_{g}+6) d\mbox{vol} + \frac{1}{4w_{2}} \int_{\Sigma} u H d\sigma,
\end{equation*}
for ALH hypersurfaces in Riemannian Kottler manifolds satisfying points (1)-(3). The proof immediately generalizes to a general Einstein background with an ALH static base, see pages 9-11 of \cite{GWWX13}. Inputting $R_{g} + 6 \geq 0$ and the areal Penrose inequality into the above formula, we obtain \eqref{mass_inequality}, and the last term vanishes if $(N^{3},h,u)$ is Kottler.
\end{proof}
\begin{proof}[Proof of Corollary \ref{alh_graph}]
If $(M^{3},g)$ embeds as a graph in Riemannian Kottler manifolds with the same infinity $\widehat{\Sigma}$ and $\partial M$ encloses the Kottler horizon, then we must have $\chi(\partial M) \leq \chi(\{ \rho=\rho_{m}\}) = \chi(\widehat{\Sigma})$. Therefore, \eqref{RPI} and \eqref{reverse_penrose} together imply uniqueness.
\end{proof}
\appendix
\section{Allard's Regularity Theorem}
We first recall Allard's regularity theorem from Simon \cite[Theorem 5.2]{S18}.
\begin{theorem}\label{Allard_Euclidean}
Let $V = \underline{v}(M,\theta)$ be a rectifiable $n$-varifold with generalized mean curvature $\underline{H}$ in $\R^{n+k}$. For every $p > n$ and every $\gamma \in (0,1)$ there is $\delta_0 = \delta_0 (n,k,p,\gamma) $ such that if
\begin{align}
&\theta \ge 1 \quad \mu-a.e., \quad 0 \in \mathrm{spt} V, \\
&\mu (B_\rho(0)) \le (1 + \delta)\omega_n \rho^{n}, \quad  \lt( \rho^{p-n} \int_{B_\rho(0)} |\underline{H}|^p d\mu \rt)^{1/p} \le \delta
\end{align}
holds, then there is an orthogonal transformation $Q$ of $\R^{n+k}$ and a vector-valued function $u = (u^1,\cdots,u^k) \in C^{1,1 - \frac{n}{p}}(B^n_{\gamma\rho}(0);\R^k)$ with $Du(0)=0$, \underline{$\mathrm{spt}V \cap B_{\gamma\rho}(0) = Q(\mathrm{graph} \, u) \cap B_{\gamma\rho}(0)$ }, and
\begin{align*}
\rho^{-1} \sup |u| &+ \sup |Du| \\
&+ \rho^{1 - \frac{n}{p}} \sup |x-y|^{-(1 - \frac{n}{p})} |Du(x)-Du(y)| \le C\delta^{\frac{1}{2n+2}}
\end{align*}
where $C = C(n,k,p,\gamma) > 0$.  
\end{theorem}

It is well-known and pointed out by Allard in page 2 of \cite{A72} that the result holds in Riemannian manifolds using Nash's isometric embedding theorem. 
We formulate a version suitable for our purpose and outline the argument for the reader's convenience.

\begin{theorem}\label{Allard_Riemannian}
Let $U$ be an open set with compact closure in an $(n+k)$-dimensional Riemannian manifold $N$ (this includes the case $U=N$ is a closed Riemannian manifold). Let $V = \underline{v}(M,\theta)$ be a rectifiable $n$-varifold with generalized mean curvature $\underline{H}$ in $U$. For every $p > n$ and every $\gamma \in (0,1)$ there are $\delta, \rho_0$ such that if
\begin{align}
&\theta \ge 1 \quad \mu-a.e., \quad x \in \mathrm{spt} V, \\
&\mu (B_\rho(x)) \le (1 + \delta)\omega_n \rho^{n}, \quad  \lt( \rho^{p-n} \int_{B_\rho(x)} |\underline{H}|^p d\mu \rt)^{1/p} \le \delta \label{Allard assumption 2}
\end{align}
holds for $\rho < \rho_0$, $B_\rho(x) \subset U$, then there exists a $C^{1,1-\frac{n}{p}}$ function $u$ defined on $B^n_{\gamma\rho}(0) \subset T_x V$ satisfying the conclusion of Theorem \ref{Allard_Euclidean}. Here it is understood that $\rho_0$ is small so that $B_\rho(x)$ is contained in some normal coordinate neighborhood centered at $x$ and in the conclusion the norm is taken with respect to an orthonormal frame at $x$.
\end{theorem}
\begin{proof}
By Nash's isometric embedding theorem, there exists an isometric embedding $\iota: N \rw \R^L$. Let $\vec{H}_N$ denote mean curvature vector of $\iota(N)$ in $\R^L$ and \[ K = \max_{\iota(\bar U)} | \vec{H}_N |. \] The generalized mean curvature of $V$ in $\R^L$ is decomposed orthogonally $\vec{H} = \underline{H} + \vec{H}_N$ and we get
\begin{align*}
\rho^{p-n} \int_{B_\rho(x)} |\vec{H}|^p d\sigma &\le  2^{p-1} \lt( \rho^{p-n} \int_{B_\rho(x)} |\underline{H}|^p d\mu + (1+\delta) \omega_n \rho^p K^p \rt) \\
&\le \delta_0(L,n,\alpha,\gamma)^p.
\end{align*}
if $\delta$ and $\rho_0$ are sufficiently small.\footnote{$\delta = \delta(n,k,\alpha,\gamma)$ since $L$ is bounded above by a polynomial of $n+k$ but the dependence of $K$, hence $\rho_0$, on the geometry of $N$ is not given in Nash's theorem.} Applying Allard's regularity to $V$ in $\R^L$ gives the desired conclusion. 
\end{proof}
\textbf{Conflict of Interest Statement} On behalf of all authors, the corresponding author states that there is no conflict of interest. \\

\textbf{Data Availability Statement} Data sharing is not applicable to this article as no datasets were generated or analyzed during the current study.

\printbibliography[title={References}]

\begin{center}
\textnormal{ \large Center for Geometry and Topology \\
University of Copenhagen \\
Copenhagen, DE 2100 \\
e-mail: brdh@math.ku.dk}\\
\end{center}
\vspace{1cm}
\begin{center}
\textnormal{ \large Department of Applied Mathematics \\
National Yang-Ming Chiao Tung University \\
Hsinchu, Taiwan 30010 \\
e-mail: yekaiwang@nycu.edu.tw}\\
\end{center}
\end{document}